\documentclass[11pt]{article}
\usepackage{amsmath,amsthm,amssymb,mathrsfs,cases}
\usepackage{color}
\usepackage{arydshln}
\usepackage{latexsym}
\usepackage{datetime}
\usepackage{hyperref}
\usepackage[all]{xy}

\theoremstyle{plain}
  \newtheorem{thm}{Theorem}[section]
  \newtheorem{lem}[thm]{Lemma}
  \newtheorem{prop}[thm]{Proposition}
  \newtheorem{cor}[thm]{Corollary} 
\theoremstyle{definition}
  \newtheorem{defn}[thm]{Definition}
	\newtheorem{rmk}[thm]{Remark}
  \theoremstyle{plain}

\numberwithin{equation}{section}
\newcommand\CB{\mathcal{P}}
\newcommand\CL{\mathcal{L}}

\newcommand\CI{\mathcal{I}}
\newcommand\A{\mathcal{A}}

\newcommand\CA{\Omega}

\newcommand\bo{\mathbf{1}}

\newcommand\BZ{\mathbb{Z}}
\newcommand\BR{\mathbb{R}}

\newcommand\HD[1]{H_{d}^{#1}}

\newcommand\PHPR[1]{PH_{d\dl}^{#1}}
\newcommand\PHPL[1]{PH_{d+\dl}^{#1}}
\newcommand\PHDP[1]{PH_{\dpp}^{#1}}
\newcommand\PHDM[1]{PH_{\dpm}^{#1}}
\newcommand\FO[2]{F^{#1}\CA^{#2}}
\newcommand\BFO[2]{\overline{{F^{#1}\CA^{#2}}}}
\newcommand\FHP[2]{F^{#1}H_+^{#2}}
\newcommand\FHM[2]{F^{#1}H_-^{#2}}

\newcommand\Massey[2]{\langle{#1},{#2}\rangle_{p}}
\newcommand\FHPM[2]{F^{#1}H_{\pm}^{#2}}

\newcommand\Ain{{A_\infty}}

\newcommand\dl{d^\Lambda}

\newcommand\dpp{{\partial_+}}
\newcommand\dpm{{\partial_-}}
\newcommand\dppm{{\partial_\pm}}
\newcommand\ddp{{d_+}}
\newcommand\ddm{{d_-}}

\newcommand\sstar{{*_s}}
\newcommand\rstar{{*_r}}

\newcommand\bA{\bar{A}}

\newcommand\pa{\partial}
\newcommand\om{\omega}
\newcommand\w{\wedge}
\newcommand\Bw{\bigwedge}
\newcommand\La{\Lambda}
\newcommand\si{\sigma}
\newcommand\be{\beta}

\newcommand\tmu{{\tilde \mu}}

\DeclareMathOperator{\im}{im}
\DeclareMathOperator{\cok}{coker}

\DeclareMathOperator{\spn}{span}

\newcommand\pil[1]{{\Pi^{#1}}}

\newcommand\rhos{{\rho^*}}
\newcommand\phis{{\phi^*}}
\newcommand\psis{{\psi^*}}

\newcommand\pac{{\delta_C}}
\newcommand\pad{{\delta_D}}
\newcommand\pae{{\delta_E}}
\newcommand\paf{{\delta_F}}

\newcommand\pads{{\delta_D^*}}
\newcommand\paes{{\delta_E^*}}

\def\dmi{\dpm B_{i-2p}}
\def\dmj{\dpm B_{j-2p}}
\def\dmk{\dpm B_{k-2p}}
\def\moi{(-1)^{i}}
\def\moj{(-1)^{j}}
\def\moij{(-1)^{i+j}}
\def\nn{\nonumber}
\def\pip{\pil{p}}
\def\pips{{{\pil{p}}^*}}
\def\Fp{{\mathcal{F}}_p}

\def\lpod{{L^{-(p+1)} d}}
\def\psdl{{\pil{p} \rstar d\, L^{-(p+1)}}}
\def\lpo{{L^{-(p+1)}}}
\def\opo{{\omega^{p+1}}}
\def\pprs{{\pil{p}\,\rstar\,}}

\def\aa{A_i}
\def\ab{A_j}
\def\ac{A_k}
\def\ad{A_l}
\def\bac{{\bar{A}_k}}
\def\pp{\times}
\def\mm{m_3}

\begin{document}

\title{\bf{Cohomology and Hodge Theory on \\
 Symplectic Manifolds: III}}

\author{Chung-Jun Tsai, Li-Sheng Tseng and Shing-Tung Yau \\
\\
}

\date{February 3, 2014}

\maketitle
\begin{abstract} 
We introduce filtered cohomologies of differential forms on symplectic manifolds.  They generalize and include the cohomologies discussed in Paper I and II as a subset.  The filtered cohomologies are finite-dimensional and can be associated with differential elliptic complexes.  Algebraically, we show that the filtered cohomologies give a two-sided resolution of Lefschetz maps, and thereby, they are directly related to the kernels and cokernels of the Lefschetz maps.  We also introduce a novel, non-associative product operation on differential forms for symplectic manifolds.  This product generates an $A_\infty$-algebra structure on forms that underlies the filtered cohomologies and gives them a ring structure.   As an application, we demonstrate how the ring structure of the filtered cohomologies can distinguish different symplectic four-manifolds in the context of a circle times a fibered three-manifold.

\

\end{abstract}

\tableofcontents

\section{Introduction}

On a symplectic manifold $(M^{2n}, \om)$ of dimension $2n$, there is a well-known
$\mathfrak{sl}(2)$ action on the space of differential forms, $\CA^*(M)$.  This action leads directly to what is called the Lefschetz decomposition of a differential $k$-form, $A_k\in\CA^k(M)$, 
\begin{equation}\label{ALdec}
A_k = B_k + \om\wedge B_{k-2} + \om^2\wedge B_{k-4} + \om^3\wedge B_{k-6} + \ldots ~,
\end{equation}
where the forms $B_s\in \CB^s(M)$, for $0\leq s \leq n$, denote the so-called {\it primitive} forms.  The primitive forms are the highest weight elements of the $\mathfrak{sl}(2)$ action and in \eqref{ALdec} are uniquely determined by the given $A_k$.  

In Paper I and II \cite{TY1, TY2}, several symplectic cohomologies of differential forms,  labeled by $\{H_{d+\dl}, H_{d\dl}, H_{\dpp}, H_{\dpm}\}$, were introduced and all were shown to commute with this $\mathfrak{sl}(2)$ action.  Hence, in essence, all information of these cohomologies are encoded in their primitive components, $\{PH_{d+\dl}, PH_{d\dl}, PH_{\dpp}, PH_{\dpm}\}\,$, which can be defined purely on the space of primitive forms, $\CB^*(M)$.   In short, the cohomologies introduced in Paper I and II are truly just primitive cohomologies.

This may seem to suggest if one wants to study cohomologies of forms on symplectic manifolds that one should focus on the primitive forms and their cohomologies.  However, this certainly can not be the case as we know from explicit examples in \cite{TY1, TY2} that primitive cohomologies in general contain different information than the de Rham cohomology, and of course, the de Rham cohomology is defined on $\CA^*(M)$ which are generally non-primitive.  So one may wonder, besides the de Rham cohomology, are there any other {\it non-primitive} cohomologies of differential forms on $(M, \om)$?  

Another curiosity comes from the relations between the known primitive cohomologies.  In Table 1, we list the main primitive cohomologies that were studied in \cite{TY1, TY2}.  As arranged, the cohomologies listed above the top horizontal line are all associated with a single symplectic elliptic complex \cite{TY2}.  It would seem rather unnatural if somehow the other primitive cohomologies, between the two vertical dashed lines, do not also arise from some elliptic complexes.  For instance, what makes $\PHPL{n}$ so different from $\PHPL{n-1}$?  Certainly from their definitions in \cite{TY1}, the only difference is just the degree of the space of primitive forms $\CB^*(M)$ which they are defined on and nothing more.  But if they are not so different, then what other elliptic complexes are there on symplectic manifolds?  Would these new elliptic complexes involve non-primitive forms?

\begin{table}[t]
\begin{center}
\renewcommand{\tabcolsep}{.1cm}
\begin{tabular}{r :  c :  l }
$\PHDP{0}, \PHDP{1}, \ldots, \PHDP{n-1},$&  $\PHPR{n}$,  $\PHPL{n}$, &  $\PHDM{n-1}, \ldots, \PHDM{1}, \PHDM{0}$
\\ 
\hline
&$\PHPR{n-1}$,  $\PHPL{n-1}$& 
\\
& ~\vdots~~~~~~~~~  \vdots~ &
\\
&$\PHPR{1}$,  $\PHPL{1}$& 
\\
&$\PHPR{0}$, $\PHPL{0}$& 
\end{tabular}
\end{center}
 \caption{The primitive cohomologies introduced in Paper I and  II \cite{TY1, TY2} for  symplectic manifolds of dimension $2n$.
\label{tabcoh1}} 
\end{table}

\medskip

\noindent{\it (1) Filtered forms and symplectic elliptic complexes}

These questions concerning the existence of new non-primitive cohomologies and other elliptic complexes on symplectic manifolds turn out to be closely related.  For at the level of the differential forms, one can think of the primitive forms as the result of a projection operator, $\Pi^0: \CA^k(M) \to \CB^k(M)$, that projects any form to its primitive component and thereby discarding all terms of order $\om$ and higher.  Generalizing this, we can introduce the projection operator, $\pil{p}$, for $0 \leq p \leq n$, that keeps terms up to the $\om^p$-th order of the Lefschetz decomposition in \eqref{ALdec}:
\begin{align*}
A_k & = B_k + \om \w B_{k-2} + \om^2 \w B_{k-4} + \om^3 \w B_{k-6} + \ldots  ~,\\
\pil{0}  A_k &= B_k ~,\\
\pil{1} A_k &= B_k + \om \w B_{k-2}~,\\
 &  ~\vdots  \\
\pil{p} A_k &= B_k + \om \w B_{k-2} + \om^2 \w B_{k-4}+ \ldots + \om^p \w B_{k-2p}~,\\
 & ~  \vdots ~
\end{align*}
We shall use the notation $\FO{p}{*}$ to denote the projected space of $\pil{p}\CA^*\subset \CA^*$ and call it the space of $p$-filtered forms.  We shall also call the index $p$ the filtration degree as it parametrizes a natural filtration: 
\begin{align*}
\CB^* = \FO{0}{*} \subset \FO{1}{*} \subset \FO{2}{*} \subset \ldots \subset \FO{n}{*} = \CA^*~.
\end{align*}
Notice that the zero-filtered forms are precisely the primitive forms, i.e. $\CB^* = \FO{0}{*}$, and the $n$-filtered forms are just $\CA^*$.  In this way, increasing the filtration degree from $p=0$ to $p=n$ allows us to interpolate from $\CB^*$ to $\CA^*\,$. 

The introduction of filtered forms turns out to be a fruitful enterprise.  For one, it allows us to generalize the symplectic elliptic complex for primitive forms to obtain an elliptic complex of $p$-filtered forms of any fixed filtration degree $p$.  Specifically, we shall show in Theorem \ref{thrmecomplex} that the following complex is elliptic:  
\begin{align} \label{filterecomplex} \xymatrix{
0\; \ar[r] &\; \FO{p}{0} \; \ar[r]^-{\ddp} & \; \FO{p}{1} \; \ar[r]^-{\ddp} & ~ \ldots ~ \ar[r]^-{\ddp} & \; \FO{p}{n+p-1} \; \ar[r]^-{\ddp} & \; \FO{p}{n+p} \ar[d]^{\dpp\dpm} \\
0\; &\; \FO{p}{0} \; \ar[l] & \; \FO{p}{1} \; \ar[l]_-{\ddm} & ~ \ldots ~ \ar[l]_-{\ddm} & \; \FO{p}{n+p-1} \; \ar[l]_-{\ddm} & \; \FO{p}{n+p} \ar[l]_{\ddm}
} \end{align}
The three differential operators -- two first-order differential operators $\{\ddp, \ddm\}$ and a second-order differential operator $\dpp\dpm$ -- appearing in this complex will be defined in Section \ref{sec_pre}.   What is important here is that associated with the above elliptic complex  are {\it filtered} cohomologies defined on the space of $p$-filtered forms, $\FO{p}{*}$.   We shall denote these cohomologies by $F^pH$.  Let us note that the elliptic complex in \eqref{filterecomplex} has two levels: a top level associated with the ``$+$" operator $\ddp$ and a bottom one associated with the ``$-$" operator $\ddm$.  Hence, it is natural to notationally split the cohomologies within each grouping of $F^pH$ into two as follows:
\begin{align*}
F^pH&=\left\{\FHP{p}{}, \FHM{p}{} \right\}\\
&=\left\{\left(\FHP{p}{0}, \ldots, \FHP{p}{n+p-1}, \FHP{p}{n+p}\right), \left(\FHM{p}{n+p}, \FHM{p}{n+p-1}, \ldots, \FHM{p}{0}\right)\right\}~.
\end{align*}
Of particular interest, we point out the isomorphisms  $\FHP{p}{n+p}\! \cong\! \PHPR{n-p}$ and $\FHM{p}{n+p}\! \cong \!\PHPL{n-p}$.  Thus, Table 1 can now be filled-in precisely by the filtered cohomologies as seen in Table 2.

\begin{table}[t]
\begin{center}
\renewcommand{\tabcolsep}{.1cm}
\begin{tabular}{c | c c c c |  c c c c }
~~~~$F^*H$~~~~ &    &  &\!\!\!\!\!\!\!\!\!\!\!\!\!\!$\FHP{*}{*}$&   & & $\FHM{*}{*}$\!\!\!\!\!\!\!\!\!\! & & \\
\hline
$F^0H$ & $\PHDP{0},$& $\ldots~,$ &$\PHDP{n-1},$ & $\PHPR{n},$ & $ \PHPL{n},$ &  $\PHDM{n-1},$ &  $\ldots~,$ &  $\PHDM{0}$
\\ 
\hline
$F^1H$ &  $\FHP{1}{0},$ & $ \ldots~,$ & $\FHP{1}{n},$ & $\PHPR{n-1},$&  $\PHPL{n-1},$& $\FHM{1}{n},$& $ \ldots~,$ & $\FHM{1}{0}$
\\
\hline
\vdots & & &\!\!\!\!\!\!\!\!\!\!\!\!\!\!\vdots& & & \vdots\!\!\!\!\!\!\!\!\!\!
\\
\hline
$F^pH$ &  $\FHP{p}{0},$ & $ \ldots~,$ & \!\!$F^p\!H_{+}^{n+p-1}\!\!\!\!,$ & $\PHPR{n-p},$&  $\PHPL{n-p},$& \!\!$F^p\!H_{-}^{n+p-1}\!\!\!\!,$& $ \ldots~,$ & $\FHM{p}{0}$
\\
\hline
\vdots & & & \!\!\!\!\!\!\!\!\!\!\!\!\!\!\vdots& & & \vdots\!\!\!\!\!\!\!\!\!\!
\end{tabular}
\end{center}
 \caption{The filtered cohomologies $F^pH=\left\{\FHP{p}{}, \FHM{p}{} \right\}$ with $0\leq p \leq n$ with isomorphisms  $\FHP{p}{n+p}\! \cong\! \PHPR{n-p}$ and $\FHM{p}{n+p}\! \cong \!\PHPL{n-p}$.
\label{tabcoh2}} 
\end{table} 
 
Having introduced filtered cohomologies, it may seem that we have now a full-blown array of cohomologies arranged together by the filtration degree $p$ in $F^pH$.  But why so many?  And what information do these cohomologies contain?  It turns out the answers are directly related to Lefschetz maps, which are fundamental algebraic operations in symplectic geometry.  Let us turn to describe them now. 

\medskip

\noindent{\it (2) Cohomologies and Lefschetz maps}

For any symplectic manifold, there is a most distinguished set of elements of the de Rham cohomology consisting of the symplectic form and its powers, $\{\om, \om^2, \ldots , \om^n\}\,$.  As the de Rham cohomology has a product structure given by the wedge product, it is natural to focus in on the product of $\om^r \in \HD{2r}(M)$, for $r=1, \ldots, n$, with other elements of the de Rham cohomology, i.e. $\om^r \otimes \HD{k}(M)$.   Such a product by $\om^r$ can be considered as a map, taking an element of $\HD{k}(M)$ into an element of $\HD{k+2r}(M)$.  This action is  referred to as the Lefschetz map (of degree $r$):
\begin{align}\begin{split}\label{L1}
L^r:& ~~\HD{k}(M) ~~\to ~~\HD{k+2r}(M)~,\\
& ~~~~ [A_k] ~\,\quad \to~~~ [\om^r \w A_k]~,
\end{split}
\end{align}
where $[A_k] \in \HD{k}(M)$.  
Clearly, Lefschetz maps are linear and only depend on the cohomology class of $[\om^r] \in \HD{2r}(M)$.  But importantly, Lefschetz maps are in general neither injective nor surjective.  So a basic question one can ask for any symplectic manifold is what are the kernels and cokernels of the Lefschetz maps? 

In the special case in which the symplectic manifold is K\"ahler, this question is directly answered by the well-known Hard Lefschetz theorem.  But for a generic symplectic manifold, the Hard Lefschetz theorem does not hold.  We can nevertheless address this question in full generality by first analyzing the Lefschetz action on differential forms.  Indeed, the degree one Lefschetz map, $L$, is one of the three $\mathfrak{sl}(2)$ generators that lead to the Lefschetz decomposition of differential forms.   We will show in Section \ref{sec_pre} that the information of this Lefschetz decomposition can be neatly re-packaged in terms of a series of short exact sequences of differential forms involving Lefschetz maps.  Though these exact sequences do not naturally fit into a single short exact sequence of chain complexes, we will prove in Section \ref{sec_Lef} that they do nevertheless give a long exact sequence of cohomologies involving the Lefschetz maps.  

These long exact sequences of cohomologies turn out to contain precisely the data of the kernels and cokernels of the Lefschetz maps.   As we will show in Section \ref{sec_Lef}, for Lefschetz maps of degree $r$, there is a two-sided resolution that involves precisely the $(r-1)$-filtered cohomologies, $F^{r-1}H\,$.  (For the $r=1$ case, this result concerning the primitive cohomologies was also found independently by M. Eastwood via a different method \cite{Eastwood2}).  We can in fact encapsulate the resolution of the degree $r$ Lefschetz map in a simple, elegant, exact triangle diagram of cohomologies:
\begin{align}
\xymatrix{
 &\; F^{r-1}H^*(M) \ar[ld] &  \\
 \HD{*}(M) \ar[rr]^-{L^r} & & \HD{*}(M)\ar[lu]}\label{Tp}
\end{align}
For example, in dimension $2n=4$, the triangle for $r=1$ represents the following long exact sequence:
\begin{align*} 
\xymatrix
@R-6pt
@H-10pt
{&\;0 \ar[r] & \;\HD{1}(M) \ar[r] &\;\PHDP{1}(M) \ar`[d]`[l]`[dlll]`[dll][dll]
\\
&\;\HD{0}(M) \ar[r]^-{L}  &\;\HD{2}(M) \ar[r]  &\;\PHPR{2}(M) \ar`[d]`[l]`[dlll]`[dll][dll]
\\
&\;\HD{1}(M) \ar[r]^-{L}   &\;\HD{3}(M) \ar[r] &\;\PHPL{2}(M)\ar`[d]`[l]`[dlll]`[dll][dll]
\\
&\;\HD{2}(M) \ar[r]^-{L}  &\;\HD{4}(M) \ar[r]&\;\PHDM{1}(M) \ar`[d]`[l]`[dlll]`[dll][dll]
\\
&\;\HD{3}(M) \ar[r] & \; 0 &
}
\end{align*}
Since the information of the long exact sequence at each element can be written as a split short exact sequence of kernel and cokernel of maps, we immediately find from the above exact sequence for example in four dimensions that
\begin{align*}
\PHPR{2}(M) &\cong \cok[L: \HD{0}(M) \to \HD{2}(M)] \oplus \ker[L: \HD{1}(M) \to \HD{3}(M)]~,\\
\PHPL{2}(M)&\cong \cok[L: \HD{1}(M) \to \HD{3}(M)] \oplus \ker[L: \HD{2}(M) \to \HD{4}(M)]~.
\end{align*}

In general, the triangle \eqref{Tp} and its associated long exact sequence implies that the $(r-1)$-filtered cohomologies $F^{r-1}H^*(M)$ are isomorphic to the direct sum of kernels and cokernels of the Lefschetz maps of degree $r$.  More explicitly, we have the following isomorphisms:
\begin{align*}
\FHP{r-1}{k}(M) &\!\cong\! \cok\!\big[L^{r}\!: \!\HD{k-2r}(M)\to \HD{k}(M)\big] \oplus \ker\big[L^{r}\!: \!\HD{k-2r+1}(M) \to \HD{k+1}(M)\big] \,,\\
\FHM{r-1}{k}(M) &\!\cong\! \cok\!\big[L^{r}\!\!:\! \HD{2n-k-1}\!(M)\! \to \!\HD{2n-k+2r-1}\!(M)\big]\! \oplus\! \ker\!\big[L^{r}\!\!:\! \HD{2n-k}\!(M)\! \to\! \HD{2n-k+2r}\!(M)\big],
\end{align*}
for $0\leq k \leq n+r-1$.

In fact, we can also think of the Lefschetz map triangle \eqref{Tp} as a special case of the exact triangle relating filtered cohomologies:
\begin{align}
\xymatrix{
 &\; F^{r-1}H^*(M) \ar[ld] &  \\
 F^{l}H^*(M) \ar[rr] & & F^{l+r}H^*(M)\ar[lu]}\label{Tpf}
\end{align}
as we will also show in Section \ref{sec_Lef}.

\medskip

\noindent{\it (3) Filtered cohomology rings and their underlying $A_\infty$-algebras}

It is indeed rather interesting that the filtered cohomologies $F^pH$, defined differentially by the elliptic complex \eqref{filterecomplex}, are isomorphic via the exact triangle \eqref{Tp} with the kernels and cokernels of the Lefschetz maps, which are purely algebraic quantities.  Considering the exact triangle \eqref{Tp}, it is further tempting to think that the filtered cohomologies may have similar algebraic properties to the two de Rham cohomologies that accompany it.  For instance, could the $p$-filtered cohomology group $F^pH$ actually form a cohomology ring?  Ideally, to consider this question, one would like to have a grading for $p$-filtered forms and introduce a product operation that preserves the grading.  But what grading should one use for $p$-filtered forms?  This is not at all immediate as each filtered space $\FO{p}{k}$ for $0\leq k\leq n+p$ noteworthily appears twice in the elliptic complex of \eqref{filterecomplex}?  To settle on a grading, we can appeal to the analogy with the de Rham complex, and heuristically, just ``bend" the elliptic complex of \eqref{filterecomplex} and rearrange it into a single line
\begin{align*}
\xymatrix{
0\; \ar[r] &\; \FO{p}{0} \; \ar[r]^-{\ddp} & ~ \ldots ~ \ar[r]^-{\ddp}& \; \FO{p}{n+p} \; \ar[r]^-{\dpp\dpm} & \; \BFO{p}{n+p} \ar[r]^-{\ddm} & ~ \ldots ~ \ar[r]^-{\ddm} & \; \BFO{p}{0} \; \ar[r]^-{\ddm} &\;0 } 
\end{align*}
where we have used a bar, $\BFO{p}{*}$, to distinguish those $\FO{p}{*}$ associated with the bottom level of the elliptic complex.  Writing the complex in this form,  we can construct a new algebra 
$$\Fp=\{\FO{p}{0}, \FO{p}{1}, \ldots, \FO{p}{n+p}, \BFO{p}{n+p}, \dots, \BFO{p}{1}, \BFO{p}{0}\}$$
with elements $\Fp^j$, for $0\leq j \leq 2n+2p+1$, given by
\begin{equation*}\Fp^j =\begin{cases}
\FO{p}{j} &  \text{ if } 0\leq j\leq n+p ~,\\
\BFO{p}{2n+2p+1- j} &  \text{ if } n+p+1\leq j \leq 2n+2p +1 ~.
\end{cases}
\end{equation*}
Following closely the elliptic complex, we define the differential $d_j: \Fp^j \to \Fp^{j+1}$ to be
\begin{align}\label{diffdef}
d_j= \begin{cases}
~\ddp &  \text{if~~ } 0\leq j < n+p-1~,\\
-\dpp\dpm & \text{if ~~} j = n+p~,\\
-\ddm & \text{if~~ } n+p+1 \leq j \leq 2n + 2p +1 ~.
\end{cases}
\end{align}
One can then try to construct a multiplication which preserves the grading
$$\Fp^j \pp \Fp^k \to \Fp^{j+k}$$
and is graded commutative, i.e. $\Fp^{j} \pp \Fp^{k} = (-1)^{jk} \Fp^{k} \pp \Fp^{j}$.  In fact, as we will describe in Section \ref{sec_a_infinity}, such a multiplication operation $\pp$ can indeed be constructed based on the exact triangle \eqref{Tp}.  (See Definition \ref{defprod}.) This multiplication on forms is rather novel in that it involves first-order derivative operators.  The presence of these derivatives turn out to be important as they allow us to prove the Leibniz rule:
\begin{align*}
d_{j+k} (\Fp^{j} \pp \Fp^{k}) = (d_j \Fp^{j}) \pp \Fp^{k} + \moj \Fp^j \pp (d_k \Fp^k)~.
\end{align*}
This Leibniz rule represents a rather subtle balancing between the definition of the differential $d_j$ and the product $\pp$.   However, one of the consequences of having derivatives in the definition of a multiplication is that the product $\pp$ generally is non-associative.  This means that the algebra $(\Fp, d_j, \pp)$ can not be a differential graded algebra as in the de Rham complex case.  Nevertheless, as we will show in Section \ref{sec_a_infinity},  the non-associativity of the $p$-filtered algebra $\Fp$ can be captured by a trilinear map $m_3$.  Together, $(\Fp, d_j, \pp, m_3)$ turns out to fit precisely the $\Ain$-algebra structure, with the higher $k$-linear maps, $m_k$, for $k\geq 4$ taken to be zero.   And as an immediate corollary of satisfying the requirements of an $\Ain$-algebra, the cohomology $F^pH=H(\Fp)$ indeed has a ring structure.

The existence of the ring structure of $F^pH$ provides a new set of invariants for distinguishing different symplectic manifolds.  We will show in Section \ref{example_product}, how the product structure can be different for two symplectic four-manifolds that are both a product of a circle times a fibered three-manifold.  We give an example of a pair of such symplectic manifolds that have isomorphic de Rham cohomology ring and identical filtered cohomology dimensions, but with different filtered product structure.  

\medskip

\noindent{\it Acknowledgements.~} 
We would like to thank M. Abouzaid, V. Baranovsky, V. Guillemin, N. C. Leung, T.-J. Li, B. Lian, T. Pantev, R. Stern, C. Taubes, C.-L. Terng, S. Vidussi, L. Wang, and B. Wu for helpful comments and discussions.  The first author is supported in part by Taiwan NSC grant 102-2115-M-002-014-MY2.  The third author is supported in part by NSF grants 1159412, 1306313 and 1308244.


\section{Preliminaries}\label{sec_pre}

We here present some of the properties of differential forms and differential operators on symplectic manifolds that will be relevant for our analysis in later sections.   We begin first by describing the $\mathfrak{sl}(2)$ action and other natural operations on differential forms.  We then introduce the filtered forms and discuss the linear differential operators that act on them.  We then use these filtered forms to write down short exact sequences involving Lefschetz maps. 

\subsection{Operations on differential forms}

On a symplectic manifold $(M^{2n},\om)$, the presence of a non-degenerate two-form, $\om$, allows for the decomposition of differential forms into representation modules of an $\mathfrak{sl}(2)$ Lie algebra which has the following generators, 
\begin{align}\label{def_sl2}\begin{split}
L &: ~~ A \rightarrow \om \w A~,  \\
\La &:~~ A \rightarrow \frac{1}{2}(\om^{-1})^{ij}\,\iota_{\pa_{x^i}} \iota_{\pa_{x^j}} A~,  \\
H &:~~ A \rightarrow (n-k)\, A  \qquad {{\rm for~}} A\in \CA^k(M)~,
\end{split}\end{align}
and commutation relations,
\begin{align}\label{sl2com}
[H,\La] = 2\La\,, \quad [H, L] = -2 L\,,  \quad [\La, L] = H\,.
\end{align}
Here, $L$ is called the Lefschetz operator, and is just the operation of wedging a form with $\om$.  The operator $\La$ represents the action of the associated Poisson bivector field.    
The ``highest weight" forms are the primitive forms, which are denoted by $B_s\in\CB^s(M)$.  The primitive forms are characterized by the following condition:
\begin{align}\label{primcond}
(\text{primitivity condition})\qquad\qquad \La\, B_s = 0\,, \quad \text{or equivalently},\quad L^{n+1-s}\, B_s =0\,.
\end{align}
An $\mathfrak{sl}(2)$ representation module then consists of the set
$$\left\{B_s\,, ~\om \w B_s\,, ~\om^2 \w B_s\,,~ \ldots \, ,~ \om^{n-s}\w B_s \right\}~,$$
where each basis element can be labeled by the pair $(r,s)$ with
$$\CL^{r,s}(M)= L^r \CB^s(M) = \left\{A \in \CA^{2r+s}(M){\big \arrowvert} A= \om^r \w B_s\, ~{\rm and}~~ \La\, B_s =0 \right\}\,.$$
Indeed, the $\mathfrak{sl}(2)$ decomposition of $\CA^*(M)$ can be simply pictured by an $(r,s)$-pyramid diagram \cite{TY2} as for example drawn in Figure \ref{sympr} for dimension $2n=8$.
\begin{figure}
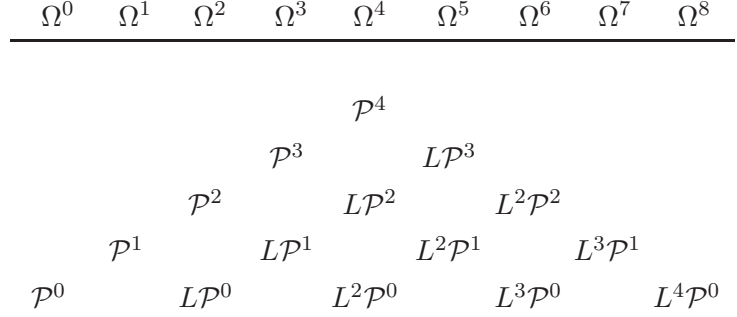

$$\begin{matrix}
~\;\CA^0 & ~\;\CA^1 & \;\CA^2 & \;\CA^3 & \CA^4 & \CA^5 & \CA^6 & \CA^7 & \CA^8 
\\
\hline\\
 & & & & \CB^4& & & & \\
 & & &\CB^3& &\!\!L \CB^{3} & & & \\
 & &\CB^2& & L\CB^{2}& &\!\!\!L^2 \CB^{2} & & \\
 & \CB^1 & & L \CB^{1} & &\!\!L^2 \CB^{1}& &\!\!\!L^3\CB^{1} & \\
\CB^0& &L \CB^{0}& &\!\!L^2 \CB^{0}& &\!\!\!L^3\CB^{0}& &\!\!\!L^4\CB^{0}\\
\end{matrix}$$
\caption{The decomposition of differential forms into an $(r,s)$-pyramid diagram in dimension $2n=8$.   The degree of the forms starts from zero (on the left) to $2n=8$ (on the right).}
\label{sympr}
\end{figure}

We now introduce some natural operations on differential forms that will be useful for the discussion to follow.  

First, note that there is an obvious reflection symmetry about the central axis of the $(r,s)$-pyramid diagram as in Figure \ref{sympr}.  This central axis lies on forms of middle degree, $\CA^n$.  An example of an operator that reflects forms is the standard symplectic star operator $\sstar\,$, which can be defined by the Weil's relation \cite{Weil,Guillemin}
\begin{align} \label{ssrel}
\sstar \;\frac{1}{r!}\, L^r\, B_s = (-1)^{s(s+1)/2}\, \frac{1}{(n-r-s)!}\, L^{n-r-s}\, B_s~.
\end{align}
where $B_s \in \CB^s$.  The minus sign and combinatorial factors in \eqref{ssrel} however can become rather cumbersome for calculations.  So for simplicity, we introduce another reflection operator, denoted as $\rstar\,$, and defined it simply as
\begin{align} \label{rsrel}
\rstar \, (L^r\, B_s) =  L^{n-r-s}\, B_s~.
\end{align}
It is easy to check as expected that 
$$(\rstar)^2 = \bo ~.$$ 

Second, we can broaden the definition of the Lefschetz operator $L^r$ to allow for negative integer powers, i.e. $r<0\,$.  For a differential $k$-form $A_k\in\CA^k\,$, consider its Lefschetz decomposition
\begin{align}\label{LLdec}
A_k= B_k +  L B_{k-2} + \ldots + L^p B_{k-2p} + L^{p+1} B_{k-2p-2} + \ldots ~.
\end{align}
The map $L^{-p}: \CA^k \to \CA^{k-2p}$ for $p>0$ is then defined to be
\begin{align}\label{LLdec1}
L^{-p} A_k = B_{k-2p} + L B_{k-2p-2} + \ldots ~.
\end{align}
Notice that this action of $L^{-p}$ is similar to $\La^p$ in that both lower the degree of a differential form by $2p$; however, they are \emph{not} identical.  This can be easily seen by noting that by definition, $L^{-1}(L\,B_s) = B_s$, while $\La (L\,B_s) = H B_s$, using the third $\mathfrak{sl}(2)$ commutation relation in \eqref{sl2com}.  Defining $L$ raised to a negative power by \eqref{LLdec1} is useful in that it allows us to express the reflection $\rstar$ operator simply as
\begin{align}\label{rsak}
\rstar\,  A_k = L^{n-k} A_k~,
\end{align}
for $k$ arbitrary.   In fact, $L^{-p}$ can be heuristically thought of as the $\rstar$ adjoint of $L^p$: 
\begin{align}\label{rsL}
L^{-p} = \rstar \, L^{p} \, \rstar
\end{align}
which can be checked straightforwardly using \eqref{rsak}.   This relation \eqref{rsL} indeed is analogous to the standard adjoint relation, $\La = \sstar\, L\, \sstar\,$  \cite {Yan}.

Comparing \eqref{LLdec1} with  \eqref{LLdec}, it is clear that the operator $L^{-p}$ completely removes the Lefschetz components of a form that have powers of $\om$ less than $p$.  This suggests it may be useful to introduce an operator that projects onto the Lefschetz components with powers of $\om$ bounded by some integer.  Thus, we define a projection operator, $\pip: \CA^k \to \CA^k$ for $p>0$ which acts on the Lefschetz decomposed form of \eqref{LLdec} as
\begin{align}\label{pipdef}
\pil{p} A_k = B_k +  L B_{k-2} + \ldots + L^p B_{k-2p}~.
\end{align}
In other words, it projects out components with higher powers of $\om$, i.e. $(L^{p+1}B_{k-2p-2}+ \ldots )$.   With such a projection operator, we can express any differential form as
\begin{align*}
A_k = \pil{p} A_k + L^{p+1} (L^{-(p+1)} A_k)~,
\end{align*}
which simply implies 
\begin{align}\label{relation1}
\bo = \pip +  L^{p+1} L^{-(p+1)} ~.
\end{align} 
Written in this form, it is clear that $L^{p+1} L^{-(p+1)}$ is also a projection operator and we can alternatively define the projection operator as $\pip = \bo - L^{p+1} L^{-(p+1)}\,$.

As will be useful later, we can take the $\rstar$ adjoint of \eqref{relation1} and obtain the following:
\begin{align}\label{relation2}
\bo = \pips + \lpo L^{p+1}~,
\end{align}
where  $\pips = \rstar \pip \rstar$.  Note that both $\pips$ and $\lpo L^{p+1}$ are also projection operators.

Regarding differential operators, we first point out that the exterior derivative operator, $d$, has a natural decomposition into two linear differential operators from the above $\mathfrak{sl}(2)$ or Lefschetz decomposition \cite{TY2}:
\begin{align}
d &= \dpp +  L \, \dpm \label{dcomp}
\end{align}
where $\dppm: \CL^{r,s} \to \CL^{r,s\pm 1}$.  The differential operators $(\dpp,\dpm)$ have the desirable properties that
\begin{align*}
(\dpp)^2=(\dpm)^2 =0~, \qquad L\,(\dpp\dpm + \dpm\dpp) = 0~,\qquad [L, \dpp]=[L, L\,\dpm]=0~.
\end{align*}

Another useful operator is the symplectic adjoint of the exterior derivative \cite{EH,Brylinski} 
\begin{align}
\dl :&= d\,\La - \La\, d \notag\\
&= (-1)^{k+1}\,\sstar\,d\,\sstar \label{dldef}
\end{align}
where the second relation is defined acting on a differential $k$-form.  Analogous to $\dl$, we shall introduce the $\rstar$ adjoint operator defined to be 
\begin{align}\label{dfddm}
\ddm : =  \rstar \, d\, \rstar~.
\end{align}
It follows trivially from $(\rstar)^2 = \bo$ that $\ddm\, \ddm =0\,$.  In fact, it is can be straightforwardly checked for any $\CL^{r,s}(M)$ space that 
\begin{align}\label{dmsimp}
\ddm = \dpm + \dpp \,L^{-1}~.
\end{align}

\subsection{Filtered forms and differential operators}

The Lefschetz decomposition is suggestive of a natural filtration for differential forms on $(M^{2n} ,\om)$.  The space of differential $k$-forms $\CA^k$ has the following Lefschetz decomposition:
\begin{align} \label{sldec}
\CA^k = \bigoplus_{\max(0,k-n)\,\leq \,r \,\leq\, \lfloor{k/2}\rfloor} \CL^{r,k-2r} ~, 
\end{align}
where $\lfloor\ \rfloor$ denotes rounding down to the nearest integer.   Applying the projection $\pil{p}$, which caps the sum over the index $r$ to some fixed integer $p$, we define the $p$-filtered forms of degree $k$ as  
\begin{align} \FO{p}{k} = \;\pil{p}\CA^k= \bigoplus_{\max(0,k-n)\leq \,r \, \leq \min(p,\lfloor{k/2}\rfloor)} \CL^{r,k-2r} ~. \end{align}
We call $p$ the filtration degree which can range from $0$ to $n\,$. 
Let us note the two special cases of $\FO{p}{k}$: (i)  $p=0$ consists of primitive $k$-forms, $\FO{0}{k}= \CB^k$; (ii) $p= \lfloor{k/2}\rfloor$ consists of all differential $k$-forms, $\FO{\lfloor{k/2}\rfloor}{k}=\CA^k\,$.  Clearly then, 
\begin{align*}
\CB^k = \FO{0}{k} \subset \FO{1}{k} \subset \FO{2}{k} \subset \ldots \subset \FO{\lfloor{k/2}\rfloor}{k} = \CA^k~.
\end{align*}

Now, consider the elements of $\FO{p}{*}$ for a fixed filtration number $p$.  In this case, the degree $k$ in $\FO{p}{k}$ has the range $0\leq k \leq n+p$.  For $p\geq 1$, we have
\begin{align}
\FO{p}{k}& = ~\CA^k~,\qquad\qquad\qquad\qquad\qquad\qquad{\text{for~~}} 0 \leq k \leq 2p+1~,\label{klep}\\
 \vdots \qquad &  \qquad \vdots \notag\\
\FO{p}{n+p-1} &= L^{p-1}\CB^{n-p+1}\,\oplus\,L^p\CB^{n-p-1}~,\label{kpmo}\\
\FO{p}{n+p} & = L^{p} \CB^{n-p}~.\label{kpnp}
\end{align}
Hence, for $k$ sufficiently small, $\FO{p}{k}$ contains all differential $k$-forms.  On the other hand, for the largest value $k=n+p$, $\FO{p}{n+p}$ has only one component of the Lefschetz decomposition of \eqref{sldec} and is isomorphic to $\CB^{n-p}\,$.  The example of $\FO{2}{*}$ in dimension $2n=8$ is illustrated in Figure \ref{fthree}.

\begin{figure}
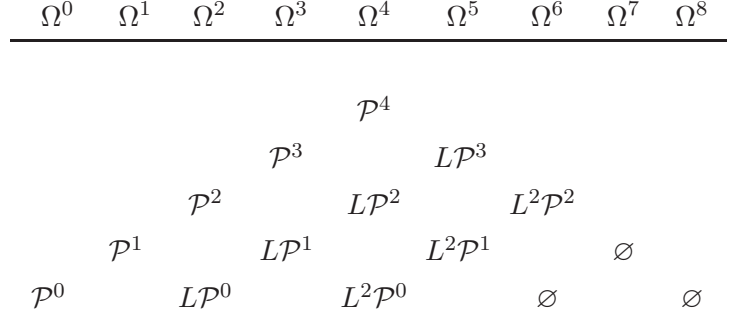

$$\begin{matrix}
~\;\CA^0 & ~\;\CA^1 & \;\CA^2 & \;\CA^3 & \CA^4 & \CA^5 & \CA^6 & \CA^7 & \CA^8 
\\
\hline\\
 & & & & \CB^4& & & & \\
 & & &\CB^3& &\!L \CB^{3} & & & \\
 & &\CB^2& & L\CB^{2}& &\!\!L^2 \CB^{2} & & \\
 & \CB^1 & & L \CB^{1} & &\!\!L^2 \CB^{1}& & {~\varnothing~} & \\
\CB^0& &L \CB^{0}& &L^2 \CB^{0}& & {~\varnothing~ }& & {~\varnothing~ }\\
\end{matrix}$$
\caption{The forms of $\FO{2}{k}$ with $0\leq k\leq 6\,$ decomposed in an $(r,s)$ pyramid diagram in dimension $2n=8$.   Notice that $\FO{2}{k} = \CA^k\,$, for $ k \leq 5\,$. }
\label{fthree}
\end{figure}

Alternatively, we can also define the filtered form based on loosening the primitivity condition of \eqref{primcond}.
\begin{defn}\label{pfcond}
A differential $k$-form $A_k$ with $k\leq n+p$ is called $p$-filtered, i.e. $A_k \in \FO{p}{k}\,$, if it satisfies the two equivalent conditions: (i) $ \La^{p+1} A_k = 0\,$; 
(ii) $ L^{n+p+1-k} A_k = 0\,$. 
\end{defn}
By equation \eqref{rsak}, the $p$-filtered condition can be equivalently expressed as
\begin{align} \label{pistarf}
L^{p+1}\, \rstar \, A_k =0~.
\end{align}
for $A_k\in \FO{p}{k}\,$.

Turning now to the differential operators that act within the filtered spaces, the composition of the projection operator with the exterior derivative induces the differential operator 
\begin{align*}
\ddp: \FO{p}{k} \stackrel{d}{\longrightarrow} \Omega^{k+1} \stackrel{\pil{p}}{\longrightarrow} 
\FO{p}{k+1}~.
\end{align*}
The projection operator $\pil{p}$ effectively drops the $L\dpm$ action on the $\CL^{p,k-2p}$ component.   Explicitly, letting $A_k\in \FO{p}{k}$, we can write
\begin{align*}
\ddp A_k &= \ddp \left(B_k + L B_{k-2} + \ldots + L^p B_{k-2p}\right) \\
& = d \left(B_k + \ldots + L^{p-1} B_{k-2p+2}\right) + L^p \dpp B_{k-2p}
\end{align*}
where in the second line, we have applied \eqref{dcomp} on the $L^p\,d B_{k-2p}$ term and projected out the resulting term $L^{p+1} \dpm B_{k-2p}\,$.   Now, applying $\ddp$ again, we find
\begin{align*}
\ddp(\ddp A_k)  = d^2\left( B_k + \ldots + L^{p-1} B_{k-2p+2}\right) + L^p (\dpp)^2 B_{k-2p} = 0~,
\end{align*}
hence, we have shown that $(\ddp)^2 =0\,$.  We remark that with \eqref{klep}, $d_+$ is just the exterior derivative $d$ when $k\leq 2p$.    

Now we will also be interested in the action of the $\ddm$ operator \eqref{dfddm} on filtered forms.  Indeed, acting on $\FO{p}{*}$, it preserves the filtration number and decreases the degree by one: 
$$ \ddm: \FO{p}{k} {\longrightarrow} \FO{p}{k-1}
~. $$
The convention for the $\pm$ signs for $\ddp$ and $\ddm$ indicates that the differential operator raises and lowers the degree of differential forms by one, just like the notation for $(\dpp, \dpm)$. 

What follows is a formula involving the relation of the operators $\ddm$, $\pip$, and $\rstar$.  It will be helpful for calculations in later sections.
\begin{lem}\label{Leibniz02}
For $A_k \in \CA^k$,
\begin{align}\label{Lei020}
\pil{p}\, \rstar\, (dA_k) = \ddm \left(\pil{p}\, \rstar\,  A_k\right) + \pprs d \, \lpo(\opo \w A_k) 
\end{align}
\end{lem}
\begin{proof}
Since $\ddm: \FO{p}{k} \to \FO{p}{k-1}$ preserves the filtration degree,
\begin{align*}
\ddm \pil{p} A_k & = \pil{p} \ddm \pil{p} A_k \\
&=\pil{p} \ddm A_k - \pprs d \, \rstar \, \opo\lpo A_k ~.
\end{align*}
Replacing $A_k$ with $\rstar\, A_k$ and using \eqref{dfddm} and \eqref{rsL}, we obtain \eqref{Lei020}.
\end{proof}

\subsection{Short exact sequences}

The data of the $(r,s)$ pyramid diagram can be nicely repackaged in terms of short exact sequences.  For instance, from the pyramid diagram of Figure \ref{sympr}, it is not hard to see that the following sequences involving Lefschetz map of degree one are exact:
\begin{align*}
\xymatrix
@R-10pt
@H-10pt
{
\;0\ar[r] & \; \CA^{0} \ar[r]^-L & \; \CA^{2} \ar[r]^-{\pil{0}} &\; \CB^2 \ar[r] & \; 0\\
\;0\ar[r] & \; \CA^{1} \ar[r]^-L & \; \CA^{3} \ar[r]^-{\pil{0}} &\; \CB^3 \ar[r] & \; 0\\
\;0\ar[r] & \; \CA^{2} \ar[r]^-L & \; \CA^{4} \ar[r]^-{\pil{0}} &\; \CB^4 \ar[r] & \; 0 ~.\\
}
\end{align*}
At the middle of the pyramid, we have
\begin{align*}
\xymatrix
@R-10pt
@H-10pt
{
\;0\ar[r] & \; \CA^{3} \ar[r]^-L & \; \CA^{5} \ar[r] & \; 0~.\\
}
\end{align*}
For forms of degree four or greater, we can write
\begin{align*}
\xymatrix
@R-10pt
@H-10pt
{
\;0\ar[r] & \; \CB^{4} \ar[r]^-{\rstar} & \; \CA^{4} \ar[r]^-L &\; \CA^6 \ar[r] & \; 0\\
\;0\ar[r] & \; \CB^{3} \ar[r]^-{\rstar} & \; \CA^{5} \ar[r]^-L &\; \CA^7 \ar[r] & \; 0\\
\;0\ar[r] & \; \CB^{2} \ar[r]^-{\rstar} & \; \CA^{6} \ar[r]^-L &\; \CA^8 \ar[r] & \; 0~.\\
}
\end{align*}
Exact sequences involving higher degree Lefschetz maps $L^r$ can similarly be written using $\FO{r-1}{*}$.  Generally, we can arrange the short exact sequence of any filtration number together in a suggestive commutative diagram as follows.

\begin{lem}\label{ccomplex1}On a symplectic manifold $(M^{2n},\omega)$, there is the following commutative diagram of short exact sequences for $1\leq r<n$:
\begin{align}
\xymatrix
@R-9pt
@H+12 pt
{
& & &  \vdots \ar[d]_-d &  \vdots \ar[d]_-{\ddp} &\; \\
& & \;0\ar[r] & \; \CA^{2r-1} \ar[r]^-{\pil{r-1}} \ar[d]_-d & \; \FO{r-1}{2r-1} \ar[r] \ar[d]_-{\ddp} &\; 0  \\
&\;0\ar[r]  &\; \CA^0 \ar[r]^-{L^r} \ar[d]_-d & \; \CA^{2r} \ar[r]^-{\pil{r-1}} \ar[d]_-d  &\; \FO{r-1}{2r} \ar[r]  \ar[d]_-{\ddp}  &\; 0\\
& & \vdots \ar[d]_-d & \vdots  \ar[d]_-d & \vdots \ar[d]_-{\ddp} & \\
&\; 0 \ar[r] &\; \CA^{n-r-2}\ar[r]^-{L^r} \ar[d]_-d & \; \CA^{n+r-2} \ar[r]^-{\pil{r-1}} \ar[d]_-d &\;\FO{r-1}{n+r-2} \ar[r] \ar[d]_-{\ddp} &\;0\\
&\; 0 \ar[r] &\; \CA^{n-r-1}\ar[r]^-{L^r} \ar[d]_-d & \; \CA^{n+r-1} \ar[r]^-{\pil{r-1}} \ar[d]_-d &\; \FO{r-1}{n+r-1} \ar[r] &\; 0\\
&\; 0\ar[r] &\; \CA^{n-r} \ar[r]^-{L^r} \ar[d]_-d & \; \CA^{n+r} \ar[r] \ar[d]_-d &\; 0 &\\
0\ar[r] &\; \FO{r-1}{n+r-1} \ar[r]^-{\rstar} \ar[d]_-{\ddm}&\; \CA^{n-r+1} \ar[r]^-{L^r} \ar[d]_-d 
&\; \CA^{n+r+1}\ar[r] \ar[d]_-d &\; 0 &  \\
0\ar[r] &\; \FO{r-1}{n+r-2} \ar[r]^-{\rstar} \ar[d]_-{\ddm}  &\; \CA^{n-r+2} \ar[r]^-{L^r} \ar[d]_-d  &\; \CA^{n+r+2}\ar[r] \ar[d]_-d   &\; 0 &  \\
& \vdots \ar[d]_-{\ddm} & \vdots \ar[d]_-d & \vdots \ar[d]_-d &  & \\
0\ar[r] &\; \FO{r-1}{2r} \ar[r]^-{\rstar} \ar[d]_-{\ddm} &\; \CA^{2n-2r} \ar[r]^-{L^r} \ar[d]_-d & \;\CA^{2n} \ar[r]   &\; 0  & \\ 
 0\ar[r] & \; \FO{r-1}{2r-1} \ar[r]^-{\rstar} \ar[d]_-{\ddm} &\; \CA^{2n-2r+1}\ar[r]  \ar[d]_-d &\; 0 &  &\\
\: & \; \vdots & \;\vdots & \; & & \\ 
}\label{ccomplex1e}
\end{align}
\end{lem}
Note that $\FO{r-1}{2r-1} = \FO{\lfloor 2r-1\rfloor}{2r-1}=\CA^{2r-1}\,$.   The above is not a standard exact sequence between the three complexes $(\CA^*, \CA^*, \FO{r-1}{*})$ due to a ``shift" in the middle of the diagram.  This shift is due to the structure of the $(r,s)$ pyramid, and as we will see in Section \ref{sec_Lef}, it provides an explanation for the presence of cohomologies that involve the 2nd order differential operator $\dpp\dpm\,$.

Additionally, for the Lefschetz map, we can also write short exact sequence involving filtered forms $(\FO{l}{*}, \FO{l+r}{*}, \FO{r-1}{*})$.  For instance, for the pyramid diagram Figure \ref{fthree}, the elements of the complexes $(\FO{1}{*}, \FO{2}{*}, \FO{0}{*})$ gives the following series of short exact sequences (with three shifts):
\begin{align*}
\xymatrix
@R-8pt
@H+2pt{
& && \;0\ar[r] & \; \FO{1}{0} \ar[r]^-L & \; \FO{2}{2} \ar[r]^-{\pil{1}} &\; \FO{0}{2} \ar[r] & \; 0\\
& &&\;0\ar[r] & \; \FO{1}{1} \ar[r]^-L & \; \FO{2}{3} \ar[r]^-{\pil{1}} &\; \FO{0}{3} \ar[r] & \; 0\\
& &&\;0\ar[r] & \; \FO{1}{2} \ar[r]^-L & \; \FO{2}{4} \ar[r]^-{\pil{1}} &\; \FO{0}{4} \ar[r] & \; 0\\
& &&\;0\ar[r] & \; \FO{1}{3} \ar[r]^-L & \; \FO{2}{5} \ar[r] &\; 0 & \\
& &\;0\ar[r]&\; \FO{0}{4} \ar[r]^{\rstar} &  \; \FO{1}{4} \ar[r]^-L & \; \FO{2}{6} \ar[r] &\; 0 \\
& &\;0\ar[r]&  \; \FO{0}{3} \ar[r]^-{\rstar} & \; \FO{1}{5} \ar[r] &\; 0 & &\\
&\;0\ar[r]&  \; \FO{2}{6} \ar[r]^-{L^{-2}} & \; \FO{0}{2} \ar[r] &\; 0 & &\\
\;0\ar[r]&\; \FO{1}{5} \ar[r]^-{\iota} &  \; \FO{2}{5} \ar[r]^-{L^{-2}} & \; \FO{0}{1} \ar[r] &\; 0 & & &\\
\;0\ar[r]&\; \FO{1}{4} \ar[r]^-{\iota} &  \; \FO{2}{4} \ar[r]^-{L^{-2}} & \; \FO{0}{0} \ar[r] &\; 0 & & &\\}
\end{align*}

\

In general, we can write the following chain of short exact sequences. 
\begin{lem}\label{ccomplex2}On a symplectic manifold $(M^{2n},\omega)$, there is the following commutative diagram of short exact sequences for $1\leq r<n$:
\begin{align*}
\xymatrix
@R-10pt
@C-10pt
@H-13pt
{
& & & & & \vdots \ar[d]_-{\ddp} & \vdots \ar[d]_-{\ddp} & \\
& && &\;0\ar[r] & \; \FO{l+r}{2r-1} \ar[r]^-{\pil{r-1}} \ar[d]_-{\ddp}&\; \FO{r-1}{2r-1} \ar[r] \ar[d]_-{\ddp}& \; 0\\
& && \;0\ar[r] & \; \FO{l}{0} \ar[r]^-{L^r}\ar[d]_-{\ddp}& \; \FO{l+r}{2r} \ar[r]^-{\pil{r-1}}\ar[d]_-{\ddp} &\; \FO{r-1}{2r} \ar[r] \ar[d]_-{\ddp}& \; 0\\
& & & & \vdots \ar[d]_-{\ddp}& \vdots \ar[d]_-{\ddp}& \vdots \ar[d]_-{\ddp}& \\
& &&\;0\ar[r] & \; \FO{l}{n-r-1} \ar[r]^-{L^r}\ar[d]_-{\ddp}& \; \FO{l+r}{n+r-1} \ar[r]^-{\pil{r-1}} \ar[d]_-{\ddp}&\; \FO{r-1}{n+r-1} \ar[r] \ar[d]_-{\ddp}& \; 0\\
& &&\;0\ar[r] & \; \FO{l}{n-r} \ar[r]^-{L^r}\ar[d]_-{\ddp}& \; \FO{l+r}{n+r} \ar[r] \ar[d]_-{\ddp}&\; 0 & \\
& &\;0\ar[r]&\; \FO{r-1}{n+r-1} \ar[r]^-{\rstar} \ar[d]_-{\ddm}&  \; \FO{l}{n-r+1} \ar[r]^-{L^r}\ar[d]_-{\ddp}& \; \FO{l+r}{n+r+1} \ar[r] \ar[d]_-{\ddp}&\; 0 \\
& & & \vdots \ar[d]_-{\ddm} & \vdots \ar[d]_-{\ddp}& \vdots \ar[d]_-{\ddp}&  &\\
& &\;0\ar[r]&\; \FO{r-1}{n+r-l} \ar[r]^-{\rstar}\ar[d]_-{\ddm} &  \; \FO{l}{n-r+l} \ar[r]^-{L^r}\ar[d]_-{\ddp}& \; \FO{l+r}{n+r+l} \ar[r]&\; 0 \\
& &\;0\ar[r] & \FO{r-1}{n+r-l-1} \ar[r]^-{\rstar} \ar[d]_-{\ddm}& \; \FO{l}{n-r+l+1} \ar[r] \ar[d]_-{\ddp}&\; 0 & &\\
&\;0\ar[r] & \FO{l+r}{n+r+l} \ar[r]^{L^{-(l+1)}}\ar[d]_-{\ddm}& \; \FO{r-1}{n+r-l-2} \ar[r]^-{\rstar}\ar[d]_-{\ddm} & \; \FO{l}{n-r+l+2} \ar[r]\ar[d]_-{\ddp}&\; 0 & &\\
& & \vdots \ar[d]_-{\ddm}& \vdots \ar[d]_-{\ddm}& \vdots\ar[d]_-{\ddp} & & &\\
&\;0\ar[r] & \FO{l+r}{n+l+2} \ar[r]^-{L^{-(l+1)}}\ar[d]_-{\ddm}& \; \FO{r-1}{n-l} \ar[r]^-{\rstar}\ar[d]_-{\ddm} & \; \FO{l}{n+l} \ar[r]&\; 0 & &\\
&\;0\ar[r]&  \; \FO{l+r}{n+l+1} \ar[r]^{L^{-(l+1)}}\ar[d]_-{\ddm} & \; \FO{r-1}{n-l-1} \ar[r] \ar[d]_-{\ddm}&\; 0 & &\\
0\ar[r]&\; \FO{l}{n+l} \ar[r]^-{\iota}\ar[d]_-{\ddm} &  \; \FO{l+r}{n+l} \ar[r]^-{L^{-(l+1)}}\ar[d]_-{\ddm} & \; \FO{r-1}{n-l-2} \ar[r] \ar[d]_-{\ddm} &\; 0 & & &\\
& \vdots \ar[d]_-{\ddm} & \vdots \ar[d]_-{\ddm} & \vdots \ar[d]_-{\ddm} & & & & \\
0\ar[r]&\; \FO{l}{l+2} \ar[r]^-{\iota}\ar[d]_-{\ddm} &  \; \FO{l+r}{l+2} \ar[r]^-{L^{-(l+1)}}\ar[d]_-{\ddm} & \; \FO{r-1}{0} \ar[r] \ar[d]_{\ddm} &\; 0 & & &\\
0\ar[r]&\; \FO{l}{l+1} \ar[r]^-{\iota}\ar[d]_-{\ddm} &  \; \FO{l+r}{l+1} \ar[r] \ar[d]_-{\ddm} & \; 0 & & & &\\
& \; \vdots & \vdots & & & & & \\
}
\end{align*}
\end{lem}


\section{Filtered cohomologies}

\subsection{Elliptic complexes and associated cohomologies}

In Paper II \cite[Proposition 2.8]{TY2}, the following differential complex of primitive form was shown to be elliptic:
\begin{align} \label{elliptic_cpx1} \xymatrix{
0\; \ar[r] &\; \CB^0 \;\ar[r]^-\dpp &\; \CB^1 \;\ar[r]^-\dpp &~ \ldots ~\ar[r]^-\dpp &\; \CB^{n-1} \;\ar[r]^-\dpp &\; \CB^{n} \; \ar[d]^-{\dpp\dpm}\\
0\; &\; \CB^0 \;\ar[l] &\; \CB^1 \;\ar[l]_-\dpm &~ \ldots ~\ar[l]_-\dpm &\; \CB^{n-1} \;\ar[l]_-\dpm &\; \CB^{n} \ar[l]_-\dpm \;
} \end{align}
This complex was found in the four-dimensional case by Smith in 1976 \cite{Smith}.   In higher dimensions, besides \cite{TY2}, it was also independently found by Eastwood and Seshadri \cite{Eastwood1} (see also \cite{BryantE, Eastwood2}) who were motivated by the hyperelliptic complex of Rumin in contact geometry \cite{Rumin}. 

In the context of filtered forms, primitive forms correspond to $\FO{p}{k}$ with $p=0$, and therefore, we can rewrite the primitive elliptic complex equivalently as 
\begin{align*}\xymatrix{
0\; \ar[r] &\; \FO{0}{0} \;\ar[r]^-\ddp &\; \FO{0}{1} \;\ar[r]^-\ddp &~ \ldots ~\ar[r]^-\ddp &\; \FO{0}{n-1} \;\ar[r]^-\ddp &\; \FO{0}{n} \; \ar[d]^-{\dpp\dpm}\\
0\; &\; \FO{0}{0} \;\ar[l] &\; \FO{0}{1} \;\ar[l]_-\ddm &~ \ldots ~\ar[l]_-\ddm &\; \FO{0}{n-1} \;\ar[l]_-\ddm &\; \FO{0}{n} \ar[l]_-\ddm \;
} \end{align*}
Written in this form and with the introduction of more general $p$-filtered forms, it is then natural to consider complexes with higher filtration degree $p$ by replacing in the complex above $\FO{0}{*}$ with $\FO{p}{*}$.  Indeed, the resulting complexes are elliptic as well.

\begin{thm}\label{thrmecomplex}
The following differential complex is elliptic for $0\leq p\leq n$.
\begin{align} \label{fecomplex} \xymatrix{
0\; \ar[r] &\; \FO{p}{0} \; \ar[r]^-{\ddp} & \; \FO{p}{1} \; \ar[r]^-{\ddp} & ~ \ldots ~ \ar[r]^-{\ddp} & \; \FO{p}{n+p-1} \; \ar[r]^-{\ddp} & \; \FO{p}{n+p} \ar[d]^-{\dpp\dpm} \\
0\; &\; \FO{p}{0} \; \ar[l] & \; \FO{p}{1} \; \ar[l]_-{\ddm} & ~ \ldots ~ \ar[l]_-{\ddm} & \; \FO{p}{n+p-1} \; \ar[l]_-{\ddm} & \; \FO{p}{n+p} \ar[l]_-{\ddm}
} \end{align}
\end{thm}
\begin{proof}
Recall from the previous section that $(\ddp)^2= (\ddm)^2 =0$.  Moreover, it is straightforward to check that $(\dpp\dpm)\ddp=\ddm(\dpp\dpm)=0$ acting on any $p$-filtered form.  Hence, \eqref{fecomplex} is a differential complex.

To prove that the complex is elliptic, we need to show that the associated symbol complex is exact at each point $x\in M$.  Let $\xi\in T_x^*\backslash\{0\}$.  By an $\mathrm{Sp}(2n)$ transformation, we can set $\xi = e_1$ and take the symplectic form to be $\om = e_1\w e_2 + e_3\w e_4 + \ldots + e_{2n-1} \w e_{2n}$, where $e_1, \ldots, e_{2n}$ spans a basis for $T_x^*$.   Let $\eta_k\in F^p\Bw^kT_x^*$.  Then we can write
\begin{align}\label{sidec}\eta_k= \mu_k + \om \w \mu_{k-2} + \ldots + \om^p \w \mu_{k-2p}~,
\end{align}
where the $\mu$'s denote elements of the primitive exterior vector space, $\CB\Bw T_x^*$.  Each primitive vector can also be decomposed as \cite[Lemma 2.3]{TY2}
\begin{align} \label{pvde}
\mu_l  = e_1 \w \be_{l-1}^1 + e_2 \w \be_{l-1}^2 + e'_{12} \w \be_{l-2}^3 + \be_l^4 
\end{align}
where $\be^1,\be^2,\be^3,\be^4 \in\CB\Bw^*T_x^*$ are primitive exterior products involving only $e_3, e_4, \ldots, e_{2n}$, and
\begin{align*}
e'_{12} = e_1 \w e_2 - \frac{1}{H+1} \sum^n_{j=2} e_{2j-1}\w e_{2j}~.
\end{align*}
Here, $H$ is the operator defined by \eqref{def_sl2}.  In the following argument, the letter $\eta$ always means a $p$-filtered element, and $\mu$ always means a primitive element.  The symbol is denoted by $\sigma$.

We will show the exactness in four steps.

(1) Exactness of the symbol sequence corresponding to the top line  $0 \to \FO{p}{0} \to \ldots \to \FO{p}{n+p-1}$.  

Since $\ddp\ddp=0$, it is clear that $\im\si(\ddp)\subset \ker \si(\ddp)$.  We need to show that $\ker\si(\ddp) \subset \im \si(\dpp)$.  Now, $\si(\ddp) = \pil{p}(e_1\w \cdot)\,$.   So if $\eta_k\in \ker\si({\ddp})$, then either (i) $e_1 \w \eta_k =0$  or  (ii) $e_1 \w \eta_k = \om^{p+1} \w \mu_{k-2p-1}\neq0$.  In case (i), it follows by the exactness of the symbol complex associated with the de Rham complex that there exists an $\zeta_{k-1}\in \Bw^{k-1}T_x^*$ such that $\eta_k = e_1 \w \zeta_{k-1}$.  But since the operation $e_1 \w \cdot$ can only preserve the filtration degree $p$ or increases it by one, we conclude that $\eta_k=\si(\ddp)(\pip\zeta_{k-1})$.  In case (ii), $\eta_k$ must contain a nontrivial Lefschetz component $\om^p \w \mu_{k-2p}$ \eqref{sidec}.  By \eqref{pvde},
\begin{align*}
e_1 \w \mu_{k-2p} & = e_1 \w \left( e_1 \w \be^1_{k-2p-1} + e_2 \w \be^2_{k-2p-1} + e'_{12} \w \be^3_{k-2p-2} + \be^4_{k-2p}\right)\\
&= c_1\, e'_{12}\w \be^2_{k-2p-1} + c_2\, \om \w \be^2_{k-2p-1} + c_3 \,\om \w e_1 \w \be^3_{k-2p-2} + e_1 \w \be^4_{k-2p}
\end{align*}
for some non-zero constants $c_1, c_2$ and $c_3$.  This implies that $\mu_{k-2p}$ must have a nonzero $\be^2$ or $\be^3$ term.  However, a $\be^2$ term is not possible since the first term $e'_{12}\w \be^2$ can not be canceled to satisfy $\pil{p}(e_1\w\eta_k)=0$.  This is because  $e'_{12}\w \be^2 \notin \im \si(\dpm)$ \cite[(2.36)]{TY2} (or see \eqref{siim} below).  Hence, we only need to worry about the $e'_{12}\w\be^3$ term in $\mu_{k-2p}\,$, and express it as an element of $\im\si(\ddp)$.  To do so, we note that
\begin{align*}
e_1\w e_2 \w \be^3_{k-2p-2} = \dfrac{H+1}{H+2} \, e'_{12} \w \be^3_{k-2p-2} + \dfrac{1}{H+2} \,\om \w \be^3_{k-2p-2}~.
\end{align*}
Therefore, we can write $e'_{12} \w \be^3_{k-2p-2}  =  \frac{H+2}{H+1}\,  \pil{0} (e_1\w e_2 \w \be^3_{k-2p-2})=  \pil{0} \left\{e_1 \w \left[\frac{H+1}{H} (e_2 \w \be^3_{k-2p-2})\right]\right\}\,$.

\

(2) Exactness of the symbol sequence corresponding to the bottom line $0  \gets  \FO{p}{0} \gets \ldots \gets \FO{p}{n+p-1}$.

Note that the reflection of the de Rham complex by $\rstar\,$ gives an elliptic $\ddm$-complex, and therefore, $\ker\si(\ddm)= \im\si(\ddm)\,$ for the $\Bw T_x^*$ sequence.   Now for a filtered space, suppose that $\eta_k\in F^p\Bw^{k}T_x^*$ and $\si(\ddm) \eta_k =0$.   By the exactness of the $\ddm$-complex, $\eta_k = \si(\ddm) \xi_{k+1}$ for some $\xi_{k+1}\in \Bw^{k+1}T_x^*\,$.  It suffices to show that we can choose a $\xi_{k+1}$ such that $\xi_{k+1} \in F^p\Bw^{k}T_x^*\,$.

Consider the Lefschetz decomposition of $\eta_k$ as in \eqref{sidec}, and write 
\begin{align*}
\xi_{k+1} = \tmu_{k+1} + \ldots + \om^p \w \tmu_{k-2p+1} + \om^{p+1} \w \tmu_{k-2p-1} + \ldots ~,
\end{align*}
where $\tmu$'s are elements of the primitive vector space $\CB\Bw^{*}T_x^*$.  Using \eqref{dmsimp}, we have
\begin{align*}
\eta_k &= \si(\ddm)\xi_{k+1} \\
&= \si(\ddm)\left(\tmu_{k+1} + \ldots + \om^{p-1}\w \tmu_{k-2p+3}\right) +  \si(\dpp) (\om^{p-1} \w \tmu_{k-2p+1}) 
\\ &~~~~~~~~ + \si(\dpm)(\om^p \w \tmu_{k-2p+1}) + \si(\dpp)(\om^{p} \w \tmu_{k-2p-1})~. 
\end{align*}
Since $\si(\dppm): \CB\Bw^{k}T_x^*\to\CB\Bw^{k\pm 1}T_x^*\,$, it implies for the $\om^p$ term that
\begin{align*}
\mu_{k-2p} = \si(\dpm) \tmu_{k-2p+1} + \si(\dpp) \tmu_{k-2p-1}~.
\end{align*}  
Now the condition $\si(\ddm) \eta_k =0$ requires  $\si(\dpm) \mu_{k-2p} =0$.  Hence, in the decomposition \eqref{pvde} for $\mu_{k-2p}$, there are only two non-zero terms,
\begin{align}\label{mutwo}
\mu_{k-2p} = e_1 \w \be^1_{k-2p-1} + \be^4_{k-2p}~.
\end{align}
as $\{e_2\w \be^2, e'_{12}\w \be^3\} \notin \ker \si(\dpm)$ \cite[(2.39)]{TY2}.
Let us also recall from \cite[(2.35) and (2.36)]{TY2} that acting on the decomposition \eqref{pvde}, we have
\begin{align}
\im \si(\dpp) &= \left\{e'_{12} \w \be^2\,, e_1 \w \be^4\right\} ~,\label{siim1}\\
\im\si(\dpm) &= \left\{ \be^2\,, e_1 \w \be^3\right\}~. \label{siim}
\end{align} 
Hence, it is clear that if $\si(\dpp) \tmu_{k-2p-1}$ is non-zero, then $\si(\dpp)\tmu_{k-2p-1} =  e_1 \w \be'_{k-2p-1}$ for some primitive element $\be'$ independent of $e_1$ and $e_2$.  But then by \eqref{siim}, we can write
\begin{align*}
\si(\dpp)\tmu_{k-2p-1} = c\, \si(\dpm) e'_{12} \w \be'_{k-2p-1}
\end{align*}
where $c$ is some non-zero constant.   Letting $\mu'_{k-2p+1} = c\, e'_{12} \w \be'_{k-2p-1}\,$ and noting that $\si(\dpp) \mu'_{k-2p+1} =\si(\dpp)(c\, e'_{12} \w \be'_{k-2p-1}) =0$,  we obtain the desired result
\begin{align*}
\eta_k  = \si(\ddm)\left[\tmu_{k+1} + \ldots + \om^{p-1}\w \tmu_{k-2p+3} + \om^p \w (\tmu_{k-2p+1} + \mu'_{k-2p+1}) \right]~.
\end{align*}

\

(3) Exactness of the symbol sequence corresponding to $ \FO{p}{n+p-1} \to \FO{p}{n+p} \to \FO{p}{n+p}$. 

By \eqref{kpmo} and \eqref{kpnp},  we need to show that
\begin{align*}\xymatrix{
&\; \CB\!\Bw^{n-p+1}T_x^* \oplus \CB\!\Bw^{n-p-1}T_x^* ~\ar[r]^{~~~~~~~~~~~\si(\ddp)}   &\;\CB\!\Bw^{n+p}T_x^* ~\ar[r]^{\si(\dpp\dpm)} & \; \CB\!\Bw^{n+p}T_x^* &\;}
\end{align*}
is exact at the middle.  In terms of the decomposition of \eqref{pvde}, it is straightforward to check that 
\begin{align*}
\ker \si (\dpp\dpm) = \im \si(\ddp) = \left\{ e_1 \w \be^1_{n-p-1}\, , e'_{12} \w \be^3_{n-p-2}\, , \be^4_{n-p}\right\}
\end{align*}

\

(4) Exactness of the symbol sequence corresponding to $\FO{p}{n+p} \to \FO{p}{n+p} \to \FO{p}{n+p-1}$.

Here we need to show that the symbol sequence
\begin{align*}\xymatrix{
&\;\CB\!\Bw^{n+p}T_x^* ~\ar[r]^{\si(\dpp\dpm)} & \; \CB\!\Bw^{n+p}T_x^*~ \ar[r]^{\!\!\!\!\!\!\!\!\!\!\!\!\!\!\!\!\!\!\!\!\!\!\!\!\!\!\!\!\!\!\!\si(\ddm)}   &\; \CB\!\Bw^{n-p+1}T_x^* \oplus \CB\!\Bw^{n-p-1}T_x^* }
\end{align*}
is exact at the middle.  In terms of the decomposition of \eqref{pvde}, it is clear that
\begin{align*}
\ker \si(\ddm) = \im \si(\dpp\dpm) =  \left\{ e_1 \w \be^1_{n-p-1}\right\}~.
\end{align*}
\end{proof}

Having established the ellipticity of the complex \eqref{fecomplex}, we have also shown the finite-dimensionality of the associated {\it filtered} cohomologies which we shall denote by
\begin{align}\label{fecomplexcoh}
F^pH
&=\left\{\FHP{p}{0}, \FHP{p}{1}, \ldots, \FHP{p}{n+p},\, \FHM{p}{n+p}, \ldots, \FHM{p}{1}, \FHM{p}{0}\right\}
\end{align}
where 
$$\FHP{p}{k} = \frac{\ker(\ddp)\cap \FO{p}{k}}{\ddp(\FO{p}{k-1})}~, \qquad
\FHM{p}{k} = \frac{\ker(\ddm)\cap \FO{p}{k}}{\ddm(\FO{p}{k+1})}~,$$
for $k=0,1,\cdots,n+p-1\,$ and
$$\FHP{p}{n+p} = \frac{\ker(\dpp\dpm)\cap \FO{p}{n+p}}{\ddp(\FO{p}{n+p-1})}~,\qquad
\FHM{p}{n+p} = \frac{\ker(\ddm)\cap \FO{p}{n+p}}{\dpp\dpm(\FO{p}{n+p})}~.$$

Let us make several comments concerning these filtered cohomologies.  First, modulo powers of $L$, we can make the identification: 
$$ F^p\CA^{n+p-1} \cong~ \CB^{n-p+1}\,\oplus\,\CB^{n-p-1}~.$$   For the middle of the elliptic complex \eqref{fecomplex}, such an identification translates into 
\begin{align*} \xymatrix{
\cdots \;\ar[r] &\; \CB^{n-p+1}\oplus\CB^{n-p-1} \;\ar[r]^-{\dpm\,+\,\dpp} &\; \CB^{n-p} \;\ar[r]^-{\dpp\dpm} &\; \CB^{n-p} \;\ar[r]^-{\dpp\,\oplus\,\dpm} &\; \CB^{n-p+1}\oplus\CB^{n-p-1} \;\ar[r] &\; \cdots ~ .
} \end{align*}
Thus, the middle two cohomologies of \eqref{fecomplexcoh} are equivalent to $\PHPR{n-p}(M)$ and $\PHPL{n-p}(M)$ introduced in \cite{TY2}.  Specifically,
\begin{align}\label{fphm}
\FHP{p}{n+p}(M)\cong \PHPR{n-p}(M) &= \frac{\ker \dpp\dpm \cap \CB^{n-p}(M)}{\dpp\CB^{n-p-1} + \dpm \CB^{n+p+1}}~,\\
\FHM{p}{n+p}(M)\cong \PHPL{n-p}(M) &=  \frac{\ker(\dpp+\dpm) \cap \CB^{n-p}(M)}{\dpp\dpm \CB^{n-p}(M)} ~.\label{fphmm}
\end{align}
Second, since $\FO{p}{k}=\CA^k$ for $k \leq 2p+1$ as noted in \eqref{klep}, the section of the elliptic complex consisting of the first $2p+1$ elements of the top line of \eqref{fecomplex} is effectively equivalent to the usual de Rham complex.   Similarly, the section of the bottom line involving the last  $2p+1$ elements is equivalent the $\rstar$ dual of the de Rham complex.   Thus, we have the following relations:
\begin{align} 
\FHP{p}{k}(M) &= \HD{k}(M)~, \qquad\qquad\qquad {\rm for~}0\leq k\leq 2p ~, \label{federham1}\\ 
\FHM{p}{k}(M) & \cong \HD{2n-k}(M)~, \qquad\qquad~\, {\rm for~}0\leq k\leq 2p~. \label{fderham2}
\end{align} 

Lastly, since the filtered cohomologies are associated with elliptic complexes, we can write down an elliptic laplacian for each filtered cohomology.  Note that the laplacians associated with the cohomologies $\FHP{p}{n+p}$ \eqref{fphm} and $\FHM{p}{n+p}$ \eqref{fphmm} are of fourth-order.  But since each laplacian is elliptic, we can nevertheless associate a Hodge theory to each cohomology.  That is, with the introduction of a Riemannian metric, we can define a unique harmonic representative for each cohomology class and Hodge decompose any form into three orthogonal components consisting of harmonic, exact and co-exact forms.  An expanded discussion of the Hodge theoretical properties for those filtered cohomologies that are primitive (i.e. $p=0$ or $k=n+p$) can be found in \cite{TY1,TY2}.

\subsection{Local Poincar\'e lemmata}
We now consider the above cohomologies for an open unit disk $U$ in $\BR^{2n}$ with the standard symplectic form $\omega = \sum_{i=1}^n dx^i\wedge dx^{n+i}$.  The primitive cohomologies $PH_{d\dl}^k(U)$ and $PH_{d+\dl}^k(U)$ have been calculated by \cite[Proposition 3.12 and Corollary 3.11]{TY2}:
\begin{align*}
\dim \PHPR{k}(U) &= \begin{cases}
1 &\text{ when }k=1 ~, \\
0 &\text{ otherwise } ;
\end{cases}
&\dim \PHPL{k}(U) &= \begin{cases}
1 &\text{ when }k=0 ~, \\
0 &\text{ otherwise } .
\end{cases}
\end{align*}

\begin{prop}[$\ddp$-Poincar\'e lemma]\label{Poincare_disk1}
Let $U$ be an open unit disk in $\BR^{2n}$ with the standard symplectic form $\omega = \sum dx^i\wedge dx^{n+i}$.  Then for $0\leq p< n$,
$$ \dim \FHP{p}{0}(U) = \dim \FHP{p}{2p+1}(U) = 1 ~, $$
and $\dim \FHP{p}{k}(U) = 0$ for $1\leq k\leq n+p-1$ and $k\neq 2p+1$.
\end{prop}

\begin{proof}
When $0\leq k\leq 2p$, the cohomology $\FHP{p}{k}(U) = \HD{k}(U)$.  When $k\geq2p+1$, any element $A_k\in\FO{p}{k}$ has the Lefschetz decomposition $\sum_{s=0}^{\min(p,n+p-k)}L^{p-s}B_{k-2p+2s}$ for $B_{k-2p+2s}\in\CB^{k-2p+2s}$.  If $A_k$ is $\ddp$-closed, either (1) $dA_k=0$ or (2) $\ddp A_k=0$ but $dA_k = L^{p+1}B'_{k-2p-1} \neq0$ for some $B'_{k-2p-1}\in\CB^{k-2p-1}$.

Case (1):  The standard Poincar\'e lemma implies that $A_k = dA'_{k-1}$ for some $A'_{k-1}\in\CA^{k-1}$.   Let $A''_{k-1}= \pip A'_{k-1}$.  After taking $\pil{p}\circ d$ and using \eqref{relation1}, we find that $\ddp A''_{k-1} = A_{k}$.

Case (2a):  Let $2p+1 < k < n+p$.  Since $d^2 A_k = L^{p+1} dB'_{k-2p-1} = 0$ and $L^{p+1}$ is not zero on $\CB^{k-p}$, we have $dB'_{k-2p-1} = 0$.  It follows from the primitive Poincar\'e lemma \cite[Proposition 3.10]{TY2} that there exists a $B''_{k-2p-1}\in\CB^{k-2p-1}$ such that $\dpp\dpm B''_{k-2p-1} = B'_{k-2p-1}$.  
Note that $L^{p+1}\dpm B''_{k-2p-1}\notin\FO{p}{k}$ and
$$ d( A_k - L^{p+1}\dpm B''_{k-2p-1} ) = L^{p+1}B'_{k-2p-1} - L^{p+1}\dpp\dpm B''_{k-2p-1} = 0 ~. $$
The standard Poincar\'e lemma implies that $A_k - L^{p+1}\dpm B''_{k-2p-1} = dA'_{k-1}$ for some $A'_{k-1}\in\CA^{k-1}$.  Now let $A''_{k-1}=\pip A'_{k-1}$.  Then similar to case (1), we have  $\ddp A''_{k-1} = A_k$.

Case (2b):  Let $k = 2p+1$.  Since $dA_{2p+1} = L^{p+1}B'_0 \neq0$ and $d^2A_{2p+1} = L^{p+1} dB'_0 = 0$, $B'_0$ must be a nonzero constant function.  Since $(\FO{p}{2p}, \ddp) = (\CA^{2p},d)$, such $A_{2p+1}$ does not belong to $\ddp(\FO{p}{2p})$.  Because $B'_0$ is a constant, the argument of case (1) implies that $\dim \FHP{p}{2p+1}(U)\leq1$.  We finish the proof by taking $A_{2p+1} = L^{p}(-\sum x^{n+i} dx^i)$.
\end{proof}

\begin{prop}[$\ddm$-Poincar\'e lemma]
Let $U$ be an open unit disk in $\BR^{2n}$ with the standard symplectic form $\omega = \sum dx^i\wedge dx^{n+i}$.  Then for $0< p< n$ and $0\leq k\leq n+p-1$,
$$ \dim \FHM{p}{k}(U) = 0 ~. $$
\end{prop}

\begin{proof}
When $0\leq k\leq2p$, it follows from \eqref{fderham2} that $\dim \FHM{p}{k}(U) =\dim \HD{2n-k}(U) =0$.  When $2p<k < n+p-1$, the $\rstar$-dual of the standard Poincar\'e lemma implies that any $\ddm$-closed $A_k\in\FO{p}{k}$ is equal to $\ddm A'_{k+1}$ for some $A'_{k+1}\in\CA^{k+1}$.  Let $A''_{k+1} = \pip A'_{k+1}$.  Then the difference between $A_k$ and $\ddm A''_{k+1}$ can be expressed as $A_k - \ddm A''_{k+1} = L^pB'_{k-2p}$ for some $B'_{k-2p}\in\CB^{k-2p}$.

Now we have $\ddm( A_k - \ddm A''_{k-1}) = \ddm (L^pB'_{k-2p}) = 0$.  By \eqref{dmsimp}, this implies $\dpp B'_{k-2p} = \dpm B'_{k-2p} = 0$, and equivalently $d B'_{k-2p}=0$.  Since $k-2p>0$, the primitive $d\dl$-Poincar\'e lemma \cite[Proposition 3.10]{TY2} says that there exists a $B''_{k-2p}\in\CB^{k-2p}$ such that $B'_{k-2p}=\dpp\dpm B''_{k-2p}$.  Therefore, we have $L^pB'_{k-2p} = \ddm (L^p \dpp B''_{k-2p})$ and  $A_k = \ddm(A''_{k+1}  - L^p \dpp B''_{k-2p})$.
\end{proof}

Let us note that the above proof does not work for $p=0$, which is the primitive $\dpm$-Poincar\'e lemma \cite[Proposition 3.14]{TY2}.  The argument fails because we cannot conclude that $dB'_{k-2p} = 0$ when $p=0$.  Moreover, when $p=n$, the elliptic complex (\ref{fecomplex}) simply consists of two de Rham complexes.

\begin{cor}\label{uindex}
Let $U$ be an open unit disk in $\BR^{2n}$ with the standard symplectic form $\omega = \sum dx^i\wedge dx^{n+i}$.  For $0\leq p\leq n$, the index of the elliptic complex (\ref{fecomplex}) is zero.
\end{cor}

\section{Filtered cohomologies and Lefschetz maps}\label{sec_Lef}
Let $(M,\omega)$ be a compact symplectic manifold of dimension $2n$.  Recall that the \emph{strong Lefschetz property} means that the map
\begin{align*}
L^{k}: \HD{n-k}(M) \to \HD{n+k}(M)
\end{align*}
is an isomorphism for all $k\in\{0,1,\ldots,n\}$.  It is known that the strong Lefschetz property is equivalent to what we call the $d\dl$-lemma \cite{Merkulov, Guillemin, TY1}.  In general, the strong Lefschetz property does not hold for a non-K\"ahler symplectic manifold.  

We would like to analyze the kernel and cokernel of an arbitrary Lefschetz map, $L^r$.   Certain aspects of Lefschetz maps have appeared in the literature previously.  In four dimensions, Baldridge and Li \cite{BL} identified the symplectic invariant $\ker[L:\HD{1} \to \HD{3}]$ and called it the degeneracy.    Lefschetz maps in higher dimensions were also discussed for instance in \cite{FId, Iban}.

The commutative diagram of short exact sequences of Lemma \ref{ccomplex1} and Lemma \ref{ccomplex2} is suggestive of a long exact sequence involving Lefschetz maps.  However, the main challenge and novelty remains with the shifts in the diagram.
In order to maintain a continuous long exact sequence and also take into account of the shift, cohomologies involving 2nd-order differential operators must be introduced.   In this regard, these shifts provide a natural explanation for why cohomologies like $\PHPR{}$  and $\PHPL{}$ involving $\dpp\dpm$ operators are natural for symplectic manifolds.

\subsection{Long exact sequences}

In the following proposition, we explain how to treat the shift in the commutative diagrams of Lemma \ref{ccomplex1} and Lemma \ref{ccomplex2}

\begin{prop}\label{etriple}
Given the cochain complexes
\begin{align*}
\xymatrix
@R-10pt
@H+7pt
{
& & \vdots \ar[d]_-\pad & \vdots  \ar[d]_-\pae & \vdots \ar[d]_-{\paf} & \\
&\; 0 \ar[r] &\; D^{l-2}\ar[r]^-\phi \ar[d]_-\pad & \; E^{l-2} \ar[r]^-\psi \ar[d]_-\pae &\;F^{l-2} \ar[r] \ar[d]_-{\paf} &\;0\\
&\; 0 \ar[r] &\; D^{l-1}\ar[r]^-\phi \ar[d]_-\pad & \; E^{l-1} \ar[r]^-\psi \ar[d]_-\pae &\; F^{l-1} \ar[r] &\; 0\\
&\; 0\ar[r] &\; D^{l} \ar[r]^-\phi \ar[d]_-\pad & \; E^{l} \ar[r] \ar[d]_-\pae &\; 0 &\\
0\ar[r] &\; C^{l+1} \ar[r]^-{\rho} \ar[d]_-{\pac}&\; D^{l+1} \ar[r]^-{\phi} \ar[d]_-\pad &\; D^{l+1}\ar[r] \ar[d]_-\pae &\; 0 &  \\
0\ar[r] &\; C^{l+2} \ar[r]^-\rho \ar[d]_-{\pac}  &\; D^{l+2} \ar[r]^-\phi \ar[d]_-\pad  &\; E^{l+2}\ar[r] \ar[d]_-\pae   &\; 0 &  \\
& \vdots & \vdots  & \vdots &  & \\
}
\end{align*}
such that
$$ \rho\, \phi =0~,\qquad\qquad \phi\,\psi=0~,$$
$$\rho \, \pac = \pad\, \rho~, \qquad \phi\,\pad=\pae\,\phi~, \qquad \psi\,\pae = \paf\,\psi~,$$
there is a long exact sequence of cohomology
\begin{align*} 
\xymatrix
@R-10pt
@H+18pt
{
&\; \ldots \ar[r] &\;H^{l-1}(D) \ar[r]^\phis  &\;H^{l-1}(E) \ar[r]^\psis  &\; H^{l-1}(F) \ar`r[d]`d[l]`l^d[dlll]`d^r[dll][dll]^-{\paes} \\
& &\;H^{l}(D) \ar[r]^\phis  &\;H^{l}(E) \ar`r[d]`d[l]`l^d[dlll]`d^r[dll][dll]^-{\pads} 
&\\
&\;H^{l+1}(C) \ar[r]^\rhos  &\;H^{l+1}(D) \ar[r]^\phis  &\;H^{l+1}(E) \ar[r] & \; \ldots 
\\
} 
\end{align*}
where $\pads$ and $\paes$ are induced by the derivative operators $\pad$ and $\pae$, respectively, and except for the two cohomologies $H^{l-1}(F)$ and $H^{l+1}(C)$ which are defined as follows  
\begin{align*}
H^{l-1}(F) &= \frac{\ker (\pad\paes) \cap F^{l-1}}{\im \paf \cap F^{l-1}}~, \\
H^{l+1}(C) & = \frac{\ker \pac \cap C^{l+1}}{\im (\pads \pae) \cap C^{l+1}}~, 
\end{align*}
the other cohomologies are standardly defined, for instance,
$$H^*(D) = \frac{\ker \pad}{\im \pad}~.$$
\end{prop}
\begin{proof}
We first define the operators:

(1) Definition of $\pads$.  Let $e_l \in H^l(E)$.  Choose a $d_l \in D^l$ such that $\phi(d_l) = e_l$.  Then there exists $c_{l+1}\in C^{l+1}$ such that $\rho(c_{l+1})=\pad d_l$.   We therefore define
$$\pads[e_l] = [c_{l+1}]~.$$

That $\pads$ defines a homomorphism should be self-evident.  Let us show though that $\pads$ is well-defined.  That is, we want to show that if $e_l$ and $e'_l$ are cohomologous in $H^l(E)$, then the corresponding $c_{l+1}$ and $c'_{l+1}$ are also cohomologous in $H^{l+1}(C)$.  Here, we will see that the non-standard definition of $H^{l+1}(C)$ becomes important. 

Since, $e_l$ and $e'_l$ are cohomologous, we can write
$$ e_l = e'_l + \pae \, e_{l-1}$$
for some $e_{l-1}\in E^{l-1}$.  Note in general, $\psi(e_{l-1}) \neq 0$.  Now by surjectivity, there exist $d_l, d'_l, {\tilde d}_l\in D^l$ such that $\phi(d_l) = \phi(d'_l) + \pae\, e_{l-1}$ and therefore,
$$\phi(d_l - d'_l) = \pae\, e_{l-1} = \phi \,{\tilde d}_l~.$$
Clearly, $\pad(d_l - d'_l) = \pad\, {\tilde d}_l$.  With $\pad\,d_l = \rho \, c_{l+1}$ and $ \pad\,d'_l = \rho \, c'_{l+1}$, we then have
\begin{align*}
\rho(c_{l+1} - c'_{l+1}) & = \pad {\tilde d}_l = \pad (\phi^{-1} \pae {\tilde e}_{l-1})\\
& =\rho ( \rho^{-1} \pad \phi^{-1} \pae {\tilde e}_{l-1})
\end{align*}
By the injectivity of $\rho$, this shows that $c_{l+1}$ and $c'_{l+1}$ are cohomologous.

(2) Definition of $\paes$.  Let $f_{l-1}\in H^{l-1}(F)$.  Choose an $e_{l-1}\in E^{l-1}$ such that $\psi(e_{l-1})=f_{l-1}$.  Then there exists an $d_l\in D^l$ such that $\phi(d_l) = \pae\, e_{l-1}$.  We therefore define 
$$\paes[f_{l-1}]=[d_l]~.$$
It follows from standard arguments that $\paes$ is a well-defined homomorphism.

(3) Definition of $H^{l+1}(C)$ and $H^{l-1}(F)$.  Let us show that both 
(a) $\im (\pads \pae) \cap C^{l+1}$ and (b) $\ker (\pad\paes) \cap F^{l-1}$ are well-defined.  To show this, it is important that the map $\phi$ at degree $l$, $\phi: D^l \to E^l$, is bijective, and thus $\phi^{-1}$ is well-defined.  For (a), notice here that $\pads=\rho^{-1}\pad\phi^{-1}: E^l \to C^{l+1}$ is well-defined only if $e_l\in E^l$ is $\pae$-closed.  Hence, $\pads\pae$ is well-defined.  For (b), $\paes = \phi^{-1}\pae \psi^{-1}: F^{l-1} \to D^l$ is only defined up to a $\pad$-exact term.  Hence, $\pad\paes$ is well-defined.

Proving the exactness of the cohomology sequence follows the standard diagram-chasing arguments.  Indeed, all standard arguments can be applied to this case with the exception of the exactness at $H^l(E)$, which we will give a proof here.

Firstly, at $H^l(E)$, it is clear that $\im \subset \ker$ since $\pads \, \phis = 0$.  So we need to show also that $\ker \subset \im$.  So consider the case when $e_l \in H^l(E)$ maps to the trivial element, i.e., $\pads\, e_l = \rho^{-1} \pad \phi^{-1} \pae\, e_{l-1}= [0] \in H^{l+1}(C)$.  In this case, there exist a $d_l \in D^l$ and a $c_{l+1}\in C^{l+1}$ such that $\phi(d_l) = e_l$ and $\rho(c_{l+1}) = \pad\, d_l$.  Then it is clear that $\pad (d_l - \phi^{-1}\pae\, e_{l-1})=0$, and hence, $(d_l - \phi^{-1}\pae\, e_{l-1})$ is an element of $H^l(D)$.  Moreover, we have
\begin{align*}
\phis[d_l - \phi^{-1}\pae\, e_{l-1}] = [e_l - \pae\, e_{l-1}] = [e_l]~.
\end{align*}
This completes the proof of the proposition.
\end{proof}

\subsection{Resolution of Lefschetz maps}

With the chain of short exact sequences of Lemma \ref{ccomplex1} and now Proposition \ref{etriple}, we obtain the following long exact sequence relating filtered cohomologies and Lefschetz maps.
\begin{thm}\label{resolution}
Let $(M,\omega)$ be a symplectic manifold of dimension $2n$, which needs not be compact.  Then, the following sequence is exact for any $1\leq r\leq n$:
\begin{align*}
\xymatrix
@C+20pt
{
&\; 0 \ar[r] &\; \HD{2r-1}(M) \ar[r]^-{\pil{r-1}} & \;\FHP{r-1}{2r-1}(M) \ar`[d]`[l]`[dlll]`[dll][dll]^-{L^{-r}d}
\\
&\;\HD{0}(M) \ar[r]^-{L^{r}}  &\;\HD{2r}(M) \ar[r]^-{\pil{r-1}}  &\;\FHP{r-1}{2r}(M) \ar@{-->}`[d]`[l]`[dlll]`[dll][dll]^-{L^{-r}d}\\
&\;\HD{n-r-2}(M) \ar[r]^-{L^{r}}  &\;\HD{n+r-2}(M) \ar[r]^-{\pil{r-1}}  &\;\FHP{r-1}{n+r-2}(M) 
 \ar`[d]`[l]`[dlll]`[dll][dll]^-{L^{-r}d} \\
&\;\HD{n-r-1}(M) \ar[r]^-{L^{r}}  &\;\HD{n+r-1}(M) \ar[r]^-{\pil{r-1}}  &\;\FHP{r-1}{n+r-1}(M) \ar`[d]`[l]`[dlll]`[dll][dll]^-{L^{-r}d}\\
 &\;\HD{n-r}(M) \ar[r]^-{L^{r}}  &\;\HD{n+r}(M) \ar[r]^-{\pil{r-1}\rstar d\, L^{-r}} &\;\FHM{r-1}{n+r-1}(M)\ar`[d]`[l]`[dlll]`[dll][dll]^-{\rstar}\\
&\;\HD{n-r+1}(M) \ar[r]^-{L^{r}}  &\;\HD{n+r+1}(M) \ar[r]^-{\pil{r-1}\rstar d\, L^{-r}} &\; \FHM{r-1}{n+r-2}(M)\ar`[d]`[l]`[dlll]`[dll][dll]^-\rstar \\
&\;\HD{n-r+2}(M) \ar[r]^-{L^{r}}  &\;\HD{n+r+2}(M) \ar[r]^-{\pil{r-1}\rstar d\, L^{-r}} &\; \FHM{r-1}{n+r-3}(M) \ar@{-->}`[d]`[l]`[dlll]`[dll][dll]^-\rstar \\
&\;\HD{2n-2r}(M) \ar[r]^-{L^{r}}  &\;\HD{2n}(M)\ar[r]^-{\pil{r-1}\rstar d\, L^{-r}} &\; \FHM{r-1}{2r-1}(M) \ar`[d]`[l]`[dlll]`[dll][dll]^-\rstar\\
&\;\HD{2n-2r+1}(M) \ar[r] &\; 0 &\\
}\end{align*}

In other words, the $(r-1)$-filtered cohomologies give a resolution of the Lefschetz map $L^{r}$.
\end{thm}
\begin{proof}
The theorem follows from the short exact sequences of Lemma \ref{ccomplex1} and Theorem \ref{etriple} with the following identifications
$$\rho = \rstar~, \qquad \phi = L^r~, \qquad \psi=\pil{r-1}~,$$
and
 $$ \pac = \ddm~,\qquad \pad=\pae=d~,\qquad \paf = \ddp~.$$
\begin{align} \label{longexactseq}\xymatrix{
&    &\; 0 \ar[r] &\; \HD{2r-1}(M) \ar[r]^-{\pil{r-1}}  &\;\FHP{r-1}{2r-1}(M)\ar`[d]`[l]`[dlll]`[dll][dll]^-{L^{-r}d}  \\
&  &\;\HD{0}(M) \ar[r]^-{L^{r}}  &\;\HD{2r}(M) \ar[r]^-{\pil{r-1}}  &\;\FHP{r-1}{2r}(M) \ar@{-->}`[d]`[l]`[dlll]`[dll][dll]^-{L^{-r}d}\\
& &\;\HD{n-r-2}(M) \ar[r]^-{L^{r}}  &\;\HD{n+r-2}(M) \ar[r]^-{\pil{r-1}}  &\;\FHP{r-1}{n+r-2}(M) \ar`[d]`[l]`[dlll]`[dll][dll]^-{L^{-r}d} \\
& &\;\HD{n-r-1}(M) \ar[r]^-{L^{r}}  &\;\HD{n+r-1}(M) \ar[r]^-{\pil{r-1}}  &\;L^{r-1}\PHPR{n-r+1}(M) \ar`[d]`[l]`[dlll]`[dll][dll]^-{L^{-r}d} \\
& &\;\HD{n-r}(M) \ar[r]^-{L^{r}}  &\;\HD{n+r}(M) \ar`[d]`[l]^-{\pil{r-1} \rstar d\, L^{-r}}`[dlll]`[dll][dll] &\\
&\;L^{r-1}\PHPL{n-r+1}(M) \ar[r]^-\rstar  &\;\HD{n-r+1}(M) \ar[r]^-{L^{r}}  &\;\HD{n+r+1}(M) \ar`[d]`[l]^-{\pil{r-1} \rstar d\, L^{-r}}`[dlll]`[dll][dll] &\\
&\;\FHM{r-1}{n+r-2}(M) \ar[r]^-\rstar  &\;\HD{n-r+2}(M) \ar[r]^-{L^{r}}  &\;\HD{n+r+2}(M) \ar@{-->}`[d]`[l]^-{\pil{r-1} \rstar d\, L^{-r}}`[dlll]`[dll][dll] &\\
&\;\FHM{r-1}{2r}(M) \ar[r]^-\rstar  &\;\HD{2n-2r}(M) \ar[r]^-{L^{r}}  &\;\HD{2n}(M)\ar`[d]`[l]^-{\pil{r-1} \rstar d\, L^{-r}}`[dlll]`[dll][dll]& \\
&\;\FHM{r-1}{2r-1}(M) \ar[r]^-\rstar  &\;\HD{2n-2r+1}(M) \ar[r] & \; 0 &
} \end{align}
The long exact sequence is obtained noting that  $\FO{r-1}{n+r-1} = L^{r-1}\CB^{n-r+1}\cong\CB^{n-r+1}$ and also \eqref{fphm}--\eqref{fphmm}
\begin{align*}
\FHP{r-1}{n+r-1}(M)= L^{r-1}\PHPL{n-r+1}(M)~, \qquad \FHM{r-1}{n+r-1}(M)= L^{k-1}\PHPR{n-r+1}(M)~.
\end{align*}
This completes the proof of the theorem.
\end{proof}

We emphasize that the above theorem with its long exact sequence follows directly from the chain of short exact sequences which are all algebraic in nature.  Therefore, the theorem certainly holds true for differential forms of any type of support, e.g. compact or $L^2$, and for both closed and open symplectic manifolds.  Furthermore, as described in the Introduction, Theorem \ref{resolution} can be expressed very concisely in terms of the following exact triangle:
\begin{align}
\xymatrix{
 &\; F^{r-1}H^*(M) \ar[ld] &  \\
 \HD{*}(M) \ar[rr]^-{L^r} & & \HD{*}(M)\ar[lu]}\label{Tpth}
\end{align}
Here, $F^{r-1}H^*(M)$ represents exactly the filtered cohomologies in \eqref{fecomplexcoh} associated with the filtered elliptic complex \eqref{fecomplex} with $p=r-1$.

Now we have obtained the exact triangle \eqref{Tpth} starting from the chain of short exact sequences in Lemma \ref{ccomplex1}.  In fact, we have written down another chain of short exact sequences consisting of purely filtered forms in Lemma \ref{ccomplex2}.  Thus, we can also use Proposition \ref{etriple} to derive another long exact sequence involving only filtered cohomologies with Lefschetz type actions.   Instead of writing out explicitly the long exact sequence, we will just write down the resulting exact triangle:
\begin{align}
\xymatrix{
 &\; F^{r-1}H^*(M) \ar[ld] &  \\
 F^{l}H^*(M) \ar[rr]^-{h} & & F^{l+r}H^*(M)\ar[lu]}\label{Tpff}
\end{align}
where the map $h$ can be read off from Lemma \ref{ccomplex2} and is either $L^r$ or the inclusion map $\iota$.  Notice that $F^lH(M)$ when $l\geq n$ consists roughly of two copies of the de Rham cohomology $\HD{}(M)$.  Hence, the exact triangle of \eqref{Tpth} can be easily seen to be contained in \eqref{Tpff} when $l=n$.

\subsection{Properties of cohomologies}

Let us consider some of the implications of Theorem \ref{resolution} for the cohomologies.  We note first some immediate corollaries.

\begin{cor}
Let $(M,\omega)$ be a symplectic manifold of dimension $2n$.  Then for $k\leq n$, 
\begin{align}\label{res1}\begin{split}
\PHPR{k}(M) & \cong \ker(L^{n-k+1}: \HD{k-1}\to \HD{2n-k+1}) \oplus \cok(L^{n-k+1}:\HD{k-2} \to \HD{2n-k})~,\\
\PHPL{k}(M) & \cong \ker(L^{n-k+1}: \HD{k} \to \HD{2n-k+2}) \oplus \cok(L^{n-k+1}: \HD{k-1} \to \HD{2n-k+1})~,
\end{split}\end{align}
and for $2p<k < n+p$,
\begin{align}\label{res2}\begin{split}
\FHP{p}{k}(M) & \cong \ker(L^{p+1}: \HD{k-2p-1} \to \HD{k+1}) \oplus \cok(L^{p+1}: \HD{k-2p-2} \to \HD{k})~,\\
\FHM{p}{k}(M) & \cong \ker(L^{p+1}: \HD{2n-k} \to \HD{2n-k+2p+2}) \oplus \cok(L^{p+1}: \HD{2n-k-1} \to \HD{2n+2p-k+1})~.
\end{split}
\end{align}
In particular, when $p=0$, we have
\begin{align}\label{res3}\begin{split}
\PHDP{k}(M) & \cong \ker(L: \HD{k-1} \to \HD{k+1}) \oplus \cok(L: \HD{k-2} \to \HD{k})~,\\
\PHDM{k}(M) & \cong \ker(L: \HD{2n-k} \to \HD{2n-k+2}) \oplus \cok(L: \HD{2n-k-1} \to \HD{2n-k+1})~.
\end{split}\end{align}
where $0<k < n$.
\end{cor}

Note that for $p=0$ and $k=1$, we have
\begin{align}\label{cor_compact0}\begin{split}
\PHDP{1}(M) &\cong \HD{1}(M) \oplus \ker(L: \HD{0} \to \HD{2})\\
\PHDM{1}(M) &\cong \HD{2n-1}(M) \oplus \cok(L: \HD{2n-2} \to \HD{2n})
\end{split}
\end{align}
In the closed case, the symplectic structure and more generally its powers, $\om^r$, are non-trivial in $\HD{2r}(M)$.   Hence, the formulas above in \eqref{cor_compact0} simplify with the kernel and cokernel terms on the right vanishing.  Such simplification also holds more generally for $\FHPM{p}{2p+1}(M)$ as expressed in the following corollary.
\begin{cor}\label{cor_compact}
On a closed symplectic manifold $(M^{2n}, \omega)$,
$$\FHP{p}{2p+1}(M) = \HD{2p+1}(M)~, \qquad \FHM{p}{2p+1}(M) = \HD{2n-2p-1}(M)$$
for any $p\in\{0,1,\ldots,n-1\}$.
\end{cor}
This corollary extends the general isomorphism relations between filtered cohomologies, $F^pH_\pm^k(M)$ for $0\leq k\leq 2p$, and de Rham cohomologies in \eqref{federham1}-\eqref{fderham2}.  For the other filtered cohomologies, we write out explicitly their properties in the case of dimension four and six.
\begin{cor}\label{cor_dimension4}
Let $(M,\omega)$ be a $4$-dimensional symplectic manifold.  Then
\begin{align} \begin{split}
\PHPR{2}(M) &\cong \ker(L:\HD{1}\to\HD{3})\oplus \cok(L:\HD{0}\to\HD{2}) ~, \\
\PHPL{2}(M) &\cong \ker(L:\HD{2}\to\HD{4}) \oplus \cok(L:\HD{1}\to\HD{3}) ~.
\end{split} \end{align}
\end{cor}
\begin{proof}
The corollary follows from Theorem \ref{resolution} for $n=2$ and $r=1$.
\end{proof}
\begin{cor}
Let $(M,\omega)$ be a $6$-dimensional symplectic manifold.  Then
\begin{align} \begin{split}
\PHDP{2}(M) &\cong \ker(L:\HD{1}\to\HD{3})\oplus \cok(L:\HD{0}\to\HD{2})  ~, \\
\PHPR{3}(M) &\cong \ker(L:\HD{2}\to\HD{4})\oplus \cok(L:\HD{1}\to\HD{3})  ~, \\
\PHPL{3}(M) &\cong \ker(L:\HD{3}\to\HD{5}) \oplus \cok(L:\HD{2}\to\HD{4}) ~, \\
\PHDM{2}(M) &\cong \ker(L:\HD{4}\to\HD{6}) \oplus \cok(L:\HD{3}\to\HD{5}) ~,\\
\PHPR{2}(M) &\cong \ker(L^2:\HD{1}\to\HD{5})\oplus \cok(L^2:\HD{0}\to\HD{4}) ~, \\
\PHPL{2}(M) &\cong \ker(L^2:\HD{2}\to\HD{6}) \oplus \cok(L^2:\HD{1}\to\HD{5}) ~. \\
\end{split} \end{align}
\end{cor}
\begin{proof}
The corollary follows from Theorem \ref{resolution} for $n=3$ and $r=1,2$.
\end{proof}

Let us describe further a few more relations between filtered cohomologies.  In \cite{TY1,TY2}, it was shown that on a closed symplectic manifold, we have the following isomorphisms:
\begin{equation}\label{isomo}
\PHPR{k}(M) \cong \PHPL{k}(M)~, \qquad \PHDP{k}(M) \cong \PHDM{k}(M)~.
\end{equation}
This can also be seen from the above relations \eqref{res1} and \eqref{res3} after applying the following proposition:
\begin{prop}
Let $(M^{2n},\omega)$ be a closed symplectic manifold.  Then  
\begin{align} 
\ker(L^r: \HD{k} \to \HD{k+2r}) \cong  \cok(L^r: \HD{2n-k-2r} \to \HD{2n-k})~.
\end{align}
\end{prop}
\begin{proof}
This can be checked using the duality $\HD{k}(M) \cong \HD{2n-k}(M)$ for a closed manifold and focusing on the de Rham harmonic forms.
\end{proof}
This proposition together with \eqref{res2} then implies the following:
\begin{prop}
Let $(M,\omega)$ be a closed symplectic manifold.  Then  
\begin{align} \label{fhpmiso}
\FHP{p}{k}(M) \cong \FHM{p}{k}(M)~.
\end{align}
\end{prop}
Hence, we can now generalize the statement of Corollary \ref{uindex} to the case of a closed symplectic manifold. 
\begin{cor}
On a closed symplectic manifold, the index of the filtered elliptic complex of \eqref{fecomplex} is zero.
\end{cor}

\subsection{Examples}

\subsubsection{Cotangent bundle}

The filtered cohomologies can be straightforwardly calculated for the cotangent bundle $M=T^*N$ with respect to the canonical symplectic structure $\omega=-d\alpha$ where $\alpha$ is the tautological $1$-form.

Due to the fact that $N$ is a deformation retract of $M$ and that the de Rham cohomology is homotopically invariant, we have
$$ \HD{k}(M) = \HD{k}(N)~,$$
and hence, all the de Rham cohomological data on the bundle $M$ comes from the base $N$.  However, for filtered cohomologies, the Poincar\'e-lemma results of Section 3.2 are suggestive that $F^pH(M)$ should contain more information, for instance, they should involve the tautological one-form, $\alpha$.  With a local coordinate chart $\{x_1, \ldots, x_n, x_{n+1},\ldots, x_{2n}\}$ and the canonical symplectic form given by $\om = -d\alpha = \sum dx_i \w dx_{n+i}\,$, the following results for the primitive cohomologies, $F^0H(M)=PH(M)$, were obtained previously by direct calculation in \cite{TTY}:


\begin{prop}  The primitive symplectic cohomologies of the cotangent bundle $M=T^*N$ with respect to the canonical symplectic form are
\begin{enumerate}
\item $PH_{\dpp}^0(M)= H^0_{d}(N)$ and $ PH_\dpp^k(M) = \left\{ H^k_{d}(N)\,, \, \alpha \w H^{k-1}_{d}(N)\right\}$ for $1\leq  k < n\,$;
\item $PH_{d\dl}^0(M)= 0$, $PH_{d\dl}^k(M) = \left\{\alpha \w H^{k-1}_{d}(N)\right\}$ for $1\leq k < n\,$ and $PH_{d\dl}^n(M) = \left\{ H^n_{dR}(N)\,, \, \alpha \w H^{n-1}_{d}(N)\right\}$;
\item $PH_{d+\dl}^k(M) =  H^k_{d}(N)$ for $0\leq  k \leq n\,$;
\item $PH_\dpm^k(M) = 0$ for $0\leq  k < n\,$.
\end{enumerate}
\end{prop}
These results can now also be obtained using the long exact sequence of Theorem \ref{resolution}.  Moreover, we can also derive the following results for all filtered cohomologies by applying Theorem \ref{resolution}.
\begin{prop}  The filtered cohomologies of the cotangent bundle $M=T^*N$ with respect to the canonical symplectic form are
\begin{enumerate}
\item $\FHP{p}{k}(M) = \HD{k}(N)$ for $ 0\leq  k \leq 2p\,$;
\item $\FHP{p}{k}(M) = \left\{\HD{k}(N), \om^p \w \alpha \w \HD{k-2p-1}(N) \right\}$ for $ 2p+1 \leq  k \leq n\,$;
\item $\FHP{p}{k}(M) = \left\{ \om^p \w \alpha \w \HD{k-2p-1}(N) \right\}$ for $p>1$ and $ n+1 \leq  k \leq n+p\,$;
\item $\FHM{p}{k}(M) = \left\{\om^{k-n} \w \HD{2n-k}(N) \right\}\,$, for $n \leq k \leq n+p\,$;
\item $\FHM{p}{k}(M) = 0\, $, for $0 \leq k < n\,$.
\end{enumerate}
\end{prop}

From the above proposition, it is clear that the isomorphism relations that hold true for closed manifolds such as  $\FHP{p}{*}(M)\cong \FHM{p}{*}(M)$ \eqref{fhpmiso} or those in Corollary \ref{cor_compact} do not hold for the cotangent bundle $M=T^*N$ and are also generally not valid for open manifolds.


\subsubsection{Four-dimensional symplectic manifold from fibered three-manifold}\label{example_mapping_torus}

In this subsection, we apply Theorem \ref{resolution} to calculate the primitive cohomologies for another class of examples: the symplectic $4$-manifold which is the product of a fibered $3$-manifold with a circle.  Due to McMullen and Taubes \cite{MT}, such a construction provides the first example of a manifold with inequivalent symplectic forms.

The input is a closed surface $\Sigma$ with an orientation preserving self-diffeomorphism $\tau$.  The map $\tau$ is called the monodromy.  By Moser's trick, we may assume that there is a $\tau$-invariant symplectic form $\omega_\Sigma$ on $\Sigma$.  To be more precise, the monodromy $\tau$ might be replaced by another isotopic one.  Denote by $Y_\tau$ the mapping torus
\begin{align}\label{mapping_torus}
Y_\tau &= \Sigma\,\times_\tau S^1 = \frac{\Sigma\times[0,1]}{(\tau(x),0)\sim(x,1)} ~.
\end{align}
There is a natural map from $Y_\tau$ to $S^1$ induced by the projection $\Sigma\times[0,1]\to[0,1]$.  Let $\phi$ be the coordinate for the base of the fibration $Y_\tau\to S^1$.  Then, the $4$-manifold $X = S^1\times Y_\tau$ admits a symplectic form defined by
\begin{align*}
\omega &= dt\wedge d\phi +\omega_\Sigma
\end{align*}
where $t$ is the coordinate for the $S^1$-factor of $X$.

Noting Corollary \ref{cor_compact}, the interesting filtered/primitive cohomologies to consider for a compact symplectic $4$-manifold are $\PHPR{2}(X)$ and $\PHPL{2}(X)$.  Their dimensions are given by Corollary \ref{cor_dimension4}.  As we will see momentarily, the Lefschetz map $L$ on $X$ is determined by the map of wedging with $d\phi$ on $Y_\tau$.  Let us start with the following useful linear algebra lemma.

\begin{lem}\label{lem_linear}
Let $(V^{2n},\Omega)$ be a symplectic vector space, and let $A:V\to V$ be a linear symplectomorphism.  Then, the $\Omega$-orthogonal complement of $\ker(A-\bo)$ is $\im(A-\bo)$, where $\bo$ is the identity map on $V$.  As a consequence, $\ker(A-\bo)\cap\im(A-\bo)$ is exactly the kernel of $\Omega|_{\ker(A-\bo)}$.
\end{lem}

\begin{proof}
Suppose that $u\in\ker(A-\bo)$, which means that $Au = u$.  For any $v\in V$, we compute
\begin{align*}
\Omega(Av - v, u) &= \Omega(Av, u) - \Omega(v,u) \\
&= \Omega(Av,Au) - \Omega(v,u) = 0 ~.
\end{align*}
It follows that $\ker(A-\bo)$ and $\im(A-\bo)$ are $\Omega$-orthogonal to each other.  By dimension counting, they must also be the $\Omega$-orthogonal complement of each other.
\end{proof}

Now the de Rham cohomology of $Y_\tau$ can be standardly derived.
\begin{prop}\label{prop_mapping_torus}
Let $Y_\tau$ be the $3$-manifold defined by (\ref{mapping_torus}), and let $d\phi$ be the pull-back of the canonical $1$-form from $S^1$ to $Y_\tau$.  Then,
\begin{enumerate}
\item $\HD{1}(Y_\tau) \cong \spn\{ d\phi \}\oplus \ker\big( (\tau^*-\bo):\HD{1}(\Sigma)\to\HD{1}(\Sigma) \big)$;
\item $\HD{2}(Y_\tau) \cong \spn\{ \om_\Sigma \}\oplus \cok\big( (\tau^*-\bo):\HD{1}(\Sigma)\to\HD{1}(\Sigma) \big)$;
\item with the above identifications, the kernel of wedging with $d\phi$ from $\HD{1}(Y_\tau)$ to $\HD{2}(Y_\tau)$ is $\spn\{d\phi\} \oplus \big( \ker(\tau^*-\bo)\cap\im(\tau^*-\bo) \big)$.
\end{enumerate}
\end{prop}
\begin{proof}
These assertions basically follow from the Wang exact sequence:
\begin{align}\label{wang_exact}\xymatrix{
\cdots \ar[r]  &\,\HD{0}(\Sigma) \ar[r]  &\,\HD{1}(Y_\tau) \ar[r]  &\,\HD{1}(\Sigma) \ar[r]^-{\tau^*-\bo}  &\,\HD{1}(\Sigma) \ar[r]  &\,\HD{2}(Y_\tau) \ar[r]  &\,\HD{2}(\Sigma) \ar[r]  &\cdots
}\end{align}
which can be proved by the Mayer--Vietoris sequence.  Explicit construction of the differential forms will be given in Section \ref{example_product}.

\end{proof}
By the K\"unneth formula, the de Rham cohomologies of the $4$-manifold $X=S^1\times Y_\tau$ is then given as follows:
\begin{align*}
\HD{1}(X) &\cong \spn\{ dt, d\phi \}\oplus \ker\big( (\tau^*-\bo):\HD{1}(\Sigma)\to\HD{1}(\Sigma) \big) ~, \\
\HD{2}(X) &\cong \big(dt\wedge\HD{1}(Y_\tau)\big) \oplus \HD{2}(Y_\tau) ~, \\
\HD{3}(X) &\cong \spn\{ d\phi\wedge\omega_\Sigma, dt\wedge\omega_\Sigma \}\oplus \cok\big( (\tau^*-\bo):\HD{1}(\Sigma)\to\HD{1}(\Sigma) \big) ~.
\end{align*}
For a compact symplectic $4$-manifold, the only interesting Lefschetz map is the one from $\HD{1}(X)$ to $\HD{3}(X)$.  In the current case, the map is determined by the third item of Proposition \ref{prop_mapping_torus}.  With the help of Lemma \ref{lem_linear}, Theorem \ref{resolution} leads to the following proposition.

\begin{prop}\label{prop_mapping_torus1}
Suppose that $\Sigma$ is a closed surface, $\tau$ is a monodromy, and $\om_\Sigma$ is a $\tau$-invariant area form.  Then the $4$-manifold $X= S^1\times Y_\tau = S^1 \times \left(\Sigma \times_\tau S^1\right)$ with the symplectic form $\omega = dt\wedge d\phi + \om_\Sigma$ has the following properties:
\begin{enumerate}
\item Consider $\tau^*-\bo$ acting on $\HD{1}(\Sigma)$.  The dimension of ${\ker(\tau^*-\bo)}/{\big( \ker(\tau^*-\bo)\cap\im(\tau^*-\bo) \big)}$ is even, and denote it by $2p$.  Let $q+p$ with $q\geq p$ be the dimension of $\ker(\tau^*-\bo)$ and $q-p$ be the dimension of $\ker(\tau^*-\bo)\cap\im(\tau^*-\bo)$.
\item $\dim\HD{1}(X) = \dim\HD{3}(X) = q+p+2$ and $\dim\HD{2}(X) = 2q+2p+2$.
\item $\dim\PHPR{2}(X) = \dim\PHPL{2}(X) = 3q+p+1$ and $\dim\PHDP{1}(X) = \dim\PHDM{1}(X) = q+p+2$.
\end{enumerate}
\end{prop}

We remark that the dimensions of the de Rham cohomologies only depend on the dimension of the $\tau^*$-invariant subspace of $\HD{1}(\Sigma)$.  The dimensions of the primitive cohomologies involve the degeneracy of the intersection pairing on the $\tau^*$-invariant subspace of $\HD{1}(\Sigma)$.  We will return to this example in Section \ref{example_product} to demonstrate aspects of the product structures which we shall describe next.


\section{$A_\infty$-algebra structure on filtered forms}\label{sec_a_infinity}

The exact triangle \eqref{Tp} relates the filtered cohomologies closely with the de Rham cohomologies through Lefschetz maps.  It is thus tempting to think that some of the algebraic properties of the de Rham cohomology should also be present for filtered cohomologies.  For instance, an important property of the de Rham cohomology is its ring structure with the product operation taken to be the exterior product on forms.  Underlying this ring structure is the standard differential graded algebra on the space of differential forms, $(\CA^*, \w, d)$, with the two operations being the exterior product and the exterior derivative.  So could the filtered cohomology groups also be rings?  As we shall see in this section, the answer turns out to be yes.  However, there is not a differential graded algebra for filtered forms.  What we have instead is a generalization, that of an $A_\infty$-algebra on the space of $p$-filtered forms.

Let us first recall the definition of an $A_\infty$-structure (see, for example \cite{Stas,Keller}).  An $A_\infty$-algebra is a $\BZ$-graded vector space $\A= \oplus_{j\in \BZ}\, \A^j\,$, with graded maps,
$$m_k: \A^{\otimes k} \to \A~, ~~~~~~k =1, 2, 3, \ldots $$
of degree $2-k$ that satisfy the strong homotopy associative relation:
\begin{equation}\label{Leialg}
\sum_{r,\, t \,\geq \,0\,,\, s>0} \, (-1)^{r+s\, t} \; m_{r+t+1} \left(\bo^{\otimes r} \otimes m_s \otimes \bo^{\otimes t} \right)  = 0~,
\end{equation}
for each $k=r+s+t\,$.   Here, when acting on elements, the standard Koszul sign convention applies:
\begin{align}\label{signc}
(\varphi_1 \otimes \varphi_2)(v_1 \otimes v_2) &= (-1)^{|\varphi_2||v_1|} \varphi_1(v_1) \otimes \varphi_2(v_2)~,
\end{align}
where $\varphi_i$ are graded maps, $v_i$ are homogeneous elements, and the absolute value denotes their degree.

Explicitly, relation \eqref{Leialg} implies the following for the first three $m_k$ maps:
\begin{itemize}
\item $m_1: \A \to \A$ satisfies $m_1 m_1 = 0\,$.   Since $m_1$ increases the degree of the grading by one and squares to zero,  it is a differential with $(\A, m_1)$ a differential complex.
\item $m_2: \A^{\otimes 2} \to \A$ satisfies 
\begin{equation}\label{Leialg2}
m_1 m_2 = m_2 \left( m_1 \otimes \bo + \bo \otimes m_1\right)~.
\end{equation}
Here, $m_2$ preserves the grading, so it is considered a multiplication operator in $\A$.  With $m_1$ as the differential, condition \eqref{Leialg2} is just the requirement that the Leibniz product rule holds.  
\item $m_3: \A^{\otimes 3} \to \A$ satisfies 
\begin{equation}\label{Leialg3}
m_2 \left( \bo \otimes m_2 - m_2 \otimes \bo \right) =  m_1 m_3 + m_3 \left(m_1 \otimes \bo \otimes \bo + \bo \otimes m_1 \otimes + \bo \otimes \bo \otimes m_1 \right)
\end{equation}
The left-hand-side measures the associativity of the multiplication $m_2$.  Equation \eqref{Leialg3} effectively stipulates that $m_2$ is associative up to homotopy.
\end{itemize}

Let us note that a differential graded algebra is just a special case of an $\Ain$-algebra with the multiplication $m_2$ being associative, and hence, $m_k=0$ for all $k\geq 3$.  Moreover, even though the multiplication $m_2$ is in general not associative on $\A$, it is always associative on the associated homology $H^*\A = H^*(\A, m_1)$.  This follows directly from \eqref{Leialg3}, since acting on elements of $H^*(\A, m_1)$ which are $m_1$-closed, the right-hand-side is zero modulo the $m_1$-exact term, $m_1m_3$.

We now construct an $\Ain$-algebra on $p$-filtered forms.  We will denote it by $\Fp$.  The first step is to specify the $\Fp^j$ subspaces.  We shall use the assignment suggested by the $p$-filtered elliptic complex \eqref{fecomplex} and its associated filtered cohomology
$$F^pH = \{\FHP{p}{0}, \FHP{p}{1}, \ldots, \FHP{p}{n+p}, \FHM{p}{n+p}, \FHM{p}{n+p-1}, \ldots, \FHM{p}{0} \}~$$  
which consists of $2(n+p)+1$ distinct objects.  Assigning each to be the homology of a subspace, the nontrivial $\Fp^j$ subspaces should have degree in the range $0\leq j \leq 2(n+p)+1\,$. Specifically, we shall label the subspaces in the following way.  (See also Table \ref{Asub}.)
\begin{align}\label{Asubdef}\begin{split}
A_j  \in \FO{p}{j}=\Fp^j     \qquad &\qquad\text{for }  0\leq j \leq n+p~, \\
\bA_j \in \BFO{p}{j}=\Fp^{2n+2p +1 - j}  \qquad &\qquad\text{for } 0\leq j \leq n+p~.
\end{split}\end{align}
For clarity, since a $p$-filtered $j$-form may be in either $\Fp^j$ or $\Fp^{2n+2p+1-j}$ subspace, we have distinguished the two spaces by adding a bar to denote those $j$-forms in $\Fp^{2n+2p+1-j}$, i.e. $\bA_j \in \BFO{p}{j}=\Fp^{2n+2p+1-j}$.  We will follow this convention for the rest of this paper as well.
\begin{table}[t]
\begin{center}
\renewcommand{\tabcolsep}{.1cm}
\begin{tabular}{c | c | c | c | c | c | c | c | c | c}
$\Fp^0$ & $\Fp^1$ & \ldots &  $\Fp^{n+p-1}$ & $\Fp^{n+p}$ & $\Fp^{n+p+1}$ & $\Fp^{n+p+2}$ & \ldots &  $\Fp^{2n+2p}$ & $\Fp^{2n+2p+1}$\\
\hline
$\FO{p}{0}$ & $\FO{p}{1}$ & \ldots &  $\FO{p}{n+p-1}$ & $\FO{p}{n+p}$ & $\BFO{p}{n+p}$ & $\BFO{p}{n+p-1}$ & \ldots &  $\BFO{p}{1}$ & $\BFO{p}{0}$ \\
\end{tabular}
\end{center}
 \caption{The $\Fp^j$ subspaces of a $p$-filtered graded algebra $\Fp$ following the notation of \eqref{Asubdef}.
\label{Asub}} 
\end{table} 

Further, mimicking closely the filtered elliptic complex, we choose the differential of the $\Ain$-algebra $d_j: \Fp^j \to \Fp^{j+1}$, i.e. the $m_1$ map, to be as follows.
\begin{align}\label{dop}
d_j  =
\begin{cases}
~\ddp &  \text{if~~ } 0\leq j < n+p-1~,\\
-\dpp\dpm & \text{if~~ } j = n+p~,\\
-\ddm & \text{if~~ } n+p+1 \leq j \leq 2n + 2p +1
\end{cases}
\end{align}
This differential clearly satisfies $d_{j+1} d_j =0\,$ on the space $\{\FO{p}{*}, \BFO{p}{*}\}\,$.  It only differs from the differential operators of the elliptic complex by a negative sign in front of the ``minus" operators $\dpm$ and $\ddm$.  We will see that these negative signs are needed for satisfying the Leibniz rule conditions in Section \ref{subsec_Leibniz} below.

\subsection{Product on filtered forms}\label{subsec_prod}

The symplectic elliptic complex \eqref{fecomplex} motivated the definition of the grading and the $m_1$ map.  To obtain the $m_2$ multiplication map, we turn to the long exact sequence of cohomology  \eqref{longexactseq} and its underlying chain of short exact sequences \eqref{ccomplex1e}.  These exact sequences are suggestive of how to define a product on $\FO{p}{*}$ for they contain maps between $\FO{p}{*}$ and $\CA^*$ such as $\pil{p}\,,\, \rstar\,, \lpod\,,$ and  $\psdl$.  So to define a product on filtered forms, we can first map $\FO{p}{*}$ to $\CA^*$, then apply the wedge product on $\CA^* \otimes \CA^*$, and finally map the resulting form back to $\FO{p}{*}$ with the desired grading.  (See Figure \ref{lmap38} for more details.)   In this way, we are led to defining the following product operation on $\Fp=\{\FO{p}{*},\BFO{p}{*}\}$:
\begin{defn} \label{defprod}  The product $\times: \Fp^j \otimes \Fp^k \to \Fp^{j+k}$ is defined as follows: 
\begin{align}
A_j \times  A_k & = \pil{p} (A_j \w A_k) \label{prod12}\\ 
& \quad + \pil{p} \rstar\left[-d\, \lpo(A_j \w A_k) + (\lpod A_j) \w A_k + (-1)^j A_j \w (\lpod A_k)\right] \nonumber\\
A_j \times \bA_k & = (-1)^j \rstar \left(A_j \w (\rstar \bA_k)\right)\label{prod3}\\
\bA_j \times A_k & = \rstar \left((\rstar \bA_j)\w A_k\right) \label{prod4}\\
\bA_j \times \bA_k &  = 0 \label{prod5}
\end{align}
where we have used the notation of \eqref{Asubdef} denoting $A_j \in \Fp^j$, and $\bA_j \in \Fp^{2n+2p+1-j}$ for  $0\leq j \leq n+p\,$.
\end{defn}

\begin{figure}
\begin{align*}
\xymatrix
@R-10pt
@H+5pt
{& Column& A & B & C & D & \\
Grading & & &  &  \vdots\ar[d]_{d}  &  \vdots\ar[d]_{\ddp} & \\
2 & & &\;0\ar[r]  & \; \CA^{2} \ar[r]^{\pil{1}} \ar[d]_{d} &\; \FO{1}{2} \ar[r] \ar[d]_{\ddp}& \; 0\\
3 & & &\;0\ar[r]  & \; \CA^{3} \ar[r]^{\pil{1}} \ar[d]_{d} &\; \FO{1}{3} \ar[r] \ar[d]_{\ddp}& \; 0\\
4& &\;0\ar[r] & \; \CA^{0} \ar[r]^{L^2}\ar[d]_{d} & \; \CA^{4} \ar[r]^{\pil{1}} \ar[d]_{d}&\; \FO{1}{4} \ar[r] \ar[d]_{\ddp} & \; 0\\
5& &\;0\ar[r] & \; \CA^{1} \ar[r]^{L^2}\ar[d]_{d} & \; \CA^{5} \ar[r]^{\pil{1}} \ar[d]_{d}&\; \FO{1}{5} \ar[r] & \; 0\\
&&\;0\ar[r] & \; \CA^{2} \ar[r]^{L^2} \ar[d]_{d}& \; \CA^{6} \ar[r]\ar[d]_{d} & \; 0& \\
6& \;0\ar[r] & \; \BFO{1}{5} \ar[r]^{\rstar}\ar[d]_{\ddm} & \; \CA^{3} \ar[r]^{L^2}\ar[d]_{d} &\; \CA^7 \ar[r] \ar[d]_{d}& \; 0&\\
7& \;0\ar[r] & \; \BFO{1}{4} \ar[r]^{\rstar}\ar[d]_{\ddm} & \; \CA^{4} \ar[r]^{L^2}\ar[d]_{d} &\; \CA^8 \ar[r] \ar[d]_{d}& \; 0&\\
8& \;0\ar[r] & \; \BFO{1}{3} \ar[r]^{\rstar} \ar[d]_{\ddm} & \; \CA^{5}\ar[d]_{d} \ar[r] &\; 0 & &\\
9&\;0\ar[r] & \; \BFO{1}{2} \ar[r]^{\rstar} \ar[d]_{\ddm} & \; \CA^{6}\ar[d]_{d} \ar[r] &\; 0 & &\\
&   & \vdots  & \vdots & &
}
\end{align*}
\caption{Consider as an example the above commutative diagram of Lemma \ref{ccomplex1} in dimension $2n=8$ for the degree two Lefschetz map which involves the $p=1$ filtered forms $\{\FO{1}{*}, \BFO{1}{*}\}$.  The filtered product of Definition \ref{defprod} can be heuristically understood as first mapping the filtered forms in Columns A and D into Columns B and C.  Once in the middle two columns, the wedge product can be applied and then the resulting form can be projected back to the outer columns.  For the case of $\FO{p}{j} \times \FO{p}{k}$ where $j+k > n+p$, the product crosses the middle row of the diagram which notably is without filtered forms and therefore has no grading assignment.  Hence,  in order to obtain the desired product grading of $j+k>n+p$, the definition of the product must involve a derivative map which shifts forms down by a row.  The three terms in \eqref{prod22} correspond to the three different ways one can apply the derivative map to a product of two filtered forms.} \label{lmap38}
\end{figure}

Let us note that the product of $p$-filtered forms $\FO{p}{*}$ in \eqref{prod12} simplifies depending on the value of $j+k$.  For if $j+k> n+p$, then $\pil{p}(A_j\w A_k)=0$.  On the other hand, the terms in the bracket of \eqref{prod12} become trivial when $j+k\leq n+p$.  Hence, we can write
\begin{numcases}{A_j \times A_k =}  
\pil{p} (A_j \w A_k)  &\!\!\!\!\!\!\!if~$j+k \leq n + p$~~~~~~~~\label{prod21}\\
{\pil{p}} \rstar\Big[-d\, \lpo(A_j \w A_k) &  \nn\\
\qquad\quad + (\lpod A_j) \w A_k + (-1)^j A_j \w (\lpod A_k) \Big] &\!\!\!\!\!\!\!if~$j+k >\!n+p$~~~~~~~~\label{prod22}
\end{numcases}

Notice also that the expressions on the right hand side of \eqref{prod3} and \eqref{prod4} are automatically $p$-filtered.  This can be seen simply by applying the $p$-filtered condition \eqref{pistarf} and using \eqref{rsak} and $(\rstar)^2 = \bo\,$.  Furthermore, the product $A_j \times \bA_k$ of \eqref{prod3} is identically zero unless $j \leq k\,$.  (Similarly, for \eqref{prod4}, a non-trivial product only occurs for $k\leq j\,$.)  This property is natural since the product $\Fp^j \times \Fp^k = 0$ if $j+k > 2n+2p+1$, as subspaces with grading greater than $2n+2p+1$ are defined to be the empty set.  This also explains why the product in \eqref{prod5} is trivial.  Lastly, the factor of $(-1)^j$ in \eqref{prod3} ensures that the product is graded commutative. That is, 
$$\Fp^j \times \Fp^k = (-1)^{jk} \Fp^k \times \Fp^j~.$$

Now we can check that our definition of the filtered product $\times$ is consistent with the exact triangle of \eqref{Tpth}.   At the level of cohomology, the long exact sequence \eqref{longexactseq} locally has the following form (setting $r=p+1$ in \eqref{longexactseq})
\begin{align}\label{wm01} \xymatrix{
\; \ldots \ar[r] &\;\HD{j-2p-2} \;\ar[r]^-{L^{p+1}}  &\;\HD{j} \;\ar[r]^-{f}  &\; F^pH^{j} \;\ar[r]^-{g}  &\;\HD{j-2p-1} \;\ar[r]^-{L^{p+1}}  &\;\HD{j+1} \ar[r]& \; \ldots
} \end{align}
where
\begin{align}\label{wm03}
(\,f\,,\,g\,) &= \begin{cases}
(\;\pil{p}~,~ L^{-(p+1)}\, d\;) &\text{for }j\leq n+p ~, \\
(\;\pip \,\rstar\, d\,L^{-(p+1)}~,~ \rstar\;) &\text{for }j>n+p ~,
\end{cases}\end{align}
and we have denoted the filtered cohomology by
\begin{align}\label{fphj}
F^pH^{j} & =  
\begin{cases} 
\FHP{p}{j} &\text{for }j\leq n+p ~, \\
\FHM{p}{2n+2p+1-j} & \text{for }j>n+p ~.
\end{cases}
\end{align}
Heuristically, we can view the product of two filtered cohomologies as tensoring two long exact sequences {\it locally} in the following way:
\begin{align}\begin{split} \label{wmd} \xymatrix
@C=-5pt
{
&\; \HD{j-2p-2} \ar[d]_-{L^{p+1}} & &\;\HD{k-2p-2} \ar[d]_-{L^{p+1}}  & {~~~~~~~~~~~~~~}  &\;\HD{j+k-2p-2} \ar[d]_-{L^{p+1}}\\
&\; \HD{j} \ar[d]_-{f}&{\otimes} &\;\HD{k}\;~~ \ar[d]_-{f} \ar[rr]^-{\wedge}  &  &~~\;\HD{j+k} \ar[d]_-{f} \\
&\; F^pH^j \ar[d]_-{g}&{\otimes}&\; F^pH^k\; \ar[d]_-{g} \ar[rr]^-{\times}  &  &\; F^pH^{j+k} \ar[d]_-{g} \\
&\; \HD{j-2p-1} \ar[d]_-{L^{p+1}}&{\otimes} &\;\HD{k-2p-1} \ar[d]_-{L^{p+1}} \ar[rr]^-{\Massey{\;}{\;}}  &  &\;\HD{j+k-2p-1} \ar[d]_-{L^{p+1}} \\
&\; \HD{j+1} & &\;\HD{k+1}  &  &\;\HD{j+k+1} 
} \end{split}\end{align}
In above, the product at the top, $\HD{j} \otimes \HD{k} \to \HD{k+j}$, is clearly just the standard wedge product.  The product at the bottom, $\HD{j-2p-1} \otimes \HD{k-2p-1} \to \HD{j+k-2p-1}$, however does not have the correct degrees for it to be a wedge product.   Instead, for our purpose, it has the interpretation as the standard Massey triple product with the middle element of the triple product fixed to be $\om^{p+1}$.  To see this, since our concern is the filtered product, we are mainly interested in the subset of elements of $\HD{j-2p-1}$ and $\HD{k-2p-1}$ that are in the image of $g$ from $F^pH^j$ and $F^pH^k$, respectively.  By the exactness of the sequences, these elements are in the kernel of the Lefschetz map:
\begin{align*}
[\xi_{j-2p-1}]&\in\ker[L^{p+1}:\HD{j-2p-1}\to\HD{j+1}]  &\text{and}&
&[\xi_{k-2p-1}]\in\ker[L^{p+1}:\HD{k-2p-1}\to\HD{k+1}] ~.
\end{align*}
Therefore, there must exist an $\eta_{j}\in\CA^{j}$ and an $\eta_{k}\in\CA^{k}$ such that
\begin{align*}
d\eta_{j} &= \omega^{p+1}\w\xi_{j-2p-1}  &\text{and}&  &d\eta_k &= \omega^{p+1}\w\xi_{k-2p-1} ~.
\end{align*}
Given this, it is then natural to consider the Massey triple product $\langle\xi_{j-2p-1},\, \om^{p+1},\, \xi_{k-2p-1}\rangle$.  With the symplectic structure element $\om^{p+1}$ kept fixed, this Massey triple product then defines what we shall simply call the Massey product:
\begin{align}\label{wm02}
\Massey{\xi_{j-2p-1}}{\xi_{k-2p-1}} = \xi_{j-2p-1}\wedge\eta_{k} + (-1)^j\eta_{j}\w\xi_{k-2p-1}\in\frac{\HD{*}(M)}{\CI(\xi_{j-2p-1},\xi_{k-2p-1})} ~,
\end{align}
where $\CI(\xi_{j-2p-1},\xi_{k-2p-1})$ is the ideal generated by $\xi_{j-2p-1}$ and $\xi_{k-2p-1}$.  We note that this Massey product is only well-defined on the quotient since different choices of $\eta_j$ and $\eta_k$ may differ by a $d$-closed form.

We can now ask whether the filtered product $\times$ is compatible with the wedge and Massey products that surround it in \eqref{wmd}.  For the filtered product we defined in Definition \eqref{defprod}, the diagram in \eqref{wmd} in fact  commutes.  The precise statement of this is given in the following theorem whose proof is given in Appendix \ref{apx_product}.
\begin{thm}
Let $(M,\omega)$ be a symplectic manifold.  The product operator $\times$ on $F^pH^*(M)$ is compatible with the topological products.  That is, it satisfies the following properties:   
\begin{enumerate}
\item (Wedge product) For any two $[\xi_j]\in\HD{j}(M)$ and $[\xi_k]\in\HD{k}(M)$, $f(\xi_j\wedge\xi_k) = f(\xi_j)\times f(\xi_k)$.  The equality is considered on the filtered cohomology class corresponding to $f:\HD{j+k}(M)\to F^pH^j\,$;
\item (Massey product) For any two $[A]\in F^pH^*(M)$ and $[A']\in F^pH^*(M)$, $\Massey{g(A)}{g(A')} = g(A\times A')$.  To be more precise, the equality is considered on $\HD{*}(M)/\CI(g(A),g(A'))$;
\end{enumerate}
where the maps $f$ and $g$ are defined by \eqref{wm03} and the Massey product is defined by \eqref{wm02}.
\label{compthrm}
\end{thm}


\subsection{Leibniz rules}\label{subsec_Leibniz}

Having defined $m_1=d_j$ and $m_2 = \times$, we now show that they satisfy the Leibniz rule \eqref{Leialg2}:
$$d_{j+k}(\Fp^j \times \Fp^k) = d_j \Fp^j \times \Fp^k + \moj \Fp^j \times d_k \Fp^k~.$$
Note that the differential $d_j$ as given in \eqref{dop} varies with the subspace $\Fp^j$ and can be either a first- or second-order differential operator.  Thus we will need to consider the Leibniz rule condition for different $\Fp^j \times \Fp^k$ cases separately.

\begin{thm}[Leibniz's Rule]\label{LeiRule}
For $A_j, A_k\in \FO{p}{*}$ and $\bA_k \in \BFO{p}{*}$, the following holds:
\begin{align}
\textup{(i)}~&\, \ddp(\ab \times \ac) = d_j \ab \times \ac + \moj \ab \times d_k \ac  \quad  &{\text for~ } j+k < n+p \label{lei1}\\
\textup{(ii)}~& -\dpp\dpm(\ab \times \ac) = d_j \ab \times \ac + \moj \ab \times d_k \ac  \quad  &{\text for~ } j+k = n+p \label{lei2}\\
\textup{(iii)}~& -\ddm(\ab \times \ac) = d_j \ab \times \ac + \moj \ab \times d_k \ac  \quad  &{\text for~ } j+k > n+p \label{lei3}\\
\textup{(iv)}~& -\ddm(\ab\times\bac) = d_j \ab \times \bac + \moj \ab \times (-\ddm) \bac  \quad  &{\text for~ } j\leq k ~~~~~~~~~~ \label{lei4} 
\end{align}
where $d_j=\ddp$ if $j<n+p$ and $d_j= \dpp\dpm$ if $j=n+p$. 
\end{thm}

Before proving the theorem, let us set up some convention that we shall use frequently.  We shall express the decomposition of $\ab, \ac\in \FO{p}{*}$ as
\begin{align*}
A_j &= B_j + \om\w B_{j-2} + \ldots + \om^p \w B_{j-2p}~, \\
A_k &= B_k + \om \w B_{k-2} + \ldots + \om^p \w B_{k-2p}~,
\end{align*}
where the $B$'s are primitive forms and therefore,
\begin{align}\label{Lei011}
L^{-(p+1)} dA_j = \dpm B_{j-2p}~, \qquad L^{-(p+1)} dA_k = \dpm B_{k-2p} ~.
\end{align}

Thus for example, we would write
\begin{align}\label{lei02}
\ddp &\ab \pp \ac + \moj \ab \pp \ddp \ac \nn\\
&=\pprs\Big\{-d\, \lpo(\ddp\ab \w \ac)+ \lpod(\ddp\ab)\w \ac + \moj \ddp\ab \w \lpod\ac \nn\\
&\qquad + \moj\big[ -d\,\lpo(\ab\w\ddp\ac) + \lpod \ab \w \ddp\ac + \moj\ab \w \lpod(\ddp\ac) \big]\Big\}\nn\\
&=\pprs\Big\{-d\, \lpo\big(\ddp\ab\w \ac + \moj \ab \w \ddp\ac\big)-\dpp\dmj\w\ac - \ab \w \dpp\dmk\nn\\
&\qquad\qquad\quad+\moj\big(\dmj \w \ddp\ac + \ddp\ab \w \dmk\big)\Big\}
\end{align}
where the second equality uses the fact that $d(\ddp \ab) = d(d\ab - \opo\dmj) = -\opo \dpp\dmj$
for $\ab\in \FO{p}{*}$.

\begin{proof}[Proof of Theorem \ref{LeiRule}]

\

\noindent Case (i): $\Fp^j \times \Fp^{k}\,,~~ j+k < n+p\,$.
\begin{align*}
\ddp(\ab\w\ac)&= \ddp \pip(\ab\w\ac) = (\pip d)(1-\opo\lpo)(\ab\w\ac)\nn\\
&=\pip d(\ab\w\ac) - \pip\big(\opo \w d\,\lpo(\ab\w\ac)\big)\\
&=\pip\left[(\ddp\ab + \opo \w \dmj) \w \ac + \moj \ab \w (\ddp\ac + \opo\w \dmk)\right]\\
&=\ddp \ab \pp \ac + \moj \ab \w \ddp \ac ~.
\end{align*}

\noindent Case (ii): $\Fp^j \times \Fp^{k}\,,~~ j+k = n+p\,$.

Let us first note that for $j+k = n+p$,
\begin{align}\label{lei21}
\dpp\dpm(\ab\pp\ac) = \dpp\dpm[\pip(\ab\w\ac)] = \pprs d\, \lpo d(\ab\w\ac) ~.
\end{align}
To see this, we can write 
$$\ab\w\ac = \om^{p}\w B_{n-p} + \om^{p+1}\w A'_{n-p-2}$$
for some $B_{n-p}\in \CB^{n-p}$ and $A'_{n-p-2}\in\CA^{n-p-2}$.  
Then, we can calculate both sides of \eqref{lei21} 
\begin{align*}
\dpp\dpm(\ab\pp\ac) &= \dpp\dpm(\om^p\w B_{n-p})= \om^p \w \dpp\dpm B_{n-p} ~, \\
\pprs d\, \lpo d(\ab\w\ac) &= \pprs d (\dpm B_{n-p} + \om \w dA'_{n-p-2}) \\
&= \pprs \dpp\dpm B_{n-p} = \om^p\w\dpp\dpm B_{n-p}
\end{align*}
which verifies \eqref{lei21}.  Furthermore, we have
\begin{align}\label{lei22}
&\pprs d\,\lpo d(\ab\w\ac)\nn\\
&=\pprs d\, \lpo \big[\ddp\ab \w \ac + \opo\dmj \w \ac+\moj(\ab\w\ddp\ac + \ab\w\opo\dmk)\big]\nn\\
&=\pprs d\,\lpo\big[\ddp\ab \w \ac +\moj\ab\w\ddp\ac\big]\nn\\
&\quad +\pprs\big[\dpp\dmj\w\ac -\moj \dmj\w d\ac + \moj dA_j \w \dmk + \ab \w \dpp\dmk\big]\nn\\
&=\pprs d\,\lpo\big[\ddp\ab \w \ac +\moj\ab\w\ddp\ac\big]\nn\\
&\quad +\pprs\big[\dpp\dmj\w\ac -\moj \dmj\w \ddp\ac + \moj \ddp\ab \w \dmk + \ab \w \dpp\dmk\big]
\end{align}
where the last equality is obtained by canceling  $\pprs[\opo \dmj \w \dmk]$.  Note that \eqref{lei22} is precisely equal to the minus of \eqref{lei02}.  With \eqref{lei21}, this proves the Leibniz rule \eqref{lei2}.

\noindent Case (iii):  $\Fp^j \times \Fp^k \,, ~~j,k\leq n+p {\rm ~~and~~} j+k > n+p\,$.

Note first,
\begin{align}\label{lei31}
-\ddm(\ab\pp\ac) = -\pprs d \rstar (\ab\pp\ac) ~.
\end{align}
Now for $j+k > n+p$, we can use (\ref{relation2}) to write
\begin{align*}
\rstar(\ab\pp\ac)&=\pips\big[-d\,\lpo(\ab\w\ac) + \lpo d\ab \w \ac + \moj \ab \w \lpo d\ac\big]\nn\\
&=-d\,\lpo(\ab\w\ac) + \lpo d\big( \opo\lpo(\ab\w\ac) \big)\\
&\qquad+ \dmj\w\ac + \moj \ab \w \dmk\nn\\
&\qquad  - \lpo\left(\opo\dmj \w \ac + \moj \ab \w \opo\dmk\right)\nn ~.
\end{align*}
Since $j+k > n+p$, \eqref{relation1} says that $\opo\lpo(\ab\w\ac) = \ab\w\ac$.  Furthermore, we can substitute  $\opo\dmj\w\ac= (d\ab-\ddp\ab)\w\ac$ and $\ab\w \opo\dmk = \ab \w (d\ac - \ddp\ac)$.  
After some cancellations, we find
\begin{align}
\rstar(\ab\pp\ac)&=-d\, \lpo(\ab\w\ac) + \lpo\big[\ddp\ab\w\ac + \moj \ab \w\ddp\ac\big] \nn\\
&\qquad + \dmj\w\ac + \moj \ab\w\dmk ~. \label{lei32}
\end{align}
By substituting \eqref{lei32} into \eqref{lei31}, we obtain
\begin{align}\begin{split}
-\ddm(\ab\pp\ac)&= -\pprs\Big\{ d\,\lpo\big(\ddp\ab\w\ac + \moj \ab \w\ddp\ac\big)\\
&\qquad\quad +d\big(\dmj\w\ac + \moj \ab\w\dmk\big) \Big\} ~.
\end{split}\end{align}
After applying the derivative on the second line, it gives precisely \eqref{lei02} which proves the Leibniz rule \eqref{lei3}.

\noindent Case (iv): $\Fp^j \pp \Fp^\ell  \,, ~~j\leq n+p {\rm ~~and~~} \ell > n+p\,$.

Let $k=2n+2p+1-\ell$ and $\bac\in\BFO{p}{*}$.  Then we have
\begin{align*}
-\ddm(\ab \pp \bac) &= -\moj \ddm \rstar(\ab \w \,\rstar\bac)= -\moj \rstar d(\ab \w \rstar\bac) \nn\\
& = -\moj \rstar \big[d\ab \w \rstar \bac + \moj \ab \w d\, \rstar\bac\big]\nn\\
&= -\moj \rstar\big[(\ddp\ab + \opo\dmj) \w \, \rstar\bac + \moj \ab \w \rstar\ddm\bac\big]\nn\\
&=(-1)^{j+1}\rstar(\ddp\ab \w \rstar \bac) - \rstar(\ab \w \rstar\ddm\bac)\\
&=\ddp\ab \pp \bac +(-1)^{j} \ab \pp (-\ddm \bac)
\end{align*}
where \eqref{pistarf} is invoked to set $\opo \rstar \bac=0$ for the fourth equality.

\end{proof}

\subsection{Non-associativity of product}

We now analyze the associativity of the product.  In general, $A^i \times (A^j \times A^k) \neq (A^i \times A^j) \times A^k$.  Hence, there is a non-trivial $m_3$ map.  We will show that the induced $m_3$ map satisfies \eqref{Leialg3}.

Due to the definition \eqref{prod12}--\eqref{prod5}, there are three distinct cases when considering the triple product $A^i \times A^j \times A^k$:
\begin{itemize}
\item[(i)] $i+j+k \leq n+p\,$;
\item[(ii)] $i+j+k > n+p$ and $i,j,k\leq n+p\,$;
\item[(iii)] $\max\{i,j,k\}> n+p\,$.
\end{itemize}
We will show for both case (i) and case (iii) that the triple product is associative.  In contrast, case (ii) will be seen to be in general non-associative.  

Consider first case (i). Here,
\begin{align*}
A_i \times (A_j \times A_k) & = A_i \times \Big[(1 - \opo \lpo)(A_j\w A_k)\Big] \\
& = \pil{p} \Big(A_i\w A_j \w A_k  - \opo\w  A_i \w \lpo \big(A_j\w A_k\big)\Big) ~. \\
& = \pil{p} \big(A_i \w A_j \w A_k\big)
\end{align*}
It is not hard to see that $(A_i\times A_j)\times A_k$ is also equal to $\pil{p}(A_i \w A_j \w A_k)$.  Thus, the triple product for $i+j+k \leq n+p$ is associative.

For case (iii), we will show as an example the case where $i > n+p$ and $j+ k \leq n +p$.  Let $\ell= 2n + 2p +1 -i$.  The triple product corresponds to 
\begin{align*}
\bA_\ell \times (A_j \times A_k) & =  \rstar \left(\rstar \bA_\ell \w \pip( A_j \w A_k)\right)\\\
&= \rstar \left(\rstar \bA_\ell \w (1 - \opo\lpo) ( A_j \w A_k)\right)\\
& = \rstar \left(\rstar \bA_\ell \w A_j \w A_k\right)
\end{align*}
since by \eqref{pistarf}, $\opo \rstar \bA_\ell =0\,$.  From definition \eqref{prod3}, we have
\begin{align*}
(\bA_\ell \times A_j) \times A_k  =  \rstar ( \rstar \bA_\ell \w A_j) \times A_k = \rstar \left(\rstar \bA_\ell \w A_j \w A_k\right) ~.
\end{align*}
Therefore, the triple product is associative if $i > n+p$.  In a similar manner, the associativity for the case when $j > n+p$ or $k>n+p$ can also easily be shown.  We also note that if a pair or all three indices are greater than $n+p$, then the triple product is identically zero.

For case (ii), we have the following result.

\begin{prop}\label{mthreea} For $A_i, A_j, A_k \in \FO{p}{*}$ and $i + j + k \geq n+p$, 
\begin{align}
A_i \times (A_j \times A_k) & = -\ddm \pprs \left[A_i \w \lpo(A_j\w A_k)\right] \nn\\
& \quad + \pprs\Big\{ \ddp A_i \w \lpo(A_j \w A_k) + (-1)^{i} A_i \w \lpo (\ddp A_j \w A_k) \nn\\
& \qquad\qquad\qquad + (-1)^{i+j} A_i \w \lpo (A_j \w \ddp A_k) \Big\}\label{ipjk}\\
& \quad + \pprs\Big\{-d \lpo (A_i \w A_j \w A_k) + \dmi \w A_j \w A_k \nn\\
& \qquad\qquad\qquad + \moi A_i \w \dmj \w A_k + \moij A_i \w A_j \w \dmk  \Big\}\nn ~.
\end{align}
\end{prop}

\begin{proof}
Assume first $j+k \leq n+p$, then $A_j \times A_k = \pil{p}(A_j \times A_k).$  Hence,
\begin{align}
A_i \times (A_j \times A_k) & = A_i \times \big[ (1 - \opo\lpo)(A_j \w A_k)\big] \nn\\
& = \pprs \Big[ -d\, \lpo \left(A_i \w A_j \w A_k - A_i \w \opo \lpo(A_j\w A_k) \right) \nn\\
&\qquad\qquad + (\lpo dA_i) \w (1 - \opo \lpo) (A_j\w A_k) \label{ipjk1}\\ 
&\qquad\qquad + \moi A_i\w  \lpo d\big[(1- \opo\lpo) A_j\w A_k\big] \Big] ~. \nn
\end{align}
We will analyze the six terms in \eqref{ipjk1} separately.  For the 2nd, 4th, and 6th term, we obtain
\begin{align}
\pprs &\Big[ d(1-\pips)(A_i \w \lpo(A_j \w A_k) - \dmi \w \opo \lpo (A_j\w A_k) \nn \\
&\qquad\qquad - \moi A_i \w  d\, \lpo(A_j \w A_k)\Big]\nn\\
&= - \ddm \pprs \left[A_i \w \lpo(A_j \w A_k)\right] + \pprs \left[\ddp A_i \w \lpo(A_j \w A_k)\right] ~. \label{ipjk2}
\end{align}
For the first term of the first line, we have used \eqref{relation2}.  For the second line, we have set $\lpo \opo $ equal to one since it is acting on $d\, \lpo(A_j\w A_k) \in \ker(\pips)$ for $j+k \leq n+p$.
For the 3rd and 5th term, we obtain
\begin{align}
\pprs &\left\{ (\lpo dA_i) \w A_j \w A_k  + \moi A_i \w \lpo d(A_j\w A_k)\right\} \nn\\
&=\pprs \Big\{\dmi \w A_j \w A_k + \moi A_i \w \lpo \left[(\ddp A_j + \opo \dmj) \w A_k\right] \nn\\
& \qquad\qquad\qquad\qquad + \moij A_j \w \lpo \left[ A_i \w (\ddp A_k + \opo \dmk)\right] \Big\} ~. \label{ipjk3}
\end{align}
Inserting \eqref{ipjk2} and \eqref{ipjk3} into \eqref{ipjk1} then results in \eqref{ipjk}. 

The remainder case is when $j+k > n+p$, and we have
\begin{align}
A_i \times (A_j \times A_k) &= \moi \pprs \Big[A_i \w \pips \big[ -d \, \lpo(A_j\w A_k)\nn\\ 
&\qquad\qquad\qquad\qquad\qquad\qquad+ \lpo dA_j \w A_k + \moj A_j \w \lpo dA_k\big]\Big]\nn\\
&=\moi \pprs \Big[- A_i  \w d\, \lpo(A_j\w A_k) \nn\\
&\qquad\qquad\qquad\qquad + A_i \w \lpo(\ddp A_j \w A_k + \moj A_j \w \ddp A_k) \label{ipjk4}\\
&\qquad\qquad\qquad\qquad + A_i \w (\dmj \w A_k + \moj A_j \w \dmk) \Big] ~. \nn
\end{align}
The second equality uses (\ref{relation2}) and the fact that $\om^{p+1}\lpo(A_j \w A_k)= A_j\w A_k$ for $j+k>n+p$.
We can re-express the first term of \eqref{ipjk4} as follows:
\begin{align}
- \pprs &d\left(A_i \w \lpo(A_j\w A_k)\right) + \pprs \left(dA_i \w \lpo(A_j\w A_k)\right) \nn\\
&= - \ddm \pprs (A_i\w \lpo(A_j \w A_k)) - \pprs d\, \lpo(A_i \w \opo\lpo(A_j\w A_k)) \nn\\
&\qquad + \pprs \left[(\ddp A_i \w \lpo(A_j \w A_k) + \dmi \w A_j\w A_k \right]  \label{ipjk5}
\end{align}
where we have used \eqref{Lei020} and $\om^{p+1}\lpo(A_j \w A_k)= A_j\w A_k$ for the second line.  Inserting \eqref{ipjk5} into \eqref{ipjk4} then gives \eqref{ipjk}.
\end{proof}
An analogous calculation gives the result for the other order of multiplication.
\begin{prop}\label{mthreeb} For $A_i, A_j, A_k \in \FO{p}{*}$ and $i + j + k \geq n+p$, 
\begin{align}\
(A_i \times A_j) \times A_k & = -\ddm \pprs \left[\lpo(A_i \w A_j)\w A_k \right] \nn\\
& \quad + \pprs\Big\{ \lpo(\ddp A_i \w A_j) \w A_k + (-1)^{i} \lpo(A_i \w \ddp A_j) \w A_k \nn\\
& \qquad\qquad\qquad + (-1)^{i+j} \lpo(A_i \w  A_j) \w \ddp A_k \Big\}\label{pijk}\\
& \quad + \pprs\Big\{-d \lpo (A_i \w A_j \w A_k) + \dmi \w A_j \w A_k \nn\\
& \qquad\qquad\qquad + \moi A_i \w \dmj \w A_k + \moij A_i \w A_j \w \dmk  \Big\} ~.\nn
\end{align}
\end{prop}

With \eqref{ipjk} and \eqref{pijk}, we find that for $A_i, A_j, A_k \in \FO{p}{*}$ and $i + j + k \geq n+p$,
\begin{align}\label{aijk}
A_i &\times (A_j \times A_k) - (A_i \times A_j) \times A_k \nn\\
&= -\ddm \Big\{\pprs \big[A_i \w \lpo(A_j\w A_k) - \lpo(A_i \w A_j)\w A_k \big] \Big\} \nn\\
& \qquad \pprs \Big\{ \ddp A_i \w \lpo(A_j\w A_k) - \lpo(\ddp A_i \w A_j)\w A_k \\
& \qquad\qquad\quad + \moi \big[A_i \w \lpo(\ddp A_j\w A_k) - \lpo(A_i \w \ddp A_j)\w A_k\big]\nn\\
& \qquad\qquad\quad +\moij \big[A_i \w \lpo(A_j\w \ddp A_k) - \lpo(A_i \w A_j)\w \ddp A_k\big] \Big\}\nn ~.
\end{align}
This is precisely in the form \eqref{Leialg3} required for an $\Ain$-algebra with 
\begin{align}\label{m3def}
m_3(A_i, A_j, A_k) = \pprs \big[A_i \w \lpo(A_j\w A_k) - \lpo(A_i \w A_j)\w A_k \big]~.
\end{align}
Notice that if $i+j+k = n+p +1$, then the form inside the bracket of \eqref{m3def} has degree $i+j+k -2(p+1)= n-p-1 $.  But $\pprs$ acts on $\CA^{n-p-1}$ as the zero map.  Hence we find that
\begin{align}\label{m3range}
m_3(A_i, A_j, A_k) = 
\begin{cases} 
0  & \text{if } i+ j+k < n + p + 2, \\
\pprs \Big[A_i \w \lpo(A_j\w A_k)  &\\
\qquad\qquad - \lpo(A_i \w A_j)\w A_k\Big] &  \text{if } i+j+k \geq n+p +2
\end{cases}
\end{align}
for any $A_i, A_j, A_k \in \FO{p}{*}$.  

Note that with a $m_3$ map satisfying \eqref{Leialg3}, we have shown that the product is associative on $F^pH^*$.  Together with the differential satisfying the Leibniz rule, we conclude that $(F^pH^*, +, \times)$ is a ring.

\subsection{Triviality of higher order maps}

With $m_3$ found to be non-zero, we can use it to determine $m_k$ for $k>3$.  The condition \eqref{Leialg} for $k=4$ reads
\begin{align}\label{Leialg4}\begin{split}
m_2 (m_3 \otimes \bo)& + m_2 ( \bo \otimes m_3) - m_3 (m_2 \otimes \bo^{\otimes 2}) + m_3 ( \bo \otimes m_2 \otimes \bo) - m_3 ( \bo^{\otimes 2} \otimes m_2) \\
& = m_1 m_4 - m_4\left( m_1 \otimes \bo^{\otimes 3} + \bo \otimes m_1 \otimes \bo^{\otimes 2} + \bo^{\otimes 2} \otimes m_1 \otimes \bo + \bo^{\otimes 3} \otimes m_1\right) ~.
\end{split}\end{align}
We will show that the left-hand-side of \eqref{Leialg4} consisting of $m_2$ and $m_3$ is identically zero.  Therefore, $m_4$ can be taken to be zero.

From \eqref{m3range}, we know that $m_3$ is only non-trivial when all three elements it acts on has degree $\leq n+p$ and  the sum of the degrees is greater than $n+p+1$.   Thus we only have to consider forms $A_i, A_j, A_k, A_l \in \FO{p}{*}$ and we can write the left-hand-side of \eqref{Leialg4} as
\begin{align}\label{m4exp}
\text{LHS} &= m_3(A_i, A_j, A_k) \times A_l  + (-1)^i A_i \times m_3(A_j, A_k, A_l)\nn \\
&\quad - m_3(A_i \times A_j, A_k, A_l) + m_3(A_i, A_j \times A_k, A_l) - m_3(A_i, A_j, A_k \times A_l) ~.
\end{align}
Let us consider each term on the right-hand-side of \eqref{m4exp}.

For the first term, we find 
\begin{align}\label{m41}
\mm(\aa,\ab, \ac) \pp \ad &= \rstar \left\{\rstar \pprs \left(\aa\w \lpo(\ab \w \ac) - \lpo(\aa \w \ab) \w \ac\right) \w \ad \right\}\nn\\
& = \rstar \left( \aa \w \lpo(\ab \w \ac) \w \ad - \lpo(\aa \w \ab) \w \ac \w \ad \right) \nn\\
&  \quad -\rstar \Big[\lpo\opo \big(\aa \w \lpo(\ab \w \ac) \big) \w \ad \\
&\qquad\qquad - \lpo\opo \big(\lpo(\aa \w \ab)\w \ac\big)\w \ad \Big] ~. \nn
\end{align}
where we have used \eqref{relation2}.  Similarly, the second term gives
\begin{align}\label{m42}
\moi \aa \w \mm(\ab, \ac, \ad)&= \rstar\left(\aa \w \ab \w \lpo(\ac \w \ad) - \aa \w \lpo(\ab\w \ac) \w \ad \right) \nn\\
&\qquad +\rstar \Big[- \aa \w \lpo\opo(\ab \w \lpo(\ac \w \ad)) \\
& \qquad \qquad ~~+ \aa \w \lpo\opo\lpo(\ab\w \ac) \w \ad \Big] ~.\nn
\end{align}
For the third term, the only non-zero contribution comes from $i+j\leq n+p$.  Therefore, we have
\begin{align}\label{m43}
-\mm(\aa\pp\ab, \ac, \ad) & = -\mm\left(\pip(\aa \w \ab), \ac, \ad\right) \nn\\
& = \rstar\Big[-\aa \w \ab \w \lpo(\ac \w \ad) + \lpo(\aa\w\ab\w\ac) \w \ad \nn\\
& \qquad\quad +\opo \lpo(\aa\w\ab) \w \lpo(\ac\w\ad) \\
& \qquad\quad - \lpo\big(\opo\lpo(\aa\w\ab)\w \ac\big) \w \ad \Big] ~. \nn
\end{align}
By the same token, the fourth and fifth terms are
\begin{align}\label{m44}
\mm(\aa, \ab\pp\ac, \ad)& = \rstar\Big[\aa \w \lpo(\ab\w\ac\w\ad) - \lpo(\aa\w\ab\w\ac)\w\ad \nn\\
& \qquad\quad - \aa \w \lpo\big(\opo\lpo(\ab\w\ac)\w\ad\big) \\
& \qquad\quad + \lpo\big(\aa \w \opo\lpo(\ab\w\ac)\big)\w\ad\Big] ~, \nn
\end{align}

\begin{align}\label{m45}
-\mm(\aa, \ab, \ac\pp\ad)&=\rstar\Big[-\aa\w \lpo(\ab\w\ac\w\ad)+ \lpo(\aa\w\ab) \w \ac\w\ad\nn\\
&\qquad\quad + \aa \w \lpo\big(\ab\w\opo\lpo(\ac\w\ad)\big)\\
&\qquad\quad - \lpo(\aa\w\ab)\w \opo\lpo(\ac\w\ad) \Big] ~.\nn
\end{align}
Summing all the terms of \eqref{m41}--\eqref{m45} then gives zero.  This shows that $m_4$ can be chosen to be zero.  It is also straightforward to see that the higher maps $m_k$ for $k>4$ can also be chosen to be zero.  Thus, we have shown that there is an $\Ain$-algebra for the $p$-filtered forms.

\section{Ring structure of the symplectic four-manifold from fibered three-manifold}\label{example_product}

The purpose of this section is to present a pair of symplectic four-manifolds with the following properties:
\begin{itemize}
\item their de Rham cohomology rings are isomorphic and their primitive cohomologies have the same dimensions;
\item their primitive cohomology has different ring structure; in particular, the products in the component $ \PHPR{2}\otimes\PHPR{2} \to \PHDM{1} $ are different.
\end{itemize}

Both four-manifolds are topologically $S^1 \times Y_\tau$ where $Y_\tau=\Sigma \times_\tau S^1$ is a mapping torus of a closed surface $\Sigma$ identified with a monodromy $\tau$.  The construction of such symplectic four-manifolds were described in Section \ref{example_mapping_torus} and here we shall use the same notations and conventions as there. 

According to Proposition \ref{prop_mapping_torus1}, the dimensions of the de Rham cohomologies as well as the primitive cohomologies of such symplectic four-manifolds are determined by two natural numbers.  Consider the action of $\tau^*$ on $\HD{1}(\Sigma)$:
\begin{enumerate}
\item the first number is the dimension of the $\tau^*$-invariant subspace, which we denote by $q+p$;
\item the intersection pairing is not always non-degenerate on the $\tau^*$-invariant subspace; the second number is the dimension of this kernel, which is $q-p$.
\end{enumerate}
But in order to calculate the product structure of the primitive cohomologies of the four-manifolds, we need to understand well their basis elements.  So to begin, we shall first give a systematic construction of the basis elements of the de Rham cohomology of the three-manifold, $Y_\tau$ which then will lead us to the basis elements of the primitive cohomologies of $X=S^1 \times Y_\tau$ and their products.

\subsection{Representatives of de Rham cohomologies of the fibered three-manifold}

By symplectic linear algebra, we may choose a basis for the $\tau^*$-invariant subspace of $\HD{1}(\Sigma)$:
\begin{align*}
\big\{ [\alpha_1], [\alpha_2], \ldots, [\alpha_p], [\alpha_{p+1}], \ldots, [\alpha_q],\, [\beta_1], [\beta_2], \ldots, [\beta_p] \big\}
\end{align*}
where $p \leq q$ and such that $[\alpha_j]\cdot[\beta_j]=1=-[\beta_j]\cdot[\alpha_j]$ for $1\leq j\leq p$ and other pairings vanish.  Here, the pairing $[\gamma]\cdot[\gamma']$ is defined to be $\int_\Sigma(\gamma\wedge\gamma')/\int_\Sigma\omega_\Sigma$ for any $[\gamma],[\gamma']\in\HD{1}(\Sigma)$.  Again, by symplectic linear algebra, we can extend this basis to a symplectic basis for $\HD{1}(\Sigma)$:
\begin{align*}
&\big\{ [\alpha_1], \ldots, [\alpha_q],\, [\beta_1], \ldots, [\beta_p] \big\}  \,\cup\,  \big\{ [\alpha_{q+1}], , \ldots, [\alpha_g],\, [\beta_{p+1}], \ldots, [\beta_{q}], [\beta_{q+1}], \ldots, [\beta_g] \big\}
\end{align*}
where $g$ is the genus of the surface $\Sigma$.  A linear algebra argument shows that the image of $\tau^*-\bo$ is always perpendicular to $\{[\alpha_k]\}_{k=1}^p$ and $\{[\beta_k]\}_{k=1}^q$ with respect to the above symplectic pairing.  It follows that the image of $\tau^*-\bo$ is spanned by
\begin{align*}
\big\{ [\alpha_{p+1}], \ldots, [\alpha_q], [\alpha_{q+1}], \ldots, [\alpha_g],\, [\beta_{q+1}], \ldots, [\beta_g] \big\} ~.
\end{align*}
Thus, for example, there exists a $[\zeta_k]$ such that
\begin{align*}
\tau^*[\zeta_k] - [\zeta_k] = [\alpha_k]  \quad\text{ for any } p<k\leq g ~.
\end{align*}

Now the de Rham cohomology of $Y_\tau$ is determined by the Wang exact sequence in (\ref{wang_exact}) which involves the map $\tau^* -\bo$ on $\HD{1}(\Sigma)$.   Hence, we can explicitly construct the basis elements of $\HD{1}(Y_\tau)$ and $\HD{2}(Y_\tau)$ in terms of $\ker(\tau^*-\bo)$ and $\cok(\tau^*-\bo)$ as follows.

To start, consider the Jordan form of $\tau^*$.  When $1\leq k\leq q$, the element $[\alpha_k]$ must be the upper-left-most element of some Jordan block of eigenvalue $1$.  The size of the Jordan block is $1$ if and only if $1\leq k\leq p$.  That is to say, the interesting case is when $k\in\{p+1,\ldots,q\}$.  Let $[\gamma_{k,0}], [\gamma_{k,1}], \ldots, [\gamma_{k,\ell_k}]$ be the canonical basis for the block, where $\gamma_{k,0} = \alpha_k$, $\gamma_{k,1} = \zeta_k$ and with $\ell_{k}+1$ being the size of the Jordan block.  The discussion in the immediate following will be within a single Jordan block, and so for notational simplicity, we shall suppress the first subscript of $\gamma$ and write $\gamma_j$ for $\gamma_{k,j}$.  With this understood, there exist functions $\{g_j\}_{j=0}^\ell$ on $\Sigma$ such that
\begin{align*}
\tau^*(\gamma_0) &= \gamma_0 + d g_0~, \\
\tau^*(\gamma_1) &= \gamma_1 + \gamma_0 + d g_1 ~, \\
&~\,\vdots \\
\tau^*(\gamma_{\ell}) &= \gamma_{\ell} + \gamma_{\ell-1} + d g_{\ell} ~.
\end{align*}
We remark that the terminal element $[\gamma_\ell]$ does not belong to the image of $\tau^*-\bo$.  From the canonical basis, we can construct the following globally-defined differential forms on $Y_\tau$:
\begin{align*}
\gamma_0 &\leadsto \tilde\gamma_0=\gamma_0 + d(\chi\, g_0) ~, \\
\gamma_1 &\leadsto \tilde\gamma_1=\gamma_1 + \phi\,\gamma_0 +  d(\chi\, g_1) + (\phi-1)d(\chi\, g_0) ~, \\
\gamma_2 &\leadsto \tilde\gamma_2=\gamma_2 + \phi\,\gamma_1 + \frac{\phi^2-\phi}{2}\,\gamma_0 + d(\chi\, g_2) + (\phi-1)d(\chi\, g_1) + \frac{\phi^2-3\phi+2}{2}d(\chi\,g_0) ~,\\
&~\,\vdots
\end{align*}
where $\chi(\phi)$ is an interpolating function defined on the interval $\phi\in(-0.1,1.1)$ and is equal to $0$ on $(-0.1,0.1)$ and $1$ on $(0.9,1.1)$.  Here, we are covering the interval $[0,1]$ using three charts: $(-0.1, 0.1), (0,1)$, and $(0.9, 1.1)$.  The above $\tilde\gamma$'s are invariant under the identification $(x,\phi)\to(\tau(x),\phi-1)$, and thus are well-defined one-forms on $Y_\tau$.  In general, for $j\in\{0, \ldots, \ell\}$ where $\ell+1$ is the size of the Jordan block, we have 
\begin{align*}
\gamma_j &\leadsto \tilde\gamma_j = \sum_{i=0}^j \big(f_i(\phi)\gamma_{j-i} + f_i(\phi-1)\,d(\chi(\phi)\, g_{j-i})\big)  ~,
\end{align*}
where
$$f_i(\phi) = \dfrac{1}{i!}\prod^{i-1}_{m=0} \, (\phi-m) = \dfrac{1}{i!} \phi(\phi-1)\ldots(\phi-i+1) $$
and $f_0(\phi)\equiv1$.  The functions $f_i(\phi)$ have the following properties:
\begin{itemize}
\item $f_i(\phi)$ is a polynomial in $\phi$ of degree $i$;
\item $f_i(0) = 0$ for any $i>0$;  namely, $f_i(\phi)$ does not have any constant terms;
\item $f_i(\phi) - f_i(\phi-1) = f_{i-1}(\phi-1)$;
\item $f_i'(\phi) = \sum_{m=1}^i (-1)^{m+1}f_{i-m}(\phi)/m$.
\end{itemize}
Taking the exterior derivative of $\tilde\gamma_j$, we find
\begin{align}\label{eqn_d_gamma}
d\tilde{\gamma}_{j} &= d\phi\wedge\big(\sum_{i=1}^j\frac{(-1)^{i+1}}{i}\tilde{\gamma}_{j-i}\big)
\end{align}
for any $j\in\{1,\ldots,\ell\}$.  Notice that the sum above starts from $i=1$ and hence there are no terms of the form $d\phi \wedge \tilde{\gamma}_{\ell}$ on the right hand side of \eqref{eqn_d_gamma}.  Therefore, each $\alpha_k$ for $k\in\{1,\ldots,q\}$ leads to an element in $\HD{2}(Y_\tau)$ defined by
\begin{align*}
d\phi\wedge \alpha_k^\sharp &\qquad\text{ where }\alpha_k^\sharp = \tilde\gamma_{k,\ell_k} \,, \\
&\qquad\text{ and }\ell_k+1\text{ is the size of the associated Jordan block} .
\end{align*}
Notice that for $1\leq k\leq p$, $\ell_k =0$, and therfore, $\alpha_k^\sharp =  \tilde\gamma_{k,0} = \tilde\alpha_k$.
With Proposition \ref{prop_mapping_torus} and the Wang exact sequence of (\ref{wang_exact}), this construction gives the following basis for the de Rham cohomologies of $Y_\tau$ from the canonical basis of $\tau^*$:
\begin{align*}
\HD{1}(Y_\tau) &=\BR^{q+p+1} = \spn\{ d\phi,\tilde{\alpha}_1,\ldots,\tilde{\alpha}_q,\tilde{\beta}_1,\ldots,\tilde{\beta}_p \} ~, \\
\HD{2}(Y_\tau) &=\BR^{q+p+1} = \spn\{ \omega_\Sigma, d\phi\wedge\tilde{\alpha}_1, \ldots, d\phi\wedge\tilde{\alpha}_p, d\phi\wedge\tilde{\beta}_1, \ldots, d\phi\wedge\tilde{\beta}_p, \\
&\qquad\qquad\qquad\qquad d\phi\wedge{\alpha}_{p+1}^\sharp, \ldots, d\phi\wedge{\alpha}_{q}^\sharp \} ~.
\end{align*}

\subsection{Representatives of the primitive cohomologies of the symplectic four-manifold}
We now consider the four-manifold $X = S^1\times Y_\tau$ with the symplectic form $\omega = dt\wedge d\phi + \omega_\Sigma$.  Its de Rham cohomologies are
\begin{align*}
\HD{1}(X) &= \BR^{q+p+2} =  \spn\{ dt, \HD{1}(Y_\tau) \} ~, \\
\HD{2}(X) &= \BR^{2q+2p+2} = \spn\{ dt\wedge\HD{1}(Y_\tau), \HD{2}(Y_\tau) \} ~, \\
\HD{3}(X) &= \BR^{q+p+2} = \spn\{ d\phi\wedge\omega_\Sigma, dt\wedge\HD{2}(Y_\tau) \} ~.
\end{align*}
Clearly, the Euler characteristic of $X$ is zero.  It is also straightforward to check that the signature of $X$ is also zero.

We will use the exact sequence of Theorem \ref{resolution} to construct a basis for the primitive cohomologies $\PHPR{2}(X)$ and $\PHPL{2}(X)$.  As noted in Theorem \ref{prop_mapping_torus1}, they are both isomorphic to $\BR^{3q+p+1}$.

By Corollary \ref{cor_dimension4}, $\PHPR{2}(X)$ has a component isomorphic to the $\cok(L:\HD{0}(X)\to\HD{2}(X))$, and $\PHPL{2}(X)$ has a component isomorphic to the $\ker(L:\HD{2}(X)\to\HD{4}(X))$.  It is not hard to see that both these components are  isomorphic to the following basis elements in $\HD{2}(X)$:
\begin{align*}
&dt\wedge(d\phi - \omega_\Sigma) ~,~ dt\wedge\tilde{\alpha}_1 ~,~ \ldots~,~ dt\wedge\tilde{\alpha}_q ~,~ dt\wedge\tilde{\beta}_1 ~,~ \ldots ~,~ dt\wedge\tilde{\beta}_p ~,~ \\
&d\phi\wedge\tilde{\alpha}_1 ~,~ \ldots ~,~ d\phi\wedge\tilde{\alpha}_p ~,~ d\phi\wedge\tilde{\beta}_1 ~,~ \ldots ~,~ d\phi\wedge\tilde{\beta}_p ~,~ \\
&d\phi\wedge{\alpha}_{p+1}^\sharp ~,~ \ldots ~,~ d\phi\wedge{\alpha}_{q}^\sharp ~.
\end{align*}
The elements corresponding to $dt\wedge\HD{1}(Y_\Sigma)$ may not be primitive in general and so do need to be modified.  For instance, suppose that $\tau^*\alpha = \alpha + dg$.  Then, $dt\wedge(\alpha + d(\chi\, g))=dt\wedge(\alpha + \chi dg + \chi' g \,d\phi) $ 
is not necessarily primitive.  Note that $g$ is unique up to a constant.  Thus, we may assume that
\begin{align}\label{example_normalize}
\int_\Sigma g\,\omega_\Sigma = 0 ~,~ \text{ and thus } ~ g\,\omega_\Sigma = d\mu ~.
\end{align}
for some $1$-form $\mu$ on $\Sigma$.  The standard computation on Lefschetz decomposition then shows that
\begin{align*}
dt\wedge(\alpha + d(\chi\, g)) - d(\chi'\mu) = dt\wedge(\alpha + \chi dg + \chi' g d\phi) - \chi'g\omega_\Sigma - \chi''d\phi\wedge\mu
\end{align*}
is a primitive $2$-form, and is cohomologous to 
$dt\wedge(\alpha + d(\chi\, g))$.

Now the other component in Corollary \ref{cor_dimension4} for $\PHPR{2}(X)$ is $\ker(L:\HD{1}(X)\to\HD{3}(X))$.  This kernel is spanned by $\{\tilde{\alpha}_{p+1},\ldots,\tilde{\alpha}_q\}$.  
The corresponding elements of $\PHPR{2}$ are those that when acted upon by $\dpm$ gives an element in the kernel.  
To explicitly construct them, we note that for any $k\in\{p+1,\ldots,q\}$, 
\begin{align}\label{dmalp}
\omega\wedge\tilde{\alpha}_k &= dt\wedge d\phi\wedge\tilde{\gamma}_{k,0} + \chi' g_{k,0}d\phi\wedge\omega_\Sigma \notag\\
&= -d\big( dt\wedge\tilde{\gamma}_{k,1} + \chi'd\phi\wedge\mu_{k,0} \big) \\
&= -d\big( dt\wedge\tilde{\gamma}_{k,1} + \chi'd\phi\wedge\mu_{k,0} - d(\chi'\mu_{k,1} + \chi'(\phi-1)\mu_{k,0}) \big) \notag
\end{align}
where $\mu_{k,0}$ and $\mu_{k,1}$ are defined by (\ref{example_normalize}).  The expression inside the exterior derivative in line 2 above does not generally represent a primitive $2$-form. Therefore, we added the exterior derivative of $\chi'\mu_{k,1} + \chi'(\phi-1)\mu_{k,0}$ in line 3 to ensure primitivity.  We thus obtain the following elements in $\PHPR{2}$ for $k\in\{p+1,\ldots,q\}$,
\begin{align*}
dt\wedge\tilde{\gamma}_{k,1} + \chi'd\phi\wedge\mu_{k,0} - d(\chi'\mu_{k,1} + \chi'(\phi-1)\mu_{k,0}) ~.
\end{align*}
The above computation \eqref{dmalp} shows that the $\dpm$ action on the above expression is equal to $\tilde{\gamma}_{k,0} = \tilde{\alpha}_k$.

For $\PHPL{2}(X)$, the other component in Corollary \ref{cor_dimension4} is $\cok(L:\HD{1}(X)\to\HD{3}(X))$.  This is spanned by $\{(dt\wedge d\phi + \omega_\Sigma)\wedge\alpha_{k}^\sharp\}_{k=p+1}^q$.  These basis elements are cohomologous to $dt\wedge d\phi\wedge\alpha_k^\sharp$ provided that all $g_i$ has zero integration against $\omega_\Sigma$.  The corresponding elements in $\PHPL{2}(X)$ are $\dpp\alpha^\sharp_k$ which by (\ref{eqn_d_gamma}) are explicitly given by
\begin{align*}
d\phi\wedge\big(\sum_{j=1}^{\ell_k}\frac{(-1)^{j+1}}{j}\tilde{\gamma}_{k,\ell_k-j}\big) ~,
\end{align*}
for $k\in\{p+1,\ldots,q\}$.

\subsection{Two examples and their product structures}
We shall present below two explicit constructions.  The two four-manifolds will have the same de Rham cohomology ring, as well as identical primitive cohomologies.  However, their primitive product structures will be shown to be different.  In particular, the image of the following pairing has different dimensions in these two constructions:
\begin{align} \label{ppbb}\begin{array}{ccl}
\PHPR{2}(X) \otimes \PHPR{2}(X)  &\to  &\PHDM{1}(X)\\
(B_2, B_2')  &\mapsto  & \rstar \left(- d\,L^{-1}(B_2\w B'_2) + \dpm B_2 \w B'_2 + B_2 \w \dpm B'_2 \right)  ~.
\end{array} \end{align}
This is a symmetric bilinear operator.  Note that the first term $- \rstar\, d \, L^{-1}(B_2 \w B'_2) = - d \,L^{-2} (B_2 \w B'_2)$ is not only $\dpm$-closed but also $\dpp$-exact.  We remark that on a compact, symplectic four-manifold, $\dpp B_0$ is always $\dpm$-exact for any function $B_0$.  (See Proposition 3.16 of \cite{TY2}).   Hence, the first term on the right hand side of \eqref{ppbb} does not have any contribution here.

\subsubsection{Kodaira--Thurston nilmanifold}
The first example is the Kodaira--Thurston nilmanifold, which we denote by $X_1$.  The primitive cohomologies are computed in \cite{TY1}.  It is constructed from a torus with a Dehn twist.  To be more precise, let $T^2 = \BR^2/\BZ^2$, and let $(a,b)$ be the coordinate for $\BR^2$.  The monodromy map $\tau$ is induced by $(a,b)\mapsto(a,a+b)$.  It follows that $\HD{1}(T^2)$ is spanned by $da$ and $db$, and $\tau^*(db) = da+db$, $\tau^*(da) = da$.  We take the area form on $T^2$ to be $da\wedge db$, and take the symplectic form on $X = S^1\times Y_\tau$ to be $\omega = dt\wedge d\phi + da\wedge db$.

The basis for the primitive cohomologies can be constructed from the recipe explained in the previous subsection.  Note that $\tilde{\gamma}_0 = d a$ and $\tilde{\gamma}_1 = db + \phi\,da$ are well-defined on $Y_\tau$.  With this understood, we have
\begin{align*}
\PHPR{2}(X_1) &= \{ dt\wedge d\phi - da\wedge db, dt\wedge\tilde{\gamma}_0, d\phi\wedge\tilde{\gamma}_1, dt\wedge\tilde{\gamma}_1\} ~, \\
\PHDM{1}(X_1) &= \{ dt, d\phi, \tilde{\gamma}_1 \} ~.
\end{align*}
The only generator of $\PHPR{2}(X_1)$ which is not $\dpm$-closed is $dt\wedge\tilde{\gamma}_1$, and $\dpm (dt\wedge\tilde{\gamma}_1) = -\tilde{\gamma}_0$.  It follows that
\begin{align}
(d\phi\wedge\tilde{\gamma}_1) \times (dt\wedge\tilde{\gamma}_1) = d\phi   \qquad\text{and}\qquad   (dt\wedge\tilde{\gamma}_1) \times (dt\wedge\tilde{\gamma}_1) = 2dt
\end{align}
are the only non-trivial pairings between the generators of $\PHPR{2}(X_1)$.

\begin{rmk}
Here is the correspondence between the convention here and that of \cite{TY1}:
\begin{align*}
e_1 &= dt ~,  &e_2 &= d\phi ~,  &e_3 &= \tilde{\gamma}_0 ~,  &e_4 &= \tilde{\gamma}_1 ~.
\end{align*}
\end{rmk}

\subsubsection{An example involving a genus two surface}
Let $\Sigma$ be a genus two surface.  Fix a symplectic basis for its $\HD{1}(\Sigma)$: $\{\alpha_1,\alpha_2,\beta_1,\beta_2\}$.  Moreover, we may assume that the basis is an integral basis for the singular cohomology of $\Sigma$.  Let $\omega_\Sigma$ be an area form of $\Sigma$.  Normalize it by
\begin{align*}
[\omega_\Sigma] = [\alpha_1\wedge \beta_1] = [\alpha_2\wedge \beta_2] ~.
\end{align*}
Let $\tau$ be a monodromy of $\Sigma$ whose action on $\HD{1}(\Sigma)$ are
\begin{align*}
\tau^* &= \left(\begin{array}{cccc}
1 & n & \ell & 1 \\
0 & 1 & 1 & 0 \\
0 & 0 & 1 & 0 \\
0 & m & m+n & 1
\end{array}\right)
\end{align*}
with respect to $\{a_1,b_1,a_2,b_2\}$, and $\ell,m,n$ are integers with $m+n\neq0$.  A direct computation shows that $\tau^*$ preserves the intersection pairing.  It then follows from the theory of mapping class groups (see for example \cite{FM}) that $\tau^*$ does arise from a monodromy.  Note that
$$S^{-1} \tau^* S = J~,$$ 
where 
\begin{align*}
S = \left(\begin{array}{cccc}
1 & 0 & 0 & 1 \\
0 & 0 & \frac{1}{m+n} & \frac{-1}{m+n} \\
0 & 0 & 0  & \frac{1}{m+n} \\
0 & 1 & \frac{-n}{m+n} & \frac{n-\ell}{m+n}
\end{array}\right)
\qquad\text{and}\qquad
J = \left(\begin{array}{cccc}
1 & 1 & 0 & 0 \\
0 & 1 & 1 & 0 \\
0 & 0 & 1 & 1 \\
0 & 0 & 0 & 1
\end{array}\right) ~.
\end{align*}
Hence, the basis for the Jordan form is as follows: 
\begin{align*}
\gamma_0 &= \alpha_1 ~,   &\gamma_1 &= \beta_2 ~, \\
\gamma_2 &= \frac{1}{m+n} \, \alpha_2 + \frac{-n}{m+n}\, \beta_2~,   &\gamma_3 &= \frac{1}{m+n}\, \beta_1 - \frac{1}{m+n}\, \alpha_2 + \frac{n-\ell}{m+n}\, \beta_2 ~.
\end{align*}
For concreteness, we will set $m=1$ and $n=\ell=0$.  The canonical basis then becomes
\begin{align*}
\gamma_0 &= \alpha_1 ~,   &\gamma_1 &= \beta_2 ~,   &\gamma_2 &= \alpha_2 ~,   &\gamma_3 &= \beta_1-a_2 ~.
\end{align*}
They give the following well-defined differential forms on $Y_\tau$:
\begin{align*}
\tilde{\gamma}_0 &= \alpha_1 +  d(\chi\,g_0) ~, \\
\tilde{\gamma}_1 &= \beta_2 + \phi\, \alpha_1 + d(\chi\,g_1) + (\phi-1)d(\chi\,g_0) ~, \\
\tilde{\gamma}_2 &= \alpha_2 + \phi\, \beta_2 + \frac{\phi^2-\phi}{2}\, \alpha_1 + d(\chi\,g_2) + (\phi-1)d(\chi\,g_1) + \frac{\phi^2-3\phi+2}{2}d(\chi\,g_0)~, \\
\tilde{\gamma}_3 &= (\beta_1-\alpha_2) + \phi\, \alpha_2 + \frac{\phi^2-\phi}{2}\, \beta_2 + \frac{\phi^3-3\phi^2+2\phi}{6}\, \alpha_1 \\
&\quad + d(\chi\,g_3) + (\phi-1)d(\chi\,g_2) + \frac{\phi^2-3\phi+2}{2}d(\chi\,g_1) + \frac{\phi^3-6\phi^2+11\phi-6}{6}d(\chi\,g_0) ~.
\end{align*}
We shall assume that $g_i\, \omega_\Sigma = d\mu_i$ for $i\in\{0,1,2,3\}$.
According to the discussion in the previous subsection,
\begin{align*}
\omega &= dt\wedge d\phi + \omega_\Sigma ~, \\
\PHPR{2}(X_2) &= \{ dt\wedge d\phi-\omega_\Sigma ,\, dt\wedge\tilde{\gamma}_0 ,\, d\phi\wedge\tilde{\gamma}_3,\\
&\qquad dt\wedge\tilde{\gamma}_1 + \chi'd\phi\wedge\mu_0 - d(\chi'\mu_1 + \chi'(\phi-1)\mu_0) \} ~, \\
\PHDM{1}(X_2) &= \{ dt, d\phi, \tilde{\gamma}_3 \} ~.
\end{align*}
Since $\PHDM{1}(X_2) \cong \HD{3}(X_2)$, its element can be captured by integrating against $\HD{1}(X_2) = \{d\phi, dt, \tilde{\gamma}_0\}$.  There is only one generator of $\PHPR{2}(X_2)$ that is not $d$-closed:
\begin{align*}
\dpm\big( dt\wedge\tilde{\gamma}_1 + \chi'd\phi\wedge\mu_0 - d(\chi'\mu_1 + \chi'(\phi-1)\mu_0) \big) &= -\tilde{\gamma}_0 ~.
\end{align*}

We now consider the pairings between the generators of $\PHPR{2}(X_2)$.  If one of them is $d$-closed, the only non-trivial pairing is
\begin{align*}
(d\phi\wedge\tilde{\gamma}_3) \times \big( dt\wedge\tilde{\gamma}_1 + \chi'd\phi\wedge\mu_0 - d(\chi'\mu_1 + \chi'(\phi-1)\mu_0) \big) &\cong d\phi\wedge\tilde{\gamma}_0\wedge\tilde{\gamma}_3 ~.
\end{align*}
The latter expression has nonzero integration against $dt$.  That is to say, the element in $\PHDM{1}(X_2)$ is proportional to $d\phi$.

It remains to examine the square of $dt\wedge\tilde{\gamma}_1 + \chi'd\phi\wedge\mu_0 - d(\chi'\mu_1 + \chi'(\phi-1)\mu_0)$.  As an element in $\HD{3}(X_2)$, it is
\begin{align*}
-2\big( dt\wedge\tilde{\gamma}_1 + \chi'd\phi\wedge\mu_0 - d(\chi'\mu_1 + \chi'(\phi-1)\mu_0) \big)\wedge \tilde{\gamma}_0 ~.
\end{align*}
The expression has zero integration against $\tilde{\gamma}_0$.  We compute its integration against $d\phi$.
\begin{align*}
2\int_{X_2} d\phi\wedge dt\wedge \tilde{\gamma}_0\wedge\tilde{\gamma}_1 ~,
\end{align*}
and it is not hard to see that $\int_{\Sigma_\phi}\tilde{\gamma}_0\wedge\tilde{\gamma}_1 = 0$ on each fiber $\Sigma_\phi$ of the fibration $Y_\tau\to S^1$.  This implies the square of $dt\wedge\tilde{\gamma}_1 + \chi'd\phi\wedge\mu_0 - d(\chi'\mu_1 + \chi'(\phi-1)\mu_0)$ can not be proportional to $dt$, though it can be proportional to $d\phi$.  We therefore can conclude that the image of $\PHPR{2}(X_2)\otimes\PHPR{2}(X_2)\to\PHDM{1}(X_2)$ is spanned only by $d\phi$.  We have thus shown that the images of $\PHPR{2}\otimes\PHPR{2}\to\PHDM{1}$ of the above two examples are of different dimensions.

\

\appendix
\section{Compatibility of filtered product with topological products}\label{apx_product}

We here provide the proof of Theorem \ref{compthrm} demonstrating the compatibility of the filtered product $\times$ as defined in Definition \ref{defprod} with the wedge and Massey product.  We begin first with a lemma which will be useful in the proof.

\begin{lem}
For any $A_k\in\CA^k$,
\begin{align}\label{wm05}
\rstar d L^{-(p+1)}A_k - \rstar L^{-(p+1)}dA_k &= \pil{p}\rstar d L^{-(p+1)}A_k - \rstar L^{-(p+1)}d(\pil{p}A_k) ~.
\end{align}
Note that the second term of the right hand side vanishes when $k>n+p$, and the first term vanishes when $k\leq n+p$.
\end{lem}
\begin{proof}
Case (i): when $k\leq n+p$.  We invoke \eqref{relation1} to write $A_k$ as
\begin{align*}
A_k &= \pil{p}A_k + \om^{p+1}\wedge(L^{-(p+1)}A_k) ~.
\end{align*}
After taking $L^{-(p+1)}\circ d$, it becomes
\begin{align*}
L^{-(p+1)}d A_k &= L^{-(p+1)}d (\pil{p} A_k) + L^{-(p+1)}\big[\om^{p+1}\w d(L^{-(p+1)}A_k) \big] ~.
\end{align*}
Since $k\leq n+p$, the last term is equal to $d(L^{-(p+1)}A_k)$.  This finishes the proof for this case.

Case (ii): when $k> n+p$.  We write $A_k$ as
\begin{align*}
A_k = \om^{k-n}\wedge B_{2n-k} + \om^{k-n+1}\wedge A'_{2n-k-2}
\end{align*}
where $B_{2n-k}\in\CB^{n-k}$ and $A'_{2n-k-2}\in\CA^{2n-k-2}$.  A straightforward computation shows that
\begin{align*}
d L^{-(p+1)}A_k &= \om^{k-n-p-1}\wedge (\dpp B_{2n-k}) + \om^{k-n-p}\wedge(\dpm B_{2n-k} + dA'_{2n-k-2}) \quad\text{in }\CA^{k-2p-1} ~, \\
\Rightarrow \quad  \rstar d L^{-(p+1)}A_k &= \om^{p}\wedge (\dpp B_{2n-k}) + \om^{p+1}\wedge(\dpm B_{2n-k} + dA'_{2n-k-2}) \quad\text{in }\CA^{2n-k+2p+1} ~.
\end{align*}
Hence, $\om^{p}\wedge (\dpp B_{2n-k})$ is $\pil{p}\rstar d L^{-(p+1)}A_k$.  Meanwhile, it is not hard to see that $\om^{k-n-p}\wedge(\dpm B_{2n-k} + dA'_{2n-k-2})$ is $L^{-(p+1)}dA_k$.  This finishes the proof of the lemma.
\end{proof}

We now give the proof of Theorem \ref{compthrm}.  We shall consider the four different cases separately.

\

\noindent{\it{(1) $\FHP{p}{j}\times\FHP{p}{k}\to\FHP{p}{j+k}$, $j+k\leq n+p$}}

\begin{lem}\label{lem_wm01}
For $j\leq n+p$, $k\leq n+p$ and $j+k\leq n+p$, the product $\FHP{p}{j}\times\FHP{p}{k}\to\FHP{p}{j+k}$ induced by (\ref{prod21}) is compatible with the topological products.
\end{lem}

\begin{proof}
(Wedge product)~  Given $[\xi_j]\in\HD{j}$ and $[\xi_k]\in\HD{k}$, it is not hard to see that
\begin{align*}
\pil{p}(\xi_j\w\xi_k) &= \pil{p}\big( (\pil{p}\xi_j)\w(\pil{p}\xi_k) \big) = (\pil{p}\xi_j)\times(\pil{p}\xi_k) ~.
\end{align*}

\noindent(Massey product)~  Consider two elements $[A_j]\in\FHP{p}{j}$ and $[A_k]\in\FHP{p}{k}$.  Since $A_j$ and $A_k$ are $\ddp$-closed, $g(A_j) = L^{-(p+1)}dA_j$ and $g(A_k)= L^{-(p+1)}dA_k$.  Moreover,
\begin{align}\label{wm04}
dA_j = \om^{p+1}\w g(A_j)  \qquad\text{and}\qquad
d A_k = \om^{p+1}\w g(A_k) ~.
\end{align}
With (\ref{wm04}), the Massey product (\ref{wm02}) is
\begin{align}\label{wm11}
\Massey{g(A_j)}{g(A_k)} &= (L^{-(p+1)}dA_j)\w A_k + (-1)^jA_j\w(L^{-(p+1)}dA_k) ~.
\end{align}

We now calculate $g(A_j\times A_k)$.  According to Theorem \ref{LeiRule}, $A_j\times A_k$ is $\ddp$-closed.  It follows that $g(A_j\times A_k) = L^{-(p+1)}d(A_j\times A_k)$.  With (\ref{wm04}),
\begin{align*}
d(A_j\times A_k)&= d(A_j\w A_k - \omega^{p+1}\w L^{-(p+1)}(A_j\times A_k)) \\
&= \omega^{p+1}\w\big( (L^{-(p+1)}dA_j)\w A_k + (-1)^jA_j\w(L^{-(p+1)}dA_k)  + d L^{-(p+1)}(A_j\w A_k) \big) ~.
\end{align*}
Since $j+k\leq n+p$, $L^{p+1}$ is injective on $\CA^{j+k-2p-1}$.  Thus,
\begin{align*}
g(A_j\times A_k) &= \Massey{g(A_j)}{g(A_k)} + d(L^{-(p+1)}(A_j\w A_k)) ~.
\end{align*}
This completes the proof of the lemma.
\end{proof}

\

\noindent{\it{(2) $\FHP{p}{j}\times\FHP{p}{k}\to\FHM{p}{2n+2p+1-j-k}$, $j+k> n+p$}}

\begin{lem}\label{lem_wm02}
For $j\leq n+p$, $k\leq n+p$ and $j+k> n+p$, the product $\FHP{p}{j}\times\FHP{p}{k}\to\FHM{p}{2n+2p+1-j-k}$ induced by (\ref{prod22}) is compatible with the topological products.
\end{lem}

\begin{proof}
(Wedge product)~  Given $[\xi_j]\in\HD{j}$ and $[\xi_k]\in\HD{k}$, let $A_j = \pil{p}\xi_j$, $\eta_{j-2p-2} = L^{-(p+1)}\xi_j$, $A_k = \pil{p}\xi_k$ and $\eta_{k-2p-2} = L^{-(p+1)}\xi_k$.  Namely,
\begin{align*}
\xi_j = A_j + \om^{p+1}\w \eta_{j-2p-2}  \qquad\text{and}\qquad  \xi_k = A_k + \om^{p+1}\w \eta_{k-2p-2} ~.
\end{align*}

It follows from $d\xi_j = 0$ and $d\xi_k = 0$ that
\begin{align*}
dA_j = -\omega^{p+1}\w d \eta_{j-2p-2}  \qquad\text{and}\qquad  dA_k = -\omega^{p+1}\w d \eta_{k-2p-2} ~.
\end{align*}
Since $j\leq n+p$, $L^{-(p+1)}dA_j = -d\eta_{j-2p-2}$ and $L^{-(p+1)}dA_k = -d\eta_{k-2p-2}$.  According to (\ref{prod22}), $f(\xi_j)\times f(\xi_k)$ is equal to
\begin{align}
& A_j\times A_k \notag \\
=&\; \pil{p}\rstar\big[-d \lpo(A_j \w A_k) - (d\eta_{j-2p-2}) \w A_k - (-1)^j A_j \w (d\eta_{k-2p-2}) \big] ~.  \label{wm06}
\end{align}

The next task is to compute $f(\xi_j\w\xi_k) = -\rstar d L^{-(p+1)}(\xi_j\w\xi_k)$.  Since $j+k>n+p$, it is also equal to $-\pil{p}\rstar d L^{-(p+1)}(\xi_j\w\xi_k)$.  We write
\begin{align*}
\xi_j\w\xi_k &= A_j\w A_k + \om^{p+1}\w\zeta_{j+k-2p-2}
\end{align*}
where $\zeta_{j+k-2p-2} = A_j\w\eta_{k-2p-2} + \eta_{j-2p-2}\w A_k + \omega^p\w\eta_{j-2p-2}\w\eta_{k-2p-2}$.  Note that
\begin{align*}
d\zeta_{j+k-2p-2} &= (d\eta_{j-2p-2})\w A_k + (-1)^jA_j\w(d\eta_{k-2p-2}) ~.
\end{align*}
By applying $\pil{p}$ on \eqref{wm05} for $\xi_j\w\xi_k$, we find that
\begin{align*}
f(\xi_j\w\xi_k) &= \pil{p}\rstar\big[ \lpo d(\xi_j\w\xi_k) - d\lpo(A_j\w A_k) - d\big(\lpo(\om^{p+1}\w\zeta_{j+k-2p-2})\big) \big] ~.
\end{align*}
The first term on the right hand side is zero.  The third term can be calculated with the help of (\ref{Lei020}):
\begin{align*}
&-\pil{p}\rstar\big[d\big(\lpo(\om^{p+1}\w\zeta_{j+k-2p-2})\big)\big] \\
=&\, -\pil{p}\rstar(d\zeta_{j+k-2p-2}) + \ddm(\pil{p} \,\rstar\, \zeta_{j+k-2p-2}) \\
=&\, -\pil{p}\rstar\big[ (d\eta_{j-2p-2})\w A_k + (-1)^jA_j\w(d\eta_{k-2p-2}) \big] + \ddm(\pil{p} \,\rstar\, \zeta_{j+k-2p-2}) ~.
\end{align*}
To sum up, $f(\xi_j\w\xi_k)$ is equal to
\begin{align}\label{wm08}\begin{split}
&- \pil{p}\rstar\big[ (d\eta_{j-2p-2})\w A_k + (-1)^jA_j\w(d\eta_{k-2p-2}) - d\lpo(A_j\w A_k) \big] \\
&\quad + \ddm(\pil{p} \,\rstar\, \zeta_{j+k-2p-2})
\end{split}\end{align}
It follows from (\ref{wm08}) and (\ref{wm06}) that (\ref{prod22}) is compatible with the wedge product.

\noindent(Massey product)~  Given $[A_j]\in\FHP{p}{j}$ and $[A_k]\in\FHP{p}{k}$, $\Massey{g(A_j)}{g(A_k)}$ is completely the same as that in the proof of Lemma \ref{lem_wm01}, and the Massey product is given by (\ref{wm11}).

We now calculate $g(A_j\times A_k)$.  According to (\ref{prod22}) and (\ref{wm03}),
\begin{align}\label{wm12}\begin{split}
g(A_j\times A_k) &= \rstar\pil{p}\rstar\Big[-d\, \lpo(A_j \w A_k)  \\
&\qquad\qquad\quad + (\lpod A_j) \w A_k + (-1)^j A_j \w (\lpod A_k) \Big] ~.
\end{split}\end{align}
The first term of the right hand side can be computed with the help of (\ref{wm05}) for $A_j\w A_k$:
\begin{align*}
&-\rstar\pil{p}\rstar d\,\lpo(A_j\w A_k) \\
=&\, \lpod(A_j\w A_k) - d\lpo(A_k\w A_k) \\
=&\, \lpo L^{p+1}\Big[ (\lpod A_j)\w A_k + (-1)^k A_j\w(\lpod A_k) \Big] - d\lpo(A_k\w A_k)
\end{align*}
where the last inequality uses (\ref{wm04}).  By plugging it into (\ref{wm12}) and applying (\ref{relation2}), we have
\begin{align*}
g(A_j\times A_k) &= (L^{-(p+1)}dA_j)\w A_k + (-1)^jA_j\w(L^{-(p+1)}dA_k) - d\lpo(A_k\w A_k) ~.
\end{align*}
This togehter with (\ref{wm11}) shows that (\ref{prod22}) is compatible with the Massey product.
\end{proof}

\

\noindent{\it{(3) $\FHP{p}{j}\times\FHM{p}{k}\to\FHM{p}{k-j}$, $j\leq k$}}

\begin{lem}\label{lem_wm03}
For $j\leq k\leq n+p$, the product $\FHP{p}{j}\times\FHM{p}{k}\to\FHM{p}{k-j}$ induced by (\ref{prod3}) is compatible with the topological products.
\end{lem}

\begin{proof}
(Wedge product)~ Suppose that $[\xi_j]\in\HD{j}$ and $[\xi_{\ell}]\in\HD{\ell}$ where $\ell = 2n+2p+1-k$. 
Since $k\leq n+p$, $\ell>n+p$,  $\xi_\ell = \omega^{p+1}\w\eta_{2n-k-1}$ for some $\eta_{2n-k-1}\in\CA^{2n-k-1}$.  Let $A_j = \pil{p}\xi_j$ and $\eta_{j-2p-2} = L^{-(p+1)}\xi_j$.  Namely,
\begin{align*}
\xi_j = A_j + \om^{p+1}\w\eta_{j-2p-2}  \qquad\text{and}\qquad  \xi_\ell = \omega^{p+1}\w\eta_{2n-k-1}
\end{align*}
Since $d\xi_k = 0$ and $d\xi_j = 0$,
\begin{align*}
d A_j = -\omega^{p+1}\w d\eta_{j-2p-2}  \qquad\text{and}\qquad  \omega^{p+1}\w d\eta_{2n-k-1} = 0  ~.
\end{align*}

According to (\ref{wm03}), $f(\xi_j) = A_j$ and $f(\xi_\ell) = -\rstar d\eta_{2n-k-1}$.  By (\ref{prod3}),
\begin{align}\label{wm09}
f(\xi_j)\times f(\xi_\ell) &= -(-1)^j\rstar(A_j\w d\eta_{2n-k-1}) ~.
\end{align}
We now compute $f(\xi_j\w\xi_\ell) = -\rstar d L^{-(p+1)}(\xi_j\w\xi_\ell)$.  Due to (\ref{wm05}),
\begin{align*}
-\rstar d\,\lpo (\xi_j\w\xi_\ell) &= - \pil{p}\rstar d\,\lpo(\xi_j\w\xi_\ell) - \rstar\lpod(\xi_j\w\xi_\ell)
\end{align*}
where the second term vanishes.  Then write $\xi_j\w\xi_\ell$ as $\omega^{p+1}\w\xi_j\w\eta_{2n-k-1}$ and apply (\ref{Lei020}) for $\zeta = \xi_j\w\eta_{2n-k-1}$:
\begin{align}
f(\xi_j\w\xi_\ell) &= -\pil{p}\rstar\big(d(\xi_j\w\eta_{2n-k-1})\big) + \ddm(\pil{p}\rstar\zeta) \notag \\
&= -\pil{p}\rstar\big[(-1)^j(A_j + \omega^{p+1}\w\eta_{j-2p-2})\w d\eta_{2n-k-1}\big] + \ddm(\pil{p}\rstar\zeta) \notag \\
&= -(-1)^j\pil{p}\rstar(A_j\w d\eta_{2n-k-1}) + \ddm(\pil{p}\rstar\zeta) ~. \label{wm10}
\end{align}
The last equality uses the fact that $\omega^{p+1}\w d\eta_{2n-k-1}$.  Due to (\ref{pistarf}), $\rstar(A_j\w d\eta_{2n-k-1})\in\FO{p}{k-j}$.  That is to say, the first $\pil{p}$-operator in (\ref{wm10}) acts as the identity map.

When $k-j = n+p$, $\pil{p}\rstar\zeta$ belongs to $\pil{p}\CA^{n+p+1} = \{0\}$.  By (\ref{wm09}) and (\ref{wm10}), we conclude that (\ref{prod3}) is compatible with the wedge product.

\noindent(Massey product)~  Given $[A_j]\in\FHP{p}{j}$ and $[\bA_k]\in\FHM{p}{k}$, let $B_{j-2p} = L^{-p}A_j$.  Since $\ddp A_j = 0$, $d A_j = L^{p+1}\w(\lpod A_{j})$.  By (\ref{wm03}),
\begin{align*}
g(A_j) &= \lpod A_{j}  &\text{and}&  &g(\bA_k) &= \rstar\bA_k ~.
\end{align*}
Due to (\ref{pistarf}), $L^{p+1}(g(\bA_k)) = 0$.  And the Massey product (\ref{wm02}) is
\begin{align*}
\Massey{g(A_j)}{g(\bA_k)} &= \Massey{\dpm B_{j-2p}}{\rstar\bA_k} = (-1)^j A_j\w(\rstar \bA_k) ~.
\end{align*}
According to (\ref{prod3}) and (\ref{wm03}),
\begin{align*}
g(A_j\times\bA_k) &= \rstar(A_j\times\bA_k) = (-1)^j A_j\w(\rstar \bA_k) ~.
\end{align*}
This completes the proof of the lemma.
\end{proof}

Since all the products are graded anti-commutative, the product $\FHM{p}{j}\times\FHP{p}{k}\to\FHM{p}{j-k}$ induced by (\ref{prod4}) is also compatible with the topological products.

\

\noindent{\it{(4) $\FHM{p}{j}\times\FHM{p}{k}\to0$}}

\begin{lem}
For $j\leq n+p$ and $k\leq n+p$, the product $\FHM{p}{j}\times\FHM{p}{k}\to0$ defined by (\ref{prod5}) is compatible with the topological products.
\end{lem}

\begin{proof}
By the simple degree counting, it is compatible with the wedge product.  It follows from (\ref{pistarf}) that $L^{p+1}g = -L^{p+1}\rstar = 0$.  Hence, the Massey product $\Massey{g(\,\cdot\,)}{g(\,\cdot\,)} = 0$ on $\FHM{p}{*}$.
\end{proof}

\

\



\vskip 1cm
\noindent
{Department of Mathematics, National Taiwan University\\
Taipei 10617, Taiwan}\\
{\it Email address:}~{\tt cjtsai@ntu.edu.tw}
\vskip .5 cm
\noindent
{Department of Mathematics, University of California, Irvine\\
Irvine, CA 92697, USA}\\
{\it Email address:}~{\tt lstseng@math.uci.edu}
\vskip .5 cm
\noindent
{Department of Mathematics, Harvard University\\
Cambridge, MA 02138, USA}\\
{\it Email address:}~{\tt yau@math.harvard.edu}

\end{document}